\documentclass[12pt]{amsart}
\usepackage{amsmath,latexsym,amsfonts,amssymb,amsthm,xcolor}
\usepackage{geometry}
\usepackage[alphabetic]{amsrefs}
\usepackage{hyperref}
\usepackage{mathrsfs}
\usepackage{tikz-cd}
\usepackage{bbm}
\usepackage{bm}
\usepackage{comment}
\usepackage{stmaryrd}
\usepackage{enumitem}

\newtheorem{prop}{Proposition}[section]
\newtheorem{thm}[prop]{Theorem}
\newtheorem{cor}[prop]{Corollary}

\newtheorem{lem}[prop]{Lemma}

\newtheorem{setup}[prop]{Setup}

\theoremstyle{definition}

\newtheorem{defn}[prop]{Definition}

\newtheorem{rem}[prop]{\it Remark}

\newtheorem*{claim*}{Claim}

\newcommand{\bR}{\mathbb{R}}

\newcommand{\bQ}{\mathbb{Q}}
\newcommand{\bZ}{\mathbb{Z}}

\newcommand{\bN}{\mathbb{N}}

\newcommand{\bG}{\mathbb{G}}
\newcommand{\bk}{\mathbbm{k}}

\newcommand{\oX}{\overline{X}}
\newcommand{\oDe}{\overline{\Delta}}
\newcommand{\oD}{\overline{D}}
\newcommand{\oC}{\overline{C}}
\newcommand{\oL}{\overline{L}}

\newcommand{\oo}{\overline{0}}
\newcommand{\oE}{\overline{E}}
\newcommand{\ocR}{\overline{\cR}}
\newcommand{\oR}{\overline{R}}
\newcommand{\oF}{\overline{F}}

\newcommand{\ocG}{\overline{\cG}}
\newcommand{\oG}{\overline{G}}
\newcommand{\oB}{\overline{B}}
\newcommand{\fX}{\mathfrak{X}}
\newcommand{\fB}{\mathfrak{B}}
\newcommand{\cB}{\mathcal{B}}

\newcommand{\tX}{\widetilde{X}}
\newcommand{\tY}{\widetilde{Y}}
\newcommand{\tR}{\widetilde{R}}

\newcommand{\tS}{\widetilde{S}}
\newcommand{\tG}{\widetilde{G}}
\newcommand{\tDelta}{\widetilde{\Delta}}

\newcommand{\cO}{\mathcal{O}}

\newcommand{\cI}{\mathcal{I}}

\newcommand{\cF}{\mathcal{F}}
\newcommand{\cG}{\mathcal{G}}
\newcommand{\cJ}{\mathcal{J}}

\newcommand{\cR}{\mathcal{R}}

\newcommand{\fa}{\mathfrak{a}}
\newcommand{\fb}{\mathfrak{b}}
\newcommand{\fc}{\mathfrak{c}}
\newcommand{\fm}{\mathfrak{m}}

\newcommand{\fab}{\fa_{\bullet}}
\newcommand{\fbb}{\fb_{\bullet}}

\newcommand{\hX}{\hat{X}}

\newcommand{\hL}{\hat{L}}
\newcommand{\hv}{\hat{v}}

\newcommand{\la}{\lambda}

\newcommand{\Spec}{\mathrm{Spec}}
\newcommand{\Supp}{\mathrm{Supp}}

\newcommand{\mult}{\mathrm{mult}}
\newcommand{\codim}{\mathrm{codim}}
\newcommand{\lct}{\mathrm{lct}}

\newcommand{\im}{\mathrm{im}}

\newcommand{\vol}{\mathrm{vol}}
\newcommand{\Val}{\mathrm{Val}}
\newcommand{\DivVal}{\mathrm{DivVal}}
\newcommand{\ord}{\mathrm{ord}}

\newcommand{\QM}{\mathrm{QM}}

\newcommand{\gr}{\mathrm{gr}}
\newcommand{\Bl}{\mathrm{Bl}}

\newcommand{\norm}[1]{\left\lVert#1\right\rVert}

\numberwithin{equation}{section}

\title[Relative stability and properness]{Relative stability theory and properness of K-moduli spaces}

\author{Harold Blum}
\address{Department of Mathematics, University of Utah, Salt Lake City, UT 84112, USA}
\email{blum@math.utah.edu}
\address{School of Mathematics,
Georgia Institute of Technology, GA 30332, USA}
\email{haroldblum@gatech.edu}

\author{Yuchen Liu}
\address{Department of Mathematics, Northwestern University, Evanston, IL 60208, USA}
\email{yuchenl@northwestern.edu}

\author{Chenyang Xu}
\address{Department of Mathematics, Princeton University, Princeton, NJ 08544, USA}
\email     {chenyang@princeton.edu}

\author{Ziquan Zhuang}
\address{Department of Mathematics, Johns Hopkins University, Baltimore, MD 21218, USA}
\email{zzhuang@jhu.edu}

\begin{document}

\maketitle

\setcounter{tocdepth}{1}

\begin{abstract}
We define the relative stability threshold of a family of Fano varieties over a DVR and show that it is computed by a divisorial valuation. In the case when the special fiber is K-unstable, but the generic fiber is K-semistable, we use the divisorial valuation computing the threshold to replace the special fiber by a new one with a strictly larger stability threshold. Iterating this process yields a new and more direct proof of the properness of the K-moduli space that uses only birational geometry arguments. 
\end{abstract}
\tableofcontents

\section{Introduction}

In recent years, K-stability has been successfully used to construct a moduli theory for Fano varieties and, more generally, log Fano pairs. 
The main output of the theory is the K-moduli space, which is a projective good moduli space parametrizing K-polystable log Fano pairs with fixed numerical invariants.

Arguably, the most important and difficult to prove property of the K-moduli space is its properness. 
Prior to its proof, it was known that the K-moduli space is a finite type separated algebraic space by \cite{Jia20,BLX-openness,ABHLX-goodmoduli}.
Thus verifying the properness  was equivalent to the following statement whose proof was completed in \cite{LXZ-HRFG}.

\begin{thm}\label{t:propKmod}
Let $C$ be a spectrum of a DVR. 
If $(X_K,B_K)$ is a K-semistable log Fano pair over $K:=K(C)$, then there exists a dominant morphism of spectrum of DVRs $C'\to C$ 
such that 
\[
(X_{K'},B_{K'}):=(X_K,B_K)\times_K K' 
\]
extends to a family of log Fano pairs $(X',B') \to C'$ with K-semistable fibers.
\end{thm} 

A first step toward this result was achieved in \cite{LX-special}.
By taking a semistable reduction and running well-chosen MMPs, \cite{LX-special} showed  the existence of an extension of $(X_{K'},B_{K'})$ to a family of log Fano pairs  $(X',B')\to C'$ with special fiber not necessarily K-semistable.
Such an extension is often highly non-unique, and finding a K-semistable limit requires  a  more subtle analysis of which extension to take.

Theorem \ref{t:propKmod}  was proven algebraically in \cite{LXZ-HRFG} by combining  higher rank finite generation results with K-stability arguments in \cite{BHLLX-theta}, the notion of $\Theta$-stratifications in \cite{DHL-Theta}, and stack theoretic results in  \cite{AHLH}.
In the special case when the pair $(X_K,B_K)$ is log smooth, the theorem can alternatively be deduced from deep analytic results on Gromov--Hausdorff limits of K\"ahler-Einstein Fano manifolds \cite{DonSun-GHlimits,TianWang-conicKE}.

 \subsection{Properness via $\Theta$-stratification}
We now explain the prior algebraic proof of properness in more detail.
The proof begins by considering the moduli stack   of  log Fano pairs.
While this stack does not admit a good moduli space, it contains the moduli stack of K-semistable log Fano pairs as an open substack \cite{BLX-openness} and satisfies the existence part of the valuative criterion for properness  \cite{LX-special}. 
 
The  papers \cite{LXZ-HRFG,BHLLX-theta} construct a $\Theta$-stratification on the moduli stack of log Fano pairs. 
 Roughly, this amounts to stratifying the  locus of the stack parametrizing K-unstable log Fano pairs  in terms of an invariant measuring instability and choosing a uniquely determined ``optimal destabilizing'' test configuration for each K-unstable log Fano pair that varies continuously on strata.

 To construct such a stratification, a natural starting point is to consider the stability threshold of a log Fano pair $(X,B)$ defined as
 \[
 \delta(X,B):= \inf_{E} \frac{A_{X,B}(E)}{S_X(E)},
 \]
 where the infimum runs through all divisors $E$ over $X$ \cite{FO-delta,BJ-delta}. 
Recall that  $(X,B)$ is K-semistable if and only if 
$\delta(X,B) \geq 1$ by \cite{Fujvalcrit,Livalcrit}.
By  finite generation results for higher rank valuations in \cite{LXZ-HRFG}, when $(X,B)$ is K-unstable,  this infimum is a minimum and induces a special test configuration that minimizes the functional ${\rm Fut}/ \lVert \, \,  \rVert_{\min}$ on the space of special test configurations.  While this minimizer is not always unique, one can further minimize the function ${\rm Fut} / \lVert\, \,\rVert_{2}$ over the set of minimizers of ${\rm Fut}/\lVert\, \, \rVert_{\min}$. 
The latter minimizer exists, is unique, and induces a $\Theta$-stratification in \cite{ABHLX-goodmoduli}. 

With the existence of the stratification complete, 
the semistable reduction theorem for algebraic stack admitting $\Theta$-stratification  in \cite{AHLH}  then implies that Theorem \ref{t:propKmod} holds.  The latter result relies on a subtle analysis of the local  stack structure of $\Theta$-stratification on an algebraic stack.

In this paper, we give a purely birational geometry proof of the properness of K-moduli spaces. 
The proof replaces the delicate argument of studying a $\Theta$-stratification by a more explicit geometric construction using techniques from birational geometry.
While the argument still uses  higher rank finite generation results of \cite{LXZ-HRFG,XZ-SDC}, they are used in a more direct way.

\subsection{Birational approach to properness}
Let us start with a family of log Fano pairs $(X,B)\to C$ over the spectrum of a DVR $C$ with closed point $0 \in C$ and  $K:=K(C)$.
We assume  
\[
 \delta(X_0,B_0)<
\min \{1 , \delta(X_{K},B_{K})\}
,\] 
which means that the special fiber is less stable than the general fiber. 
Our goal is to construct, after a possible base change by  a  dominant morphism of spectrums of DVRs $C'\to C$,  a new family of log Fano pairs $(X',B')\to C$ such that 
\[
(X_K,B_K)\cong (X'_{K},B'_{K}) \quad \text{ and } \quad 
\delta(X_0, B_0)< \delta(X'_0, B'_0)
.\] 
To begin, observe that the first isomorphism induces a birational map $X\dashrightarrow X'$ and, hence, the new family  induces a divisorial valuation $\ord_{X'_0}$ of $X$. 
We seek to pick out this valuation intrinsically. 

To proceed, we define a relative version of the stability threshold that depends on a parameter $t\in [0,1)\cap \bQ$.
We set
\[
\delta_t( X,B)= \inf_{v} \frac{A_{X,B+(1-t)X_0}(v)}{S(v)}
,\]
where the infimum runs through all valuation $v$ on $X$ with finite log discrepancy.
Above $A_{X,B+(1-t)X_0}(v)$ denotes the log discrepancy and $S(v)$ the relative $S$-invariant (Section \ref{ss-invariants}).
If $v= \ord_E$ where $E$ is a prime divisor on a normal scheme $Y$ with a proper birational morphism $\mu:Y\to X$, then 
\[
S(\ord_E) = \frac{1}{\vol(-K_{X_0})} \int_0^\infty \vol_{Y\vert \widetilde{X}_0}(- \mu^*K_X -sE) \, \mathrm{d}s
,\]
where $\vol_{Y\vert \widetilde{X}_0}$ denotes the restricted volume as in \cite{ELMNP-resvol} onto the strict transform of $X_0$ on $Y$.
We say a valuation $v$  on $X$ \emph{computes} $\delta_t(X,B)$ if it computes the above infimum. 

Using inversion of adjunction, we show that 
$
\delta(X_0,B_0)= \delta_{0}(X,B)
$ and so the invariant recovers nothing new at $t=0$.
For $0<t\ll1$, we show that a minimizer of the infimum exists and it can be used to construct our desired extension. 

\begin{thm}\label{thm-main}
Let $C$ be a spectrum of a DVR with algebraically closed residue field.
Let $0 \in C$ denote the closed point and $K:=K(C)$  the fraction field.

If $(X,B)\to C$ is a family of log Fano pairs  with 
\[
\delta(X_0, B_0)<
\min \{1,\delta(X_K,B_K) \} ,
\]
 then there exists $\varepsilon>0$ such that for any $t\in (0,\varepsilon)\cap \bQ$ the following holds:
\begin{enumerate}
	\item There exists a divisorial valuation $v= \ord_E$  
    that computes $\delta_t(X,B)$. 
	\item If a divisorial valuation $v= \ord_E$  computes $\delta_t(X,B)$, then  there  exists a  dominant morphism of spectrums of DVRs  $C'\to C$  and a family of log Fano pairs $(X',B')\to C'$ such that
	\[
	(X_K,B_K)\times_K K' \cong (X'_{K'}, B'_{K'})  \quad \text{ and } \quad \delta(X_0, B_0)< \delta(X'_{0'},B'_{0'}),
\]
where $0'\in C'$ is the closed point and $K'= K(C')$ the function field.
Furthermore, the restriction of the valuation  $\ord_{X'_{0'}}$  via the inclusion $K(X)\subset K(X')$ is an integer multiple of $\ord_{E}$.
\end{enumerate}
\end{thm}

By repeatedly applying the above theorem and using discreteness results for the stability threshold of a log Fano pair with a fixed volume, we deduce the follow corollary, which is a generalization of Theorem \ref{t:propKmod}.

\begin{cor}\label{cor-deltaconstant}
Let $C$ be a spectrum of a DVR  with  closed point $0$ and fraction field $K$.
If $(X_K,B_K)\to {\rm Spec}(K)$ is a log Fano pair over $K$, then there is a dominant morphism of spectrums of DVRs $C'\to C$ and a family of log Fano pairs $(X',B')\to C'$ such that
\[
	(X_K,B_K)\times_K K' \cong (X'_{K'}, B'_{K'})
\]
and 
\[
\min \{1, \delta(X'_{K'},B'_{K'})\} = \min \{1,  \delta(X'_{0'},B'_{0'}) \} \, .
\]
where $0' \in C'$ is the closed point and $K' := K(C')$.
Furthermore, if $(X_K,B_K)$ is K-semistable, then all fibers of $(X',B')\to C'$ are  K-semistable. 
\end{cor}


As a consequence of the above corollary, we give a new proof of the  properness of the K-moduli space $X^{\rm K}_{n,V,N}$, parametrizing $n$-dimensional K-polystable log Fano pairs $(X,B)$ with $(-K_X-B)^n=V$ and $N B$ is integral \cite[Definition 8.15]{Xu-Kbook}

\begin{cor}\label{c:propKmod}
The K-moduli space  $X^{\rm K}_{n,V,N}$ is proper. 
\end{cor}

\subsection{Proof of Main Theorem}

The proof of  Theorem \ref{thm-main} has two main components. 
We first  begin by developing the theory of this relative stability threshold.
Much, but not all, of the theory on the stability threshold in the absolute case of a log Fano pair \cite{BJ-delta,LXZ-HRFG} extends easily to the relative setting and results in the proof of Theorem \ref{thm-main}.1.
The next step is  to use an appropriate MMP to construct the extension and  prove
\[
\delta(X_0 ,B_0) <\delta (X'_{0'},B'_{0'})
,\]
which turns out to be extremely subtle.

\subsubsection{Relative stability threshold}
We first develop the theory of the relative stability threshold for the data of 
$f:(X,\Delta;L) \to C$,  where  $(X,\Delta)$ is an lc pair, $f: X\to C$ is a flat proper morphism with connected and reduced  geometric fibers, and $L$ is a relatively ample $\bQ$-line bundle on $X$. Fix an integer $r>0$ such that $rL$ is a line bundle.
We write 
\[
\cR:= \bigoplus_{m \in r\bN} \cR_m =\bigoplus_{m \in r\bN} H^0(X,mL) 
\quad \text{and } \quad 
R:= \bigoplus_{m \in r\bN} R_m =\bigoplus_{m \in r\bN} H^0(X_0,mL_0) 
\]
for the relative section rings of $L$ and $L_0$. 
Furthermore, we  assume  that the restriction morphism $\cR \otimes k(0)\to R$ is surjective.

We define the \emph{relative stability threshold} of $(X,\Delta;L)\to C$ as
\[
\delta(X,\Delta;L) = \inf_{v}\frac{A_{X,\Delta}(v)}{S(v)},
\] 
where the infimum runs through all valuations $v$ on $X$ with finite log discrepancy. Here $A_{X,\Delta}(v)$ denotes the log discrepancy of $v$ along $(X,\Delta)$ and $S(v)$ is the relative $S$-invariant.
To define the latter invariant, let $\cF_v$ be the filtration of $\cR$ defined by 
\[
\cF_v^\la \cR_m := \{s \in \cR_m \, \vert\, v(s)\geq \la\}
\]
and $\cF_v\vert_{X_0}$ be the filtration of $R$ defined via restriction:
\[
(\cF_v\vert_{X_0})^\la R_m := \im (\cF_v^\la \cR_m \otimes k(0) \to R_m) \subset R_m.
\]
We then define  $S(v)$ as the  $S$-invariant of the filtration $\cF_v\vert_{X_0}$ of $R$.
This invariant may also be interpreted as an integral of restricted volumes when $v$ is divisorial (Remark \ref{r:relationtoresvol}).

\medspace

\noindent \emph{Relation to basis type divisors}.
For $m \in r\bN^+$,  a \emph{relative $m$-basis type divisor} of $L$ is a $\bQ$-divisor on $X$ of the form 
\[
\tfrac{1}{mN_m} \left( \{s_1=0\}+ \cdots + \{ s_{N_m}=0\} \right),
\]
where $(s_1,\ldots, s_{N_m})$ is a basis for $\cR_m$ as a free $\cO(C)$-module. 
We set 
\[
\delta_m(X,\Delta;L) := \inf \{ \lct(X,\Delta;D) \, \vert\,  \text{$D$ is an $m$-basis type divisor of $L$}\}
.\]
By the same argument as in the absolute case in \cite{BJ-delta}, the following holds:
\begin{enumerate}
	\item $S(v)= \lim_{m \to \infty} S_{mr}(v)$, where 
	\[
	S_m(v) = \max \{ v(D) \, \vert\,\text{ $D$  is an $m$-basis type divisor of $L$}\},
	\]
		and the infimum is achieved when $D$ is compatible with $v$ (see Section \ref{sss:basistypedivisors}), 
	\item $\delta_m(X,\Delta;L)= \inf_v \frac{A_{X,\Delta}(v)}{S_m(v)}$ (Proposition \ref{p:delta=inf A/S}),  and
	\item $\delta(X,\Delta;L):= \lim_{m \to \infty} \delta_{mr}(X,\Delta;L)$ (Proposition \ref{p:delta=inf A/S}).
\end{enumerate}

\vspace{.2 cm}

\noindent \emph{Existence of quasi-monomial minimizer.}
Under mild assumptions, we show that there exists a quasi-monomial valuation $v$ that computes the infimum in the definition of $\delta(X,\Delta;L)$ (Theorem \ref{t:delta computed by qm valuation}).
This is achieved by extending the generic limits argument for constructing minimizers of the stability threshold in the absolute case in  \cite{BJ-delta} to the relative case and applying the main result of \cite{Xu-quasimonomial} to choose a minimizer that is quasi-monomial.

\vspace{.2 cm}

\noindent \emph{Existence of divisorial minimizer.}
Next we specialize to the case when $(X,\Delta)\to C$ is a relative log Fano pair, i.e. $(X,\Delta)$ is klt and  $-K_{X}-\Delta$ is ample over $C$.
(The main example to  keep in mind is, given by taking a  family of log Fano pairs $(X,B)\to C$ and then considering $(X,B+(1-t)X_0) \to C$ with $t\in (0,1]\cap \bQ$.)
In this case, we set 
\[
\delta(X,\Delta):= \delta(X,\Delta; -K_{X}-\Delta)
,\]
where we view $-K_{X}-\Delta$ as a $\bQ$-line bundle.

We show that if a relative log Fano pair $(X,\Delta) \to C$ satisfies 
\[
\delta(X,\Delta)< \min \{1 , \delta(X_K, \Delta_K)\}
\]
and other mild assumptions hold, then there exists a divisorial valuation computing $\delta(X,\Delta)$.
Following the strategy in \cite{LXZ-HRFG}, we take a quasi-monomial valuation $v$ computing the stability threshold  and then show that the associated graded ring 
\[
{\rm gr}_v \cR: = \bigoplus_{m \in r\bN} \bigoplus_{\la \in \bR} \cF_v^\la \cR_m / \cF_v^{>\la} \cR_m 
\]
is finitely generated. 
Similar to arguments in the absolute case, finite generation implies a linearity property for the $S$-invariant near $v$ that implies the existence of a divisorial valuation computing $\delta(X,\Delta)$.

To prove finite generation, we aim to apply the higher rank finite generation results from \cite{LXZ-HRFG,XZ-SDC}.
For this we need to show that the quasi-monomial minimizer $v$ is an lc place of a special $\bQ$-complement of $(X,\Delta)$ (see Definition \ref{d:specialcomp}).
A key technical ingredient towards verifying the latter condition is to show that, for any horizontal prime divisor $G$ on $X$, there exists a complement $\Delta^+$ of $(X,\Delta)$ such that $v$ is an lc place of $(X,\Delta^+)$ and $\Supp(\Delta^+)\supset \Supp(G)$.
To verify this condition, we would like to take  a relative $m$-basis type divisor $D_m$ for each $m \in r\bN^+$
such that $D_m$ is compatible with $v$ and  $\ord_G$.
Unfortunately, we cannot a priori assume that both conditions hold as, unlike the vector space case, we cannot simultaneously diagonalize any two filtrations of a free module over a DVR.

To circumvent the latter issue, we consider asymptotically increasing thickenings of $X_0$. 
In particular, for $0<c\ll1$, we consider bases for $\cR_m/\fm_{C,0}^{\lceil cm\rceil }\cR_m$ over the residue field $\bk$, lift these sections to $\cR_m$, and then define the notion of an $(m,c)$-basis type divisor $D_{m,c}\sim_{\bQ} -K_{X}-\Delta$. 
Using this, we can define invariants  $S_{c}(\Delta)$, $\delta_c(X,\Delta)$ satisfying $\delta_c(X,\Delta)= \inf_{v} \frac{A_{X,\Delta}(v)}{S_c(v)}$ defined using $(m,c)$-basis type divisors rather than $m$-basis type divisors.
We show that if $v$ computes $\delta:=\delta(X,\Delta)$, then $v$ also computes $\delta_c(X,\Delta-c\delta X_0)$.
By simultaneously diagonalizing the filtrations of $\cR_m/ t^{\lceil cm\rceil}\cR_m$ induced by $v$ and $G$, we 
can construct a $(m,c)$-basis type divisor $D_m$ compatible with both $v$ and containing a bounded from zero multiple of $G$. 
Using this, we construct a complement $\Delta^+$ of $(X,\Delta)$ such that $v$ is an lc place of $(X,\Delta-c\delta X_0)$ and $\Supp(\Delta^+) \subset \Supp(G)$.
As a consequence, we deduce that $v$ is a special lc place of $(X,\Delta- c\delta X_0)$ and, hence, deduce the finite generation of ${\rm gr}_v \cR$ using \cite{XZ-SDC}.

\subsubsection{Extension induced by minimizer}
We now explain the proof of Theorem \ref{thm-main}.2. 
Let 
$
(X,B)\to C
$
 be a family of log Fano pairs satisfying the assumptions of the theorem.
We now fix  $0<t\ll1$ and set 
\[
\Delta:= B+(1-t)X_0.
\]
Let  $\ord_E$ be a divisorial minimizer of $\delta(X,\Delta)$ whose existence is guaranteed by Theorem \ref{thm-main}.1.
For simplicity, we assume that $\ord_E(X_0)=1$ in order to avoid having to take a base change in the statement of Theorem \ref{thm-main}.2.

We first show that if $0<t\ll1$, then $\ord_E$ is an lc place of a $\bQ$-complement of $(X,B+X_0)$. 
Using this, we  use the MMP to construct a flat proper family of pairs 
\[
(X',B') \to C
\]
with an isomorphism 
$
(X_{K},B_{K}) \cong (X'_{K},B'_{K})
$
over  $K$  such that  $X'_0$ is integral and corresponds to the divisor $E$ over $X$.
Furthermore, we argue that $(X',B')\to C$ is a family of log Fano pairs. 

To verify the inequality of stability thresholds, we aim to show that
\[
\delta(X_0,B_0) = \delta(X,B+X_0)<\delta(X, \Delta) 
\leq  \delta(X',B'+X'_0) = \delta(X'_0,B_0').
\]
The first and last equalities are simple consequences of adjunction  and the characterization of the stability threshold in terms of lct's, while  the second uses that
\[
\delta(X,B+X_0) \leq \frac{A_{X,B+X_0}(\ord_E)}{S(\ord_E)} < \frac{A_{X,\Delta}(\ord_E)}{S(\ord_E)} = \delta(X,\Delta)
\]
Verifying the remaining inequality
\begin{equation}\label{eq:introhardineq}
\delta(X, \Delta) 
\leq  \delta(X',B'+X'_0)
\end{equation}
requires comparing birational models and is very delicate.  

To build intuition for why \eqref{eq:introhardineq} holds, let's first impose the simplifying  assumption that $\ord_E$ also computes the infimum 
\[
\delta_m(X,\Delta) = \inf_{v}\frac{A_{X,\Delta}(v)}{S_m(v)}
\]
for all $m \in r\bN^+$.
Using that $\ord_E$ is the lc place of a complement, we construct a proper birational model of $X$
	\[
\begin{tikzcd}
	& Y \arrow[rd,dotted, " "] \arrow[ld, "g"'] & \\
	X  & & X'
\end{tikzcd}
\]
such that $E_Y := {\rm Exc}(g)$ is a prime divisor, $E_Y$ is the birational transform of $E$ on $Y$, and $-E_Y$ is ample over $X$.
A key observation is that if $D'_m$ is an $m$-basis type of $-K_{X'}-B'$ compatible with $\ord_{X_0}$, then its birational transform $D_m$ on $X$ is an $m$-basis type divisor of $-K_{X}-B$ compatible with $\ord_{X_0'}$ (Lemma \ref{lem-dualfiltration}). 
Since $\ord_E$ computes $\delta_m(X,\Delta)$, we know that
\[
(X,\Delta+\delta_m D_m)
\]
 is lc and 
\[
g^*(K_X+\Delta+\delta_m D_m)=K_Y+g_*^{-1}(\Delta+ \delta_m D_m)+E_Y
,\]
where $\delta_m:= \delta_m(X,\Delta)$, hence the pair
\[
(Y, g_*^{-1}(\Delta+\delta_mD_m)+E_Y)
\] is lc and admits a $\bQ$-complement when $m\gg 0$. 
This implies that the pair
\[
(X',B'+X'_0 +\delta_m D_m')
\]
 is lc.
Therefore 
\begin{align*}
	\delta_m(X,\Delta)& \leq   \min \{ {\rm lct}(X',X_0'+B'; D_m') \, | \,
	D_m' \text{ is a relative $m$-basis type divisor of }\\
	& \hspace{ 2 in}\text{ $-K_{X'}-B'$  compatible with }\ord_{X_0} \} \\
&= \min \{ {\rm lct}(X',X_0'+B'; D_m') \, | \, D_m'\mbox{ is a relative $m$-basis type divisor of}  \\
	&\hspace{ 2 in} \text{ $-K_{X'}-B'$} \} \\
	& =: \delta_m(X',B'+X'_0),
\end{align*}
where the first equality relies on inversion of adjunction and simultaneously diagonalizing two filtrations of the section ring of $-K_{X'_0}-B_0 $ on $X_0$ as in   \cite[Proposition 1.6]{AZ-K-adjunction}.
Sending $m \to \infty$, we then deduce that 
\[
\delta(X,\Delta)\leq \delta(X',B'+X_0)
\]
 under our simplifying assumption.

The proof of the result without this simplifying assumption is significantly more complicated and requires a delicate analysis of convergence rates. 
To proceed, we set
\begin{align*}
\la_{m}: = \inf \{ \lct( Y, g_*^{-1}(\Delta)+E_Y; g_*^{-1} D_m)\, \vert\, & D_m \text{ is a relative $m$-basis type divisor of} \\ &\text{$-K_{X}-B$ compatible with $\ord_E$ }\}
.\end{align*}
Arguing similar to before, we show that 
\[
\min\{ \la_m, \delta_m(X,\Delta)) \} \leq  \delta_m(X',B'+X'_0) 
\]
for $m\gg0$.
If we can show that 
\begin{equation}\label{e:introdeltalambda}
\delta(X,\Delta) \leq \liminf_m \la_{mr}
,
\end{equation}
then sending $m\to \infty$ gives $\delta(X,\Delta)\leq \delta(X',B'+X'_0)$ as desired. 

To verify the key inequity \eqref{e:introdeltalambda}, we first compute  that 
\[
\inf_{w\in \DivVal_X}  \frac{A_{X,\Delta}(w)-w(E_Y)A_{X,\Delta}(E)}{S_m(w)-w(E_Y)S_m(E)}
\leq 
\la_m
\] 
(see Lemma \ref{l:la_minequality}).
Since
\[
\delta(X,\Delta) = \frac{A_{X,\Delta}(\ord_E)}{S(\ord_E)}
\leq 
\frac{A_{X,\Delta}(w)}{S(w)}
\]
for $w\in \DivVal_X$,  
\[
\delta(X,\Delta) \leq \frac{ A_{X,\Delta}(w)-w(E_Y)A_{X,\Delta}(\ord_E)}{S(w) -w(E_Y)S(\ord_E)}.
\]
holds.
Therefore, if we can control the convergence of the limit
\[
S(w)-w(E_Y)S(\ord_E) =
\lim_{m \to \infty} \big(S_{mr}(w) -w(E_Y)S_{mr}(\ord_E)\big),
\]
uniformly for all $w\in \DivVal_X$, then we will be able to deduce \eqref{e:introdeltalambda} as desired. 
In particular, we need to show that for any $\varepsilon>0$, there exists $m_0:=m_0(\varepsilon)$ such that 
\begin{equation}\label{e:introconvergence}
S_{mr}(w)-w(E_Y)S_{mr}(\ord_E)  \leq (1+\varepsilon) (S(w)-w(E_Y)S(\ord_E))
\end{equation}
for all $m \geq m_0$  and  $w\in \DivVal_X$.
This is reminiscent of a key convergence result \cite[Corollary 3.6]{BJ-delta}, which shows  the corresponding statement for the inequality 
$S_{mr}(w)\leq (1+\varepsilon)S(w)$ in the absolute case.
Geometrically, the statement in \cite{BJ-delta} is related to the value
$w(D_m)$, where $D_m$ is an $m$-basis type divisor compatible with $w$, while  \eqref{e:introconvergence} is related to  vanishing order of 
\[
 w(g_*^{-1}D_m)
 =w(D_m)-\ord_E(D_m)w(E_Y)
.\]

To prove \eqref{e:introconvergence}, we first verify a version of the statement  in the absolute case by using an Okounkov body argument as in \cite{BJ-delta} by choosing $E$ to be the divisor in the chosen flag (see Theorem \ref{prop-comparetwofiltration}).
Unfortunately, we cannot deduce the relative version from the absolute version as the filtration of $\cR$ induced by $\ord_E$ does not necessarily restrict to a filtration of $R$ induced by a single valuation. 
Instead, we use $\ord_E$ to degenerate $X_0$ to a reduced projective scheme for which we can apply the absolute case to the irreducible components. 
This last step requires  generalizing various results on asymptotic invariants of filtrations of section rings of projective varieties to reduced  projective schemes (Section \ref{ss:reducible}).

\subsection{Alternative approaches to properness}
The proof of properness in \cite{LXZ-HRFG,BHLLX-theta} relies on using divisorial minimizers of the $\delta$-invariant to construct a $\Theta$-stratication of the moduli stack of log Fano pairs.
Minimizers of various other K-stability related functionals on the valuation space give rise to other approaches to proving properness.

The papers \cite{Odakanovikov,Odakacones} pursue the approach to proving properness using the minimizer of the $H$-invariant and the minimizer of the normalized volume function on the cone to construct optimal destabilizing $\mathbb{R}$-test configurations. 
These destabilizations were constructed in \cite{BLXZ-soliton,XZ-SDC} as a consequence of higher rank finite generation results. 
As the latter destabilizations are  $\mathbb{R}$-test configurations, this approach relies on defining a higher rank version of $\Theta$-stratifaction and generalizing the semistable reduction result in \cite{AHLH}  to the higher rank case \cite{Odakacones,BHLNK}.

\subsection{Extensions}
In an upcoming paper, we will  show that the techniques of this paper can be used to prove the properness of the moduli space of weighted K-semistable log Fano pairs. As a consequence, we will deduce the properness of moduli stack of  K-semistable log Fano cone singularities and log Fano pairs admitting K\"ahler-Ricci soliton.

\vspace{.5 cm}

\noindent {\bf Structure of the Paper.}
The paper is structured as follows. 
In Section \ref{s:prelims}, we discuss preliminaries on pairs, valuations, and filtrations of graded linear series. In particular, we generalize  results in the literature on filtered graded linear series of projective varieties to the case of reduced projective schemes. 
In Section \ref{sec-relative stability}, we define the relative stability threshold, prove basic properties of the invariant, and the existence of a quasi-monomial valuation computing the invariant. In Section \ref{ss-finite generation}, we prove the existence of a divisorial valuation computing the invariant in the log Fano case. 
In Section \ref{s-increasestabilitythreshold}, we use the divisorial minimizer of the stability threshold to construct a replacement of the special fiber with increased stability threshold.

\vspace{.5 cm}

\noindent {\bf Acknowledgment.}
This collaboration started at a SQuaRE sponsored by the American Institute of Mathematics, to which the authors are all grateful.
HB is partially supported by NSF Grants DMS-2441378 and DMS-2200690.
YL is partially supported by NSF Grant DMS-2237139 and an AT\&T Research
Fellowship from Northwestern University.
CX is partially supported by NSF Grant DMS-2201349 and a Simons Investigator grant.
ZZ is partially supported by the NSF Grants DMS-2240926, DMS-2234736, a Sloan research fellowship and a Packard fellowship.
CX and ZZ are also partially supported by the Simons Collaboration Grant on Moduli of Varieties.

\section{Preliminaries}\label{s:prelims}

\subsection*{Notation and conventions}\label{ss-notation}
Throughout, all schemes are defined over $\bQ$ and all DVRs contain $\bQ$. 
If $C$ is the spectrum of a DVR, we will always write $0$ and $\eta$ for the closed point and the generic point unless stated otherwise.

We will generally follow the notation for the singularities of pairs in \cites{KM-book,Kol13,Kol23} and K-stability of log Fano pairs in \cite{Xu-Kbook}.

\subsubsection{Pairs}
A pair $(X,\Delta)$ is the data of a scheme $X$ over $\bQ$ that is normal, Noetherian, excellent, integral, and with a canonical class \cite[Definition 11.3]{Kol23} and an effective $\bQ$-divisor $\Delta$ such that $K_X+\Delta$ is $\bQ$-Cartier.
For notions of singularities of pairs, e.g. lc, klt, dlt, see \cite[Definition 11.5]{Kol23}. 
A pair over a scheme $Z$ is a pair $(X,\Delta)$ with a morphism $X\to Z$.

We use the above definition of a pair (that does not assume the scheme is finite type over a field) as we are primarily interested in pairs $(X,\Delta)$ where $X$ is projective over the spectrum of a DVR.
The major results of the MMP in \cite{BCHM} are generalized to the above class of pairs in \cite{Lyu-Murayama}.

A \emph{$\bQ$-complement} of a pair $(X,\Delta)$ over $Z$ is the data of a $\bQ$-divisor $\Delta^+$ such that $(X,\Delta^+)$ is lc, $K_{X}+\Delta^+\sim_{\bQ,Z} 0$, and $\Delta^+\geq \Delta$.
We often, a bit abusively, also refer to $\Gamma:= \Delta^+-\Delta$ as the complement $(X,\Delta)/Z$.

\subsubsection{Log Fano pairs}
A \emph{log Fano pair} is a klt pair $(X,\Delta)$ such that $X$ is projective over a field $\bk$, $X$ is geometrically connected, and $-K_X-\Delta$ is ample. 
A \emph{relative log Fano pair} is a klt pair $(X,\Delta)$ over $Z$ such that $X\to Z$ is projective and $-K_X-\Delta$ is ample over $Z$. A pair $(X,\Delta)$ is of Fano type over $Z$ if there exists some $\bQ$-divisor $\Delta'\ge \Delta$ such that $(X,\Delta')$ is a relative log Fano pair.

A \emph{family of log Fano pairs} $(X,\Delta)\to C$ over the spectrum of a DVR $C$ 
is the data of a flat proper  morphism $X\to C$ with geometrically connected  fibers and a $\bQ$-divisor $\Delta$ on $X$ such that $(X,\Delta+X_0)$ is a plt pair and $-K_{X}-\Delta$ is ample over $C$.
Note that $(X,\Delta) \to C$ is a locally stable family in the sense of \cite[Section 2.1]{Kol23}
and that the fibers are log Fano pairs by adjunction. 

\begin{lem} \label{l:Fanotype}
Let $X\to Z$ and $f\colon Y\to X$ be projective morphisms. Assume that $(X,0)$ is of Fano type over $Z$ and there exists some $\bQ$-complement $\Delta$ of $X$ over $Z$ such that $A_{X,\Delta}(E)<1$ for every $f$-exceptional divisor $E$, then $(Y,0)$ is of Fano type over $Z$.
\end{lem}

\begin{proof}
By the Fano type condition on $X$, we can find some $\bQ$-divisor $\Delta'$ on $X$ such that $(X,\Delta')$ is klt and $-(K_X+\Delta')$ is ample over $Z$. Replacing $\Delta'$ with $(1-\varepsilon)\Delta+\varepsilon\Delta'$ for some $0<\varepsilon\ll 1$, we may further assume that $A_{X,\Delta'}(E)<1$ for every $f$-exceptional divisor $E$. Define a $\bQ$-divisor $\Delta'_Y$ on $Y$ by the crepant pullback formula
\[
K_Y+\Delta'_Y = f^*(K_X+\Delta').
\]
Then $\Delta'_Y\ge 0$, the pair $(Y,\Delta'_Y)$ is also klt over $Z$, and $-(K_Y+\Delta'_Y)$ is big and nef over $Z$. Thus $Y$ is of Fano type over $Z$. 
\end{proof}

\subsubsection{Valuations}
Let $X$ be an integral scheme over $\bk$. A \emph{valuation} on $X$ is a valuation $K(X)^\times \to \bR$ such that $v(\bk^\times)=0$. We write $\Val_X$ for the set of valuations on $X$, and $\DivVal_X$ for the subset of divisorial valuations.
The center of a valuation $v\in \Val_X$ is denoted by $c_X(v) \in X$. We write $C_X(v):=\overline{c_X(v)}$ for its closure.

For a pair $(X,\Delta)$, there is a log discrepancy function $A_{X,\Delta}: \Val_X \to \bR$, see \cite[Definition 1.34]{Xu-Kbook}.
We write $\Val_{X}^\circ:= \{ v\in \Val_{X} \, \vert\, A_{X,\Delta}(v) < \infty\}$, which is independent of the choice of $\Delta$ by \cite[Lemma 1.36]{Xu-Kbook}. 
An lc place of an lc pair $(X,\Delta)$ is a valuation $v$ on $X$ with $A_{X,\Delta}(v)=0$.

\subsubsection{Filtrations}\label{sss:deffilt}
 We now formulate a definition of a filtration of a graded ring that we will frequently apply to both section rings and relative section rings.

 Let $R:= \bigoplus_{m \in r\bN} R_m$ be a graded ring indexed by $r\bN$, where $r$ is a fixed positive integer. 
 A \emph{filtration} $\cF$ of $R$ is the data of $R_0$-submodules 
 $\cF^\la R_m \subseteq R_m$ for each $m \in r\cdot\bN$ and $\la \in \bR$ satisfying 
  \begin{enumerate}
 \item $\cF^\la R_m \subseteq F^{\la'} R_m$ when $\la \geq \la'$,
 \item $\cF^{\la} R_{m} = \bigcap_{\la' < \la }  \cF^{\la'} R_m$,
 \item $\cF^\la R_m \cdot \cF^{\la'} R_{m'} \subseteq \cF^{\la+\la'}R_{m+m'}$,
 \item  $\cap_{\la \in \bR}\cF^\la R_m = 0$ and  $\cF^{-\la} R_m = R_m$ for $\la\gg0$.
  \end{enumerate}
The filtration is  an \emph{$\bN$-filtration} if $\cF^0 R_m = R_m $ and $\cF^\la R_m  = \cF^{\lceil \la \rceil} R_m$ for all $\la \in \bR$ and $m \in r\bN$. 
The filtration  is  \emph{linearly bounded} if there exists $C>0$ such that $\cF^{-Cm}R_m =R_m$ and $\cF^{Cm} R_m=0$ for all $m \in r\bN$. 

 The \emph{associated graded ring} of a filtration $\cF$ of $R$ is the graded  ring
 \[
  {\rm gr}_{\cF}R := \bigoplus_{m \in r\bN} \bigoplus_{\la \in \bR} {\rm gr}_{\cF}^\la R_m,
 \]
where ${\rm gr}_{\cF}^\la R_m : = \cF^\la R_m / \cF^{>\la } R_m $ and $\cF^{>\la} R_{m} :=\cap_{\la'>\la} \cF^{\la'}R_m$. 
For $s\in R_m$, we set 
\[
\ord_{\cF}(s) := \max \{ \la \in \bR\, \vert\,  s\in \cF^\la R_m\}.
\]


\subsection{Filtered graded linear series}
We now discuss asymptotic invariants of filtered graded linear series and prove some new convergence results. 
\medskip

Throughout, let $X$ be a projective geometrically reduced scheme over a characteristic 0 field $\bk$, $L$ a $\bQ$-line bundle on $X$, and $r$ a positive integer such that $rL$ is a line bundle.

\subsubsection{Graded linear series}
We write
\[
R=R(X,L)= \bigoplus_{m\in r\bN} R_m := \bigoplus_{m \in r\bN} H^0(X,mL)
\]
for the section ring of $L$ with respect to the index $r$.
A \emph{graded linear series} of $L$ is a graded $\bk$-subalgebra
\[
V_\bullet 
= \bigoplus_{m \in r\bN}  V_m 
= \bigoplus_{m \in r\bN} R_m
\subset 
R
.\]
The \emph{volume} of a graded linear series $V_\bullet$  of 
$L$ is 
\[
\vol(V_\bullet):=   \limsup_{m\to \infty} \frac{ \dim_{\bk} V_{mr}}{(mr)^{n} /n!}
,\]
where $n:= \dim X$.
When $X$ is irreducible, we say that a graded linear series $V_\bullet \subset R_\bullet$ \emph{contains an ample series} if  $V_{mr}\neq 0$ for  $m\gg0$ and
there exists $\bQ$-Cartier divisors $A$ and $E$, ample and effective respectively, such that $L= A+E$ and 
\[
H^0(X,mA)\subset V_m \subset H^0(X,mL).
\]
for all $m>0$ sufficiently large and divisible.

\begin{prop}\label{p:vollimit}
If $X$ is  geometrically integral and $V_\bullet$ is a graded linear series of $L$ that contains an ample linear series, then $\vol(V_\bullet)$ is a limit and is positive.
\end{prop}

\begin{proof}
This was first proven in \cite{LM09} and is stated in \cite[Theorem 1.11]{Xu-Kbook} using our terminology.
\end{proof}

\subsubsection{Filtrations}
A \emph{filtration} $\cF$ of a graded linear series $V_\bullet$ is a filtration of the graded $\bk$-algebra as in Section \ref{sss:deffilt}.
A basis $(s_1,\ldots, s_{N_m})$ for $V_m$ is \emph{compatible} with $\cF$ if each $\cF^\la V_m$ is a span of the $s_i$ or, equivalently, $\overline{s_i}\in {\rm gr}_{\cF}^{\la_i}V_m$ form a basis for ${\rm gr}_{\cF} V_m$, where $\la_i := \ord_{\cF}(s_i)$.
For $t\in \bR$, we write $V_\bullet^{\cF,t}$ for the graded linear series of $L$
defined by $V_m^{\cF,t}:= \cF^{mt}V_m$.

The $S$ and $T$ invariants of a filtration $\cF$ of $V_\bullet$ are defined by setting
\[
S_m(\cF):= \frac{1}{m\dim V_m} \sum_{\la \in \bR} \lambda\cdot \dim {\rm gr}_{\cF}^\la V_m 
\quad \text{ and } \quad 
T_m(\cF) := \frac{1}{m} \sup\{ \la \in \bR \, \vert\,  \cF^\la V_m \neq 0\}.
\]
for $m >0$ divisible by $r$ and then setting
\[
S(\cF) = \limsup_{m \to \infty} S_{mr}(\cF) \quad \text{ and } \quad 
T(\cF) = 
\sup_{m \geq 1} T_{mr}(\cF) 
.\]
When the choice of $V_\bullet$ is not clear from context, we write $S(\cF,V_\bullet)$ and $T(\cF,V_\bullet)$ for these invariants.

\begin{prop}\label{p:irredSproperties}
Assume that $X$ is geometrically integral.
If $V_\bullet$  a graded linear series of $L$ that contains an ample series
and $\cF$ is a linearly bounded filtration of $V_\bullet$, then the following hold:
\begin{enumerate}
\item $S(\cF) = \lim_{m\to \infty} S_{mr}(\cF)$  and $T(\cF) = \lim_{m \to \infty} T_{mr}(\cF)$.
\item $V_\bullet^{\cF,t}$ contains an ample series for $t< T(\cF)$.

\item If $\cF^0 V_\bullet = V_\bullet$  and $n:= \dim X$, then
\[
(n+1)^{-1} T(\cF) \leq S(\cF) \leq T(\cF)
\]
and 
\[
S(\cF) = \frac{1}{\vol(V_\bullet)} \int_0^\infty \vol(V_\bullet^{\cF,t})\, dt.
\]

\end{enumerate}
\end{prop}

\begin{proof}
See \cite[Sections 3.1 and 3.2]{Xu-Kbook}.
\end{proof}

\begin{prop}\label{p:irredBJinequality}
Assume that $X$ is geometrically integral. 
If $V_\bullet$ is a graded linear series of $L$ that contains an ample series, then for any $\varepsilon>0$, there exists $m_0(\varepsilon)$ such that 
\[
S_{mr}(\cF) \leq (1+\varepsilon)S(\cF)
\]
for all linearly bounded filtration $\cF$ of $V_\bullet $ with $\cF^0 V_\bullet = V_\bullet$ and $m \geq m_0(\varepsilon)$
\end{prop}

\begin{proof}
See \cite[Theorem 3.33]{Xu-Kbook}, which is a generalization of \cite[Corollary 2.10]{BJ-delta}
\end{proof}

\subsubsection{Reducible schemes}\label{ss:reducible}
We will now extend convergence results for the $S$-invariant to the case when $X$ is not necessarily irreducible.
To begin, we write
\[
X= X_1 \cup \cdots \cup X_s,
\]
where each $X_i$ is a distinct irreducible component of $X$,
and  set $L_i := L \vert_{X_i}$.
We construct a graded linear series $V_\bullet^i$ of $L_i$ 
for each $1 \leq i \leq s$ 
by setting
\[
W_m^{i}:= V_m \cap H^0(X,L \otimes \cI_{X_1 \cup \cdots \cup X_{i-1}})
\subset H^0(X,mL)
\] 
and   
\[
V_m^i : = {\rm im} (W_{m}^{i} \to H^0(X_i,m L_i)) 
\subset H^0(X_i, mL_i)
,\]
where the previous map is induced by the restriction map  on sections
\[
r_i: H^0(X,mL)\to H^0(X_i,mL_i).
\]
Note that $V_\bullet^i$ is a graded linear series of $L_i$.

\begin{lem}\label{l:Vicontainsampleseries}
The following hold:
\begin{enumerate}
	\item For $m \in r\bN$, $\dim H^0(X,m L) = \sum_{i=1}^r \dim V_m^i$ 
    \item If $L$ is ample, then  each $V_\bullet^i$   contains an ample series.
\end{enumerate}	
\end{lem}

\begin{proof}
For each $1\leq i \leq s$ and $m \in r\bN$, we have a short exact sequence 
\[
0 \to W_m^{i+1} \to W_m^i  \to V_m^{i}  \to 0
,\]
where we use the convention that  $W_m^{s+1} =0$. 
Thus
\[
\sum_{i=1}^s \dim V_m^i
=
\sum_{i=1}^s (\dim W_m^{i}- \dim W_{m}^{i+1})
=
\dim W_m^1
=
\dim H^0(X,mL)
,\]
which completes the proof of (1).

Now assume that $L$ is ample.
Let $\cI_{i} := \cI_{X_1 \cup \cdots \cup X_{i-1}} $.
For $m\in r\bN$, the  short exact sequence
\[
0 \to \cO_X(mL)\otimes\cI_{i+1}
\to \cO_{X}(mL) \otimes \cI_{i } \to \cO_{X}(mL)\otimes \cI_{i} \vert_{X_i}
\to 0 
\]
induces a long exact sequence 
\[
0 \to W_m^{i+1} \to   W_m^i \to H^0(X_i, L_i \otimes \cI_i \vert_{X_i}) \to H^1( X, \cO_X(mL)\otimes \cI_i)
\to \cdots . 
\]
When additionally $m\gg0$, $H^1( X, \cO_X(mL)\otimes \cI_i)=0$ by Serre vanishing and so
\[
V_m^i = H^0(X_i, \cO_{X_i} (mL_i)\otimes \cI_i \vert_{X_i})
.\]
Since $L_i$ is ample, it follows that  $V_m^i\neq 0$ for $m\in r\bN$ sufficiently large.
Next,  choose a Cartier divisor $D_i$ on $X_i$ such that $\cO_{X_i}(-D_i)\subset \cI_i \vert_{X_i} $.
Fix $l\gg0$ and divisible by $r$  such that  $A_i:= l L_i -D_i$ is ample.
For $m\gg0$, we have 
\[
H^0( X_i, ml A_i)
\subset 
H^0(X, \cO_{X_i}(mlL_i) \otimes \cI_i \vert_{X_i}) 
=
V_{ml}^i 
\subset H^0(X_i, mlL_i)
.\]
Therefore $V_\bullet^i$ contains ample series.
\end{proof}

We now generalize the previous construction to the setting of filtrations. 
Given a filtration $\cF$ of $R$, 
 we define a filtration $\cF_i$ of $V_\bullet^i$ by setting 
\[
\cF_i^\la V_m^i := \im \left(\cF^\la R_m \cap W_{m}^i \to V_m^i \right)
.\]
It is straightforward to check that   $\cF_i$ is indeed a filtration of $V_\bullet^i$ and is linearly bounded when $\cF$ is linearly bounded.

\begin{lem}\label{l:formulafiltsonXi}
If $\cF$ is a filtration of $R$, then
$S_m(\cF) =  \sum_{i=1}^s\frac{ \dim V_m^i}{ \dim R_m } S_m(\cF_i)
$ 
and 
$T_m(\cF)= \sup_{i=1,\ldots, s} T_{m}(\cF_i)$ 
for each $m>0$ divisible by $r$.
\end{lem}

\begin{proof}
For each $i \in \{1,\ldots, s\}$, $m \in r\bN$, and $\la\in \bR$, there is  a short exact sequence
\begin{equation}\label{eq:sesF}
0 \to W_m^{i+1} \cap \cF^\la R_m \to W_m^i \cap \cF^\la R_m \to  \cF_i^\la V_m^i \to 0,
\end{equation}
where  $W_m^{s+1}  \cap \cF^\la R_m= 0$.
By arguing in the proof of Lemma \ref{l:Vicontainsampleseries},
\[
\dim \cF^\la R_m
 = \sum_{i=1}^s \dim \cF_i^\la V_m^i  
,\]
which implies that
\[
\dim {\rm gr}_\cF^\la R_m 
= \sum_{i=1}^s
\dim  {\rm gr}_{\cF_i}^\la V_m^i
.\]
Using the definition of  $S_m$ and $T_m$, we deduce the desired inequality.
\end{proof}

\begin{prop}\label{p:reducedSinvproperties}
If $L$ is ample, then for any linearly bounded filtration $\cF$ of $R$ the following hold:
\begin{enumerate}
\item $S(\cF)=\lim_{m \to \infty} S_{mr}(\cF)$ and $T(\cF) = \lim_{m \to \infty} T_{mr}(\cF)$.
\item $\vol(V_\bullet^{\cF,t})=\sum_{\dim X_i=\dim X} \vol(V_\bullet^{\cF_i,t})$ for $t\in \bR\setminus \{ T(\cF_1), \ldots, T(\cF_s)\}$.
\item If $\cF^0R= R$, then $S(\cF)= \frac{1}{\vol(V_\bullet)} \int_0^\infty \vol(V_\bullet^{\cF,t})\, dt$ and $S(\cF) \leq T(\cF)$. 
\item If  $\cF^0 R= R$ and $X$ is equidimensional, then 
\[
(n+1)^{-1}T(\cF) \leq S(\cF) \leq T(\cF)
,\]
where $n=\dim X$ and $c:= \min_i\vol(V_\bullet^i)/\vol(L)$. In particular, $S(\cF) >0$ when $T(\cF)>0$.
\end{enumerate} 
\end{prop}

\begin{proof}
By replacing  $\bk$ with $\overline{\bk}$ and $X$ with $X_{\overline{\bk}}$, we may assume that $\bk$ is algebraically closed. 
Thus the irreducible components of $X$ are geometrically integral and we can and will apply Proposition \ref{p:irredSproperties} on them. 
By Lemmas \ref{l:Vicontainsampleseries}
and \ref{l:formulafiltsonXi},
\[
S_m(\cF) = \sum_{i=1}^s \frac{\dim V_{m}^i}{\dim R_m} S_m(\cF_i)
\quad \text{ and } \quad T_m(\cF) = \max_{i=1,\ldots, s} T_m(\cF)
,\]
where $\cF_i$ is a linearly bounded filtration of $V_\bullet^i$, which is a graded linear series of $L_i$ that contains an ample series. 
By Proposition \ref{p:irredSproperties}.1, 
$T(\cF_i) = \lim_{m \to \infty} T_{mr}(\cF_i)$
and   
$
S(\cF_i)=
\lim_{m \to \infty} S_{mr}(\cF_i)
$.
For each $1\leq i \leq s$, let
\[
c_i:=\limsup_{m \to \infty} \frac{ \dim V_{mr}^i}{\dim R_{mr}}
=
\limsup_{m \to \infty} 
\frac{ \dim V_{mr}^i/ (mr)^{n} }{\dim R_{mr}/(mr)^n}
,\]
which is a limit  by Proposition \ref{p:vollimit}.
Furthermore $c_i= \frac{\vol( V_\bullet^i)}{ \vol(L)}>0$ when $\dim X_i=\dim X=n$ and $c_i=0$ otherwise.
 Therefore 
\begin{equation}\label{eq:Ssum}
 S(\cF) = \lim_{m\to \infty} S_{mr}(\cF) = \sum_{i=1}^s c_i \cdot S(\cF_i),
 \end{equation}
 and $T(\cF) = \lim_{m \to \infty} T_{mr}(\cF)$,
 which proves (1).

By the proof of Lemma \ref{l:formulafiltsonXi}, we have that 
$
\dim \cF^\la V_m = \sum_{i=1}^s \dim \cF_i^\la V_m^i
$.
Note that  $\vol(V_\bullet^{\cF_i,t})$ is a limit for $t\neq T(\cF_i)$ by Proposition \ref{p:vollimit} and Proposition \ref{p:reducedSinvproperties}.2. 
Thus, if $t \neq T(\cF_i)$ for all $i \in \{1,\ldots, s\}$, then $\vol(V_\bullet^{\cF,t})$ is a limit and 
\begin{equation}\label{eq:volVsum}
\vol(V_\bullet^{\cF,t}) = \sum_{\dim (X_i) =n } \vol( V_\bullet^{\cF_i,t}) 
=
\sum_{c_i\neq 0} \vol( V_\bullet^{\cF_i,t})
,
\end{equation}
which proves (2). 
Finally, combining Proposition \ref{p:irredSproperties}.3  with \eqref{eq:Ssum} and \eqref{eq:volVsum} proves (3) and (4).  
\end{proof}

\begin{prop}\label{prop-reducedBJinequality}
If $L$ is ample, then, for any $\varepsilon>0$, there   exists $m_0(\varepsilon)>0$ such that 
\[
S_{mr}(\cF) \leq (1+\varepsilon) S(\cF)
\] 
for all linearly bounded filtrations $\cF$ of $R$ with $\cF^0R=R$ and $m \geq m_0(\varepsilon)$. 
\end{prop}

\begin{proof}
As before, we may assume that $\bk$ is algebraically closed. 
Fix $\varepsilon'>0$ such that $(1+\varepsilon')^2 \leq (1+\varepsilon)$. 
By Lemma \ref{l:Vicontainsampleseries}, each $V_\bullet^i$ contains an ample series.
Therefore  Proposition \ref{p:irredBJinequality} applied to each of the $V_\bullet^i$ implies that there exists $m_1>0$ such that 
\[
S_{mr}(\cG)\leq (1+\varepsilon') S(\cG)
\]
for each $m\geq m_1$,  $i \in \{1,\ldots, s\}$, and   linearly bounded filtration $\cG$ of $V_\bullet^i$ with $\cG^0 V_\bullet^i = V_\bullet^i$.
Now set 
\[
c_i :=\limsup_{m \to \infty}  \frac{\dim V_{mr}^i}{\dim R_{mr} }
,\]
which is a limit by the proof of Proposition \ref{p:reducedSinvproperties}.
Thus, there exists $m_2>0$ such that 
\[
\frac{ \dim V_{mr}^i }{\dim R_{mr}} \leq (1+\varepsilon')c_i
\] 
for all $m \geq m_2$ and $i \in \{1,\ldots, s\}$.

Fix a linearly bounded filtration $\cF$ of $R$ with $\cF^0 R= R$. Let $\cF_i$ be the induced filtrations of $V_\bullet^i$.
If $m\geq m_0 := \max\{m_1,m_2\}$, then 
\[
S_{mr}(\cF) = \sum_{i=1}^s  \frac{\dim V_{mr}^i}{ \dim R_{mr} } S_m(\cF_i)
< (1+\varepsilon')^2
\sum_{i=1}^s 
c_i S(\cF_i)
<
(1+\varepsilon)S(\cF)
,
\]
where the first equality is Lemma \ref{l:formulafiltsonXi}, the second holds by our choice of $m_0$, and the third uses that $S(\cF) = \sum_{i=1}^s c_i S(\cF_i)$ by the proof of Proposition \ref{p:reducedSinvproperties}.
\end{proof}

\subsubsection{Non-normal varieties}
The following lemma can be used to reduce the study of graded linear series on reduced varieties to normal varieties. 
The same result also appeared in \cite[Proposition 2.4.2]{FujCouple}.

\begin{lem}\label{l:pullbackcontainsample}
Assume additionally that $X$ is irreducible. Let $f:Y \to X$  be a proper birational morphism with $Y$  a normal variety. 

If $V_\bullet$ is a graded linear series of $L$  that contains an ample series, then $f^*V_\bullet$ is a graded linear series of $f^*L$ that contains an ample series. 
\end{lem}

In the lemma, $f^*V_\bullet $ denotes the graded linear series of $f^*L$ defined by 
\[
f^* V_m = \im \left( V_m \to H^0(Y, mf^*L)\right)\,,
\]
where the above map is induced by pulling back sections. 

\begin{proof}
Since $V_\bullet$ contains an ample series,  $V_{mr} \neq 0$ for $m\gg 0$ and for some $r>0$ there exists a decomposition $L = A +E$, where $A$ and $E$ are ample and effective $\bQ$-Cartier $\bQ$-divisors such that 
\[
H^0(X,mA)\subset
V_{m}\subset H^0(X,mL)
\]
for $m>0$ sufficiently large and divisible. 
Since $\dim V_{rm} = \dim f^*V_{mr}$, it follows that $f^*V_{mr}\neq 0$ for $m\gg0$. 
To verify the remaining condition in the definition of containing an ample series, note that $Y\to X$ factors through the normalization of $X$.
Thus it suffices to verify the remaining condition in the case when  (i) $Y\to X$ is the normalization morphism and (ii) $X$ is normal. 

We first verify the condition  in case (i).
Let
\[
\mathfrak{c}:= \mathcal{H}om_{\cO_{Y}} ( f_* \cO_{Y}, \cO_{X})\subset \cO_X
\]
denote the conductor ideal.
Since $ \Supp( \cO_{X}/ \mathfrak{c})$ is contained in the non-normal locus of $X$ and $A$ is ample, 
there exists  $s\geq 1$ and an effective Cartier divisor $B \in |s A|$ on $X$ such that $\cO_X(-B) \subset \mathfrak{c}$. 
Note that $L= A'+ E'$, where 
\[
A' = A- \tfrac{1}{2s} B
\quad \text{ and } 
E'  =E+\tfrac{1}{2s}B 
,\]
 is a decomposition into an ample and effective $\bQ$-Cartier $\bQ$-divisors. 
By the definition of $\fc$, there is an inclusion $f_* \cO_Y  \otimes \cO_X(-B)\hookrightarrow \cO_X$, which induces an inclusion 
\[
H^0(Y,mf^*A')
=
f^* 
H^0(X, f_* \cO_{Y} \otimes \cO_{X}(mA') )\\
\subset
f^* H^0(X,  mA' +B)
\subset 
H^0(X,mA)
\]
for $m>0$ sufficiently large and divisible. 
Therefore we have inclusions
\[
H^0(Y,   m f^*A' ) 
\subset 
f^* H^0(X, m A)
\subset 
f^* V_{m}
\subset 
H^0(Y, mf^*L)
\]
for  $m>0$ sufficiently large and divisible. 
Thus $f^*V_\bullet$ contains an ample series in case (i).

We now consider case (ii). 
Since $f^*A$ is big, there exists a decomposition $f^*A = D+F$ into $\bQ$-Cartier $\bQ$-divisors, where $D$ is an ample and $F$ is effective.
Thus
\[
f^*L= D+(F + f^*E)
\]
is a decomposition into ample and effective $\bQ$-Cartier $\bQ$-divisors.
 For $m>0$ sufficiently large and divisible, there are inclusions
\[
H^0(Y, mD )
\subset H^0(Y, f^*A)
\subset 
f^*V_{m}
\subset 
H^0(Y, m f^*L )
,\]
where the second inclusion uses that $H^0(Y, mA) \to H^0(X,mf^*A)$ is an isomorphism as $f_*\cO_Y=\cO_X$ by the normality of $X$.
Thus $f^*V_\bullet$ contains an ample series in case (ii).
\end{proof}

\subsubsection{Approximation results for two filtrations}

\begin{prop}\label{prop-comparetwofiltration}
Assume  additionally that $X$ is geometrically integral.
Let $V_\bullet$  be a graded linear series of $L$ that contains an ample series, let $E$ be a prime divisor over $X$ that is geometrically integral, and $a,b \in \bR$.
Let $\cG$  denote the filtration of $V_\bullet$ defined by  
\[
\cG^\la V_m:= \{ s\in V_m \, \vert\, \ord_{E} (s) \geq a\la-bm\}
.\]
For each $\varepsilon>0$, there exists a $m_0:= m_0(\varepsilon ,\cG)$ such that
 \[
S_{mr}(\cF) - S_{mr}(\cG) \leq (1+\varepsilon) \left( S(\cF)- S(\cG) \right) 
 \] 
 for all $m \geq m_0$ and linearly bounded filtrations $\cF$ of $V_\bullet$ with $\cG \subset \cF$. 
\end{prop}

\begin{proof}
By taking a base change, we may again assume that $\bk$ is algebraically closed. 
First we  reduce to the case when $X$ is normal and $E$ is a prime divisor on $X$. 
Fix a proper birational morphism $f: Y \to X$  such that $Y$ is a normal variety and $E\subset Y$.
The morphism $f$ induces an injective map $f^*: H^0(X,mL)\to H^0(Y, mf^*L)$.
Write $f^*V_{\bullet}$  for the graded linear series of $f^*L$ defined by
\[
f^*V_{m}= \im ( V_m \to H^0(Y ,mf^*L))
.\]
Note that $f^*V_{\bullet}$ contains an ample series by Lemma \ref{l:pullbackcontainsample}.
For an arbitrary linearly bounded filtration $\cF$ of $V_\bullet$ and the filtration $\cG$ in the statement of the proposition, pullback induces filtrations $f^*\cF$ and $f^*\cG$ of $f^*V_\bullet$ satisfying  $S_m(\cF) = S_m(f^*\cF)$ and $S_m(\cG) = S_m(f^*\cG)$.
Therefore, by replacing $X$ and $V_\bullet$ with $Y$ and  $f^*V_\bullet$,  we may reduce to the desired case.

Now we  recall the interpretation of the $S$ and $S_m$ invariants in terms of Okounkov bodies. 
Following the presentation in \cite{Xu-Kbook},
fix a flag
\[
X\supset Y_1 \supset Y_2 \supset \cdots  \supset Y_n = \{x\}
,\]
where each $Y_i$ is an irreducible subvariety of dimension $\dim X-i$ and  $Y_i$  is smooth at $x$. In particular $n=\dim X$.
The flag induces a function 
\[
\nu: H^0(X,mL) \setminus \{0\} \to \bZ^n 
\quad 
\quad 
s\longmapsto \nu(s) =(\nu_1(s),\ldots, \nu_n(s))
\]
with the property that $\nu_1(s) = \ord_{Y_1}(s)$ and 
$
\dim W = \# \nu(W\setminus \{0\})
$
for each subspace $W\subset H^0(X,mL)$; see \cite[Section 1.1.2]{Xu-Kbook}.
For each $m \in r\bZ_{>0}$,  let 
\[
\Gamma_m :=\Gamma_m(V_\bullet):= \nu(V_m \setminus 0)
\] 
and    $\Delta := \Delta(V_\bullet)\subset \bR^n$ denote the closure of the convex hull of $\bigcup_{m \geq 1} \tfrac{1}{mr}\Gamma_{mr}$.
Write 
\[
\rho_m := m^{-n}\sum_{\alpha \in \Gamma_m} \delta_{m^{-1} \alpha}
\]
for the discrete measure on $\bR^n$ and write $\rho$ for the Lebesgue measure on $\Delta$. 

For any linearly bounded filtration $\cF$ of $V_\bullet$ and $t\in \bR$, there is an induced graded linear series $V_\bullet^{\cF,t} := \oplus_{m \in \bN} \cF^{mt} V_m$, which induces an Okounkov body
\[
\Delta(V_\bullet^{\cF,t} ) 
\subset \Delta(V_\bullet) = \Delta
.\] 
We define a function $G^{\cF}: \Delta  \to  \bR $ by 
\[
 G^{\cF}(\alpha) = \sup\{ t\in \bR\, \vert\, \alpha \in \Delta(V_\bullet^{\cF,t})\}.
\]
The function is concave as a consequence of the definition of a filtration.
We have
\begin{equation}\label{eq:S_mNO}
S_m(\cF) \leq  \frac{m^n}{\dim V_m}\int_{\Delta}G^{\cF} \, d \rho_m
\quad \text{ and } \quad
S(\cF) = \frac{1}{ \vol(\Delta)} \int_{\Delta} G^{\cF} \, d \rho
\end{equation}
by \cite[Proposition 3.27 and Lemma 3.29]{Xu-Kbook}.

Now choose a flag as above  with $Y_1= E$.
Since $\nu_1(s)= \ord_E(s)$ for $s\in H^0(X,mL)$, 
\[
\nu(\cG^\la V_m \setminus \{0\})
=  \{ \alpha \in \Gamma_m  \, \vert\,   x_1^*(\alpha) \geq a  \lambda-bm \}
.\]
Thus
\[
\frac{1}{m}\Gamma_m(V_{\bullet}^{\cG,t})=\frac{1}{m}\Gamma_m(V_{\bullet}) \cap (x^*_1(\alpha)\ge a t-b)\,,
\] which implies that
\[
\Delta(V_{\bullet}^{\cG,t})=\{\alpha\in \Delta(V_{\bullet}) \mid x^*_1(\alpha)\ge a  t-b\} \, .
\]
Therefore $G^\cG (x_1,\ldots, x_n) = (x_1+b)/a$ and $\sum_{\la \in \bR} \frac{\lambda}{m}\cdot \dim {\rm gr}_{\cG}^\la V_m = \sum_{\alpha \in m^{-1}\Gamma_m} G^{\cG} (\alpha)$.
 The latter implies that
\begin{equation}\label{eq:SmGNO}
S_m(\cG) 
=
   \frac{m^n}{\dim V_m} 
\int_{\Delta} G^{\cG} d\rho_m
,
\end{equation}
which is stronger than the inequality in \eqref{eq:S_mNO}.

Now  choose  $\varepsilon'>0$ such that $1+\varepsilon' (n+1)\vol(\Delta)^{-1}  < 1+\tfrac{\varepsilon}{2}$. 
By \cite[Lemma 3.32]{Xu-Kbook}, there exists an integer $m_1$ such that 
\[
\int_{\Delta} g d \rho_{mr} \leq \int_{\Delta} g d \rho +\varepsilon'
\]
for all $m\geq m_1$ and  concave functions $g: \Delta \to [0,1] $. 
Note that $H:=G^{\cF}-G^{\cG}$ is concave, since $G^{\cF}$ is concave and $G^{\cG}$ is linear.
In addition,  
$H$ is non-negative by the assumption that $\cG\subset \cF$. 
Therefore for $m\geq m_1$,
\begin{multline*}
 \int_{\Delta} H
 d\rho_{mr}	 
=
\int_{\Delta} H \, d \rho + \varepsilon' \max (H)\\
 \leq
 \int_{\Delta} H \, d \rho + \varepsilon' (n+1) \vol(\Delta)^{-1} \int_{\Delta} H \, d \rho 
 \leq 
\left(1+ \frac{\varepsilon}{2} \right)  \int_{\Delta} H \, d \rho.
 \end{multline*}
The second inequality above uses that if $H$ achieves its maximum at $p \in \Delta$, then
 \[\frac{1}{n+1} \vol(\Delta) \max \{H\}
 = 
 \vol({\rm cone}(\Delta\times \{0\}, (p,H(p)) ))
 \leq \int_{\Delta} H \, d \rho,
 \]
 where the inequality holds as $H$ is concave and non-negative.
 Since $\vol(V_\bullet)=n! \vol(\Delta)$ by \cite[Theorem 1.11]{Xu-Kbook}, there exists an integer $m_2$ such that 
 \[ \left(1+\frac{\varepsilon}{2}\right) \frac{(mr)^n}{\dim V_{mr}}\leq  (1+\varepsilon) \frac{1}{\vol(\Delta)}
 \] for all $m \geq m_2$.
 If $m \geq m_0:= \max\{m_1,m_2\}$, then 
 \begin{align*}
 S_{mr}(\cF)-S_{mr}(\cG)
 &\le \frac{(mr)^n}{\dim V_{mr}}\int_{\Delta} (G^{\cF}-G^{\cG}) \, d \rho_{mr}\\
 &=  \frac{(mr)^n}{\dim V_{mr}} \int_{\Delta} H \, d \rho_{mr}\\
 &\leq \frac{(mr)^n}{\dim V_{mr}}   (1+\varepsilon/2) \int_{\Delta}  H \, d \rho\\
& \leq 
 (1+\varepsilon) \frac{1}{\vol(\Delta)} \int_{\Delta} H\, d \rho \\
& =
(1+\varepsilon)( S(\cF)-S(\cG))
, \end{align*}
 where the first inequality holds by \eqref{eq:S_mNO} and \eqref{eq:SmGNO}, the second and third by our choice of $m_0$, and the fourth by \eqref{eq:S_mNO}.
\end{proof}

\subsection{Stability threshold}\label{ss:stabilitythreshold}
We now recall the definition of the stability threshold \cite{FO-delta,BJ-delta}.
Below, let $(X,\Delta)$ be a klt pair that is projective over a field $\bk$ and $L$ be an ample $\bQ$-line bundle on $X$ with a positive integer $r$ such that $rL$ is a line bundle. 
\medskip

\begin{defn}\hfill
\begin{enumerate}
    \item A valuation $v\in \Val_X$ induces a filtration $\cF_v$ of $R:=R(X,L)$ defined by 
\[
\cF^\la_v R_m := \{s \in R_m \, \vert\, v(s)\geq \la\}
.\]
We set $S(v):= S(\cF_v)$ and $S_{m}(v)= S(\cF_v)$ for any $m\in r\bN^+$.

\item 
For $m\in r\bN^+$, an \emph{$m$-basis type divisor} of $L$ is a $\bQ$-divisor of the form 
\[
D:= \frac{1}{mN_m} \left(\{s_1=0\}+ \cdots + \{s_{N_m}=0\}\right),
\]
where $(s_1,\ldots, s_{N_m})$ is a basis for $H^0(X,mL)$. 

\item An $m$-basis type divisor $D$ is \emph{compatible} with a filtration $\cF$ of $R:=R(X,L)$ if the basis $(s_1,\ldots, s_{N_m})$ is compatible with $\cF$.
An $m$-basis type divisor $D$ is compatible with a valuation $v\in \Val_{X}$ if it is compatible with $\cF_v$.
\end{enumerate}
\end{defn}

\begin{lem}
For $v\in \Val_X$, 
\[
S_m(v) = \max \{ v(D) \, \vert\,  \text{$D$ is an $m$-basis type divisor of $L$ } \}
\]
and the maximum is achieved when $D$ is compatible with $v$.
\end{lem}

\begin{proof}
The statement follows from the proof of \cite[Lemma 3.5]{BJ-delta}.
\end{proof}

\begin{defn}[\cite{FO-delta,BJ-delta}]
The \emph{stability threshold} of $(X,\Delta,L)$ is 
\[
\delta(X,\Delta,L) := \limsup_{m\to \infty}\delta_{mr}(X,\Delta,L),
\]
where 
$\delta_m(X,\Delta,L):= \min \{ \lct(X,\Delta,D) \, \vert\, \text{$D$ is an $m$-basis type divisor of $L$}\}
$.
\end{defn}

\begin{thm}[{\cite[Theorem A]{BJ-delta}}]
We have $\delta(X,\Delta,L) := \lim_{m\to \infty}\delta_{mr}(X,\Delta,L)$.
Furthermore,
\[
\delta(X,\Delta,L)=\inf_{v}\frac{A_{X,\Delta}(v)}{S(v)}
\quad \text{ and }\quad
\delta_m(X,\Delta,L)=\inf_{v}\frac{A_{X,\Delta}(v)}{S_m(v)}
\]
for all $m \in r\bZ_{>0}$,
where the infima run through all $v\in \Val_X^\circ$.
\end{thm}

We say $v\in \Val_X^\circ$ \emph{computes} $\delta(X,\Delta,L)$ if it achieves the above infimum. 

\begin{defn}
The \emph{stability threshold} of a log Fano pair $(X,\Delta)$ is defined as
\[
\delta(X,\Delta):= \delta(X,\Delta; -(K_X+\Delta))
.\]
\end{defn}

By the Fujita--Li valuative criterion and the previous theorem, we have the following characterization of K-semistability, which will be  used as its  definition in this paper.

\begin{thm}[\cite{Fujvalcrit,Livalcrit}]
A log Fano pair $(X,\Delta)$ is K-semistable if and only if $\delta(X,\Delta)\geq 1$.
\end{thm}

\begin{prop}\label{p:deltafieldext}
If $(X_{\bk},\Delta_{\bk})$ is a log Fano pair over a field $\bk$ of characteristic $0$ and $\bk\subset \bk'$ is a field extension, then 
\[
\min\{1,\delta
(X_{\bk}, \Delta_{\bk})\}
=
\min\{1,\delta
(X_{\bk'}, \Delta_{\bk'})\}
\]
\end{prop}

\begin{proof}
If $\bk'= \overline{\bk}$, then the equality holds by \cite[4.63.iii]{Xu-Kbook}, which is a consequence of the main result of \cite{Z-equivariant}.
If  $\bk$ and $\bk'$ are both algebraically closed, the result follows from \cite[Proposition 4.15]{CP-CM-positive} with $T=\Spec(\bk)$.
These two special cases imply the full result by taking algebraic closures $\overline{\bk}$ and $\overline{\bk'}$ satisfying $\overline{\bk} \subset \overline{\bk'}$.
\end{proof}

\section{Relative stability threshold}\label{sec-relative stability}

In this section, we introduce the relative stability threshold  of a polarized pair over a DVR.
In particular, we develop the general theory of this invariant and prove existence results on quasi-monomial and divisorial valuations computing the invariant.

\subsection{Invariants for relative stability}\label{ss-invariants}

\subsubsection{Setup} \label{sss:setup}
Throughout this section, we fix the data of
\[
f:(X,\Delta;L) \to C
,\]
where $C$ is the spectrum of a DVR, $(X,\Delta)$ is an lc pair, $f: X\to C$ is a flat projective morphism with geometric fibers that are  connected, reduced, and of relative dimension $n$, and $L$  is a relatively ample $\bQ$-line bundle on $X$. 
Recall that we write $0$ and $\eta$ for the closed and generic points of $C$. 
We also write $\fm:=\fm_{C,0} \subset \cO_{C,0}$ for the maximal ideal of the DVR.

We denote the
the relative section ring of $L$ by
\[
\cR = \bigoplus_{m \in r\bN} \cR_m : = \bigoplus_{m \in r\bN} H^0(X,mL)
\]
and the section ring of $L_0: =L \vert_{X_0}$ by 
\[
R
:= \bigoplus_{m \in r\bN} R_m
 : = \bigoplus_{m \in r\bN} H^0(X_0,mL_0)
.\]
Note that $\cR_m$ is a free $\cO_C(C)$-module as it is a torsion free module over a DVR. 
We furthermore assume that the natural map 
\[
\cR_m \otimes k(0)\to R_{m}
\]
is an isomorphism for all $m\in r\bN$.

\subsubsection{Filtrations}
The notion of a filtration of $\cR$ is defined in Section \ref{sss:deffilt}.
%
%
%
%
The \emph{restriction} of a filtration $\cF$ of $\cR$ to $X_0$
 is the filtration $\cF\vert_{X_0}$ of $R$ defined by
 \[
\cF^\la\vert_{X_0} R_m := \im (\cF^\la \cR_m \to R_m)
 .\]
We define the $S$ and $T$-invariants of a filtration $\cF$ of $\cR$ via the restriction. First, we set 
\[
S_m(\cF) := S_m(\cF\vert_{X_0}) \quad \text{ and } \quad T_m(\cF):=T_m(\cF\vert_{X_0}),
\]
for $m >0$ divisible by $r$ and then
\[
S(\cF):=\limsup_{m\to\infty} S_{mr}(\cF) \quad \text{ and } \quad T(\cF):=\sup_{m \geq 1} T_{mr}(\cF),\]
which are in $[0,\infty]$.

We are primarily interested in  filtrations $\cF$ of $\cR$ with $T(\cF)<+\infty$.
The latter condition implies that $\cF\vert_{X_0}$ is linearly bounded, which further gives that $S(\cF) = \lim_{m \to \infty} S_{mr}(\cF)$ and $T(\cF) = \lim_{m \to \infty} T_{mr}(\cF)$
 and both values are finite by Proposition \ref{p:reducedSinvproperties}.

\subsubsection{Invariants of valuations}
A valuation $v\in \Val_X$ induces a filtration $\cF_v$ of $\cR$ defined by 
\begin{eqnarray}\label{eq-defn-relativevaluative}
\cF_v^\la \cR_m 
:= 
\{ s\in \cR_m \, \vert\, v(s) \geq \la \}
.
\end{eqnarray}
The relative $S$ and $T$ invariants of the valuation $v$ are defined by
\[
  S_{m}(v) := S_{m}(\cF_v), \quad S(v) := S(\cF_v),\quad T_m(v) := T_m(\cF_v),\quad \text{ and } \quad T(v):=T(\cF_v).
\]
These invariants depend on the choice of $L$, which we omit from the notation.

We begin by establishing some basic properties of these invariants that are analogous to those studied in \cite[Section 3]{BJ-delta} in the absolute case.

\begin{lem}\label{lem-charTusingrelativedivisors}
For $v\in\Val_X$,  
\[
T(v):= \sup \{ v(D)\, \vert\, \text{ $D$  is an effective $\bQ$-Cartier divisor with $D\sim_{\bQ}L$ with $X_0 \not\subset \Supp(D)$}\}.
\]
\end{lem}

\begin{proof}
Observe that  $\cF_v^\la\vert_{X_0} R_m \neq 0$ if and only if there exists  $s\in \cF_v^\la \cR_m$ such that $s$ does not vanish on $X_0$.
Therefore
\[
T_m(v)= \sup \{ v(\tfrac{1}{m}H) \,\vert\, H \in  |mL| \text{ and } X_0 \not\subset \Supp(H) \}
.\]
Using that $T(v) = \sup_m T_{mr}(v)$, we deduce the desired characterization of $T(v)$.
\end{proof}

\begin{lem}\label{lem-delta positive}
If $v\in \Val_X^\circ$, then $T(v)<+\infty$. Furthermore, if $(X,\Delta)$ is klt, then there exists a constant $c>0$ such that $c\cdot T(v)\le A_{X,\Delta}(v)$ for all $v\in \Val_X$.
\end{lem}

Using Proposition \ref{p:reducedSinvproperties}, the lemma implies that $S(v)= \lim_{m \to \infty} S_{mr}(v)$ for all $v\in \Val_X^\circ$. 

\begin{proof}
Fix a projective log resolution $\mu\colon Y\to X$ of $(X,\Delta) $, and write $\Delta_Y$ for the $\bQ$-divisor on $Y$ such that $K_Y+ \Delta_Y= \mu^*(K_X+\Delta)$.
Fix a divisor $A$ on $Y$ that is ample over $C$.

We first claim that there exists a constant $c_1$ such that 
\[
\mult_y(\mu^* D) \leq c_1
\]
for all effective $\bQ$-divisor $ D\sim_\bQ L$  with $X_0 \not\subset \Supp(D)$ and $y\in Y_0$.
Write $\mu^*D=\Gamma_0+\Gamma_1$, where $\Supp(\Gamma_0) \subset Y_0$ and  $\Gamma_1$ does not contain any component of $Y_0$ in its support. By  Lemma \ref{l:mult bdd along X_0} below, the coefficients of $\Gamma_0$ are uniformly bounded
and so $\mult_y \Gamma_0$ is uniformly bounded  for any $y\in Y_0$. Thus  the degree of $\Gamma_1|_F\sim_\bQ (D-\Gamma_0)|_F$ with respect to $A\vert_F$  is uniformly bounded for every component $F$ of $Y_0$. The latter implies that $\mult_y \Gamma_1\le \mult_y (\Gamma_1|_F)$ is also uniformly bounded for all $y\in F$. Thus  the claim holds.

If $(X,\Delta)$ is klt, then $(Y,\Delta_Y)$ is klt and so is $(Y, \Delta_Y^{\geq 0})$. 
Therefore the version of Izumi's theorem  \cite[Proposition 7.45]{Xu-Kbook} (whose proof extends to the case of excellent klt pairs) 
implies that there exists a constant $c_2>0$ dependent only on $(Y,\Delta_Y^{\geq 0})$
 such that 
\[
v (H ) \leq c_2 A_{Y,\Delta^{\geq 0}_Y}( v) \mult_y(H)
\]
for all $v\in \Val_Y$ and $y \in C_Y(v)$ and effective $\bQ$-Cartier  $\bQ$-divisors $H$ on $Y$.
Using that $\Val_X = \Val_Y$, $v(D) = v(\mu^*D)$,  and
$
A_{X,\Delta}(v) = A_{Y,\Delta_Y}(v) \leq A_{Y, \Delta^{\geq 0}_Y}(v)
$, we see that 
\[
v(D) \leq c_1 c_2 A_{X,\Delta}(v) 
\]
for all $v\in \Val_{X}$ and effective $\bQ$-divisors $D\sim_{\bQ}L$.
Using Lemma  \ref{lem-charTusingrelativedivisors}, we conclude that
$T(v) \leq c_1 c_2 A_{X,\Delta}(v)$ for all $v\in \Val_X$. 

It remains to consider the case when $(X,\Delta)$ is not klt and and $v\in \Val_X$ satisfies $A_{X,\Delta}(v)<\infty$.
Since $Y$ is regular, $(Y,0)$ is klt. 
Thus \cite[Proposition 7.45]{Xu-Kbook} applied to $(Y,0)$ implies that there exists a constant $c_3>0$ such that
\[
v (H) \leq c_3 A_{Y,0}( v) \mult_y(H)
\]
for all  $y \in C_Y(v)$ and effective $\bQ$-Cartier  divisors $H$ on $Y$.
As $A_{Y,0}(v) = A_{Y,\Delta_Y}(v) +v(\Delta_Y) < \infty$, the same argument as above shows that $T(v) <\infty$.
\end{proof}

The following result is used in the above proof.

\begin{lem} \label{l:mult bdd along X_0}
Let $Y \to C$ be a projective dominant morphism with connected fibers from a regular scheme $Y$ to a spectrum of a DVR $C$.
For any Cartier divisor $M$ on $Y$, there exists a constant $c>0$ such that 
\[
{\rm mult}_F(\Gamma) \leq c
\]
for all  $\bQ$-divisors $0 \leq \Gamma\sim_\bQ M$ with $Y_0 \not\subset \Supp(\Gamma)$ and irreducible component $F\subset Y_0$.
\end{lem}

\begin{proof}
By cutting by general hyperplane sections and applying Bertini theorem (see \cite[Theorem 10.1]{Lyu-Murayama}), we may reduce to the case when $Y$ is regular of dimension two. Let $F_1,\dots,F_s$ be the irreducible components of $Y_0$.

Write $\Gamma=\Gamma_0+\Gamma_1$, where $\Gamma_0$ is supported on $Y_0$ and $\Gamma_1$ does not contain any $F_i$ in its support. 
Write
$Y_0 = \sum_{i} a_j F_j$ and $\Gamma_0 = \sum_j b_j F_j$.
Let 
\[
d:= \max_j b_j/ a_j
\]
and $i \in\{1,\ldots, s\}$ be an index on which the maximum is achieved.
By our choice of $d$, $\Gamma_0-dY_0 \leq 0$ and is supported on $\cup_{j\neq i} F_i$. 
Since $\max_j \mult_{F_j} \Gamma_0 \leq d \max_j \mult_{F_j}(Y_0)$, 
it suffices to find an upper bound for $d$.

To proceed, we use that
the  intersection form on 
$\sum_{j\neq i} \bR \cdot [F_i]$ is negative definite
by Zariski's Lemma \cite[Lemma 1.74]{Xu-Kbook}.
Thus there exists  a $\bQ$-divisor $M_i$ supported on $\cup_{j\neq i} F_j$ such that $
(M\cdot F_j)=(M_i\cdot F_j)
$
for all $j\neq i$. 
Note that $\Gamma_0-dY_0 \leq 0$ and is supported on $\cup_{j\neq i} F_i$. 
Since $\Gamma_1$ does not contain any $F_j$ in its support,
\[(M_i\cdot F_j)=(M\cdot F_j)=(\Gamma\cdot F_j)\ge (\Gamma_0\cdot F_j) = ((\Gamma_0 - d Y_0)\cdot F_j)\] for all $j\neq i$. 
Thus $D:= M_i - (\Gamma_0-d Y_0)$ satisfies  \[
D\cdot F_j \geq 0 \text{ for $j \neq i$}
\quad \text{ and } \quad 
\Supp(D) \subset \cup_{j\neq i} F_j.
\]
Using that the intersection form is negative definite, we deduce that  
$D\leq 0$. 
Thus $\Gamma_0 - d Y_0 \geq M_i$. 
For $k\in \{1,\ldots, s\}$ such that ${\rm mult}_{F_k}(\Gamma_0) = 0$,
we have 
\[
-d \cdot \mult_{F_k}(Y_0) ={\rm mult}_{F_k}(\Gamma_0- d Y_0) \geq {\rm mult}_{F_k}(M_i) . 
\]
Thus 
\[
d \leq - {\rm mult}_{F_k} (M_i)/ \mult_{F_k}(Y_0) 
\leq - \big(\min_k {\rm mult}_{F_k}(M_i) \big)/ \big(\min_j \mult_{F_j}(Y_0)\big)
.
\]
By taking the maximum of the last value over all $1\leq i\leq s$, we obtain an upper bound for $d$ independent of the choice of $\Gamma$.
\end{proof}

\begin{lem} \label{l:Afinite->linbounded}
If $v\in \Val_X$ with $T(v)<\infty$ and $d$ is a positive integer, 
then
\[
\cF^{C (m+1)}_v \cR_m\subseteq \fm^d \cdot  \cR_m
\]
for all $m\in r\bN$, where $C:=T(v)+v(\fm)d+1$.
\end{lem}

\begin{proof}
If $\la > mT(v)$, then $\cF_{v}^\la \vert_{X_0} =0$ and so
\[
\cF_v^\lambda \cR_m \subseteq \fm\cdot \cR_m
.\]
Since  $v$ is a valuation and $\cO_{C,0}$ is a DVR, 
\[
\cF_v^{\lambda}\cR_m \cap (\fm\cdot \cR_m)=\fm \cdot \cF_v^{\lambda-c} \cR_m
\]
 for all $\lambda\ge c:=v(\fm)$ and $m\in r\bN$.
Combined with the previous inclusion we obtain
\[
\cF_v^{\lambda} \cR_m = \fm \cdot \cF_v^{\lambda-c} \cR_m
\]
for all $\lambda>\max\{mT(v),c\}$ and $m \in r\bN$. 
Thus for any integer $d>0$, repeatedly applying the previous equality gives 
\begin{equation} \label{eq:F_v linearly bounded}
\cF_v^{mT(v)+cd+1}\cR_m=
\fm^d  \cdot \cF_v^{mT(v)+1}\cR_m
\subset
\fm^d \cdot \cR_m
\end{equation}
for all $m\in r\bN$. Thus if we set $C:= T(v) + cd +1$,  then 
$\cF_v^{ C(m+1) } \cR_m
\subset \fm^d \cdot \cR_m
$.
\end{proof}

\begin{lem}\label{l:S>0}
 If $v\in \Val_X^\circ$ is not equal to $c\cdot \ord_{X_0}$ for any $c\geq 0$, then 
$S(v) > 0$.
\end{lem}

\begin{proof}
By our assumption,  $v$ is not the trivial  valuation and $C_X(v)\neq \Supp(X_0)$.
Choose $m$ sufficiently large and divisible  so that $\cO_{X}(mL)\otimes \mathcal{I}_{C_X(v)}$ is globally generated. 
For such $m$, there exists $s\in \cR_m$ such that $\la:=v(s)>0$, but $s$ does not vanish along some component of $X_0$.
By Lemma \ref{lem-charTusingrelativedivisors}, $T(v)>0$. Therefore $S(v) > 0$ by Proposition \ref{p:reducedSinvproperties}.
\end{proof}

For any $v\in \Val_X$ and $t\ge 0$, let
\[
\vol_{X_0}(L;v\ge t):=\limsup_{m\to\infty} \frac{\dim \cF_v^{mt}\vert_{X_0} R_m}{m^n/n!} \, .
\]

\begin{lem}\label{l:Svintresvol}
For $v\in \Val_X^\circ$, we have 
$
S(v) = \frac{1}{\vol(L_0)} \int_0^\infty \vol_{X_0}(L; v\ge t) \, dt 
$.
\end{lem}

\begin{proof}
This follows from Proposition \ref{p:reducedSinvproperties}, Lemma \ref{lem-delta positive}, and the definition of $S(v)$.
\end{proof}

\begin{rem}\label{r:relationtoresvol}
Let $g:Y \to X$ be a proper birational morphism with $Y$ normal and $E\subset Y$ be a prime divisor. 
If $X_0$ is irreducible and $E\neq  \widetilde{X}_0:= g_*^{-1}(X_0)$, then 
\[
 \cF^{mt}_E\vert_{X_0} R_m \cong \im \big( H^0(Y, m(g^* L - tE)) \to H^0(\widetilde{X}_0, m(g^* L-tE)\vert_{\widetilde{X}_0}) \big).
\]
Thus Lemma \ref{l:Svintresvol} implies that 
\[
S(v) = \frac{1}{\vol(L_0)} \int_0^\infty \vol_{Y\vert \widetilde{X}_0}(g^*L - t E ) \, dt
,\]
where $\vol_{Y\vert \widetilde{X}_0}$ denotes the restricted volume as in \cite{ELMNP-resvol}. 
\end{rem}

\begin{lem} \label{l:volume drop}
If $v\in \Val_X^\circ$, then $\vol_{X_0}(L;v\ge t)<\vol(L_0)$ for $t>0$.
\end{lem}

\begin{proof}
For each irreducible component $F\subset X_0$ and $t \in \bR_{>0}$, let $\oF\to F$ denote the normalization morphism and  
\[
V_{\oF,m}^{t} 
:=
\im \big( \cF^{ mt}_v  \vert_{X_0} R_m \to H^0 (\oF,m L \vert_{\oF} ) \big),
\]
which we view as a graded linear series of $L\vert_{\oF}$.
Since the map 
\[
R_m \to \bigoplus_{\oF\subset X_0} H^0(F,mL\vert_{\oF})
\]
is injective,
$\vol_{X_0}( L; v\geq t)  \leq \sum_{F\subset X_0} \vol(V_{\oF,\bullet}^t)$. 
Using that 
\[
\vol(L_0 )= \sum_{F\subset X_0} \vol(L\vert_{\oF}),
\]
it suffices to show that $\vol(V_{\oF,\bullet}^{t})< \vol (L_{\oF})$ for some irreducible component $F\subset X_0$.

To proceed, fix a point $x\in C_X(v)$.
We claim that there exists $c>0$ such that
\[
v(f) \leq c\cdot \ord_x (f) \quad \text{ for all }f\in \cO_{X,x}.
\]
To verify the claim, fix a log resolution $\mu:Y \to X$ of $(X,\Delta)$ and a closed point $y \in \mu^{-1}(x)$.
Since $A_{X,\Delta}(v)< \infty$, $A_{Y}(v)<\infty$.
Therefore  \cite[Proposition 7.45]{Xu-Kbook} implies that there exists $c_1>0$ such that 
$v(g) \leq c_1\cdot  \ord_y(g)$ for all $g\in \cO_{X,x}$.
Furthermore, by Izumi's Inequality  as in \cite[Theorem 7.44]{Xu-Kbook} (in whose proof the klt assumption used in \cite[Proposition 7.43]{Xu-Kbook} can be replaced by \cite[Theorem 4.3]{BFJIzumi}), there exists $c_2>0$ such that 
$\ord_y(h) \leq c_2 \ord_x(h)$ for all $h\in \cO_{X,x}$.
Therefore the claim holds with $c:= c_2/c_1$.

Next, fix an irreducible component $F\subset X_0$ with $x\in F$. 
Fix a divisor $G$ over $\oF$ such that $d:=\ord_G(\fm_x \cdot \cO_{\oF})>0$.
Observe that
\begin{align*}
V_{\oF,m}^{t}
\subset  H^0\left(\oF,mL\vert_{\oF} \otimes (\fa_{mt}(v)\cdot \cO_{\oF})\right) 
&\subset  H^0\left(\oF,mL\vert_{\oF} \otimes (\fm_x^{\lceil mt/c \rceil} \cdot \cO_{\oF}) \right) \\
&\subset H^0\left(\oF,m L\vert_{\oF} \otimes \fa_{ \lceil mt /c\rceil d }(\ord_G)\right).
\end{align*}
Using that $L\vert_{\oF}$ is ample and the previous inclusions, \cite[Lemma 5.13]{BHJ} implies  that
$\vol(V_{\overline{F},\bullet}^{t}) < \vol(L\vert_{\overline{F}})$.
Therefore $\vol_{X_0}(L; v\geq t) < \vol(L_0)$.
\end{proof}

\subsubsection{Basis type divisors}\label{sss:basistypedivisors}
A \emph{relative $m$-basis type divisor} of $L$ is a $\bQ$-divisor of the form 
\[
D: = \frac{1}{ m N_m}  
\left(
\{s_1=0\}+ \cdots +\{ s_{N_m} =0\}
\right),
\]
where $N_m={\rm rank}f_*\cO_X(mL)$ and $( s_1,\ldots, s_{N_m})$ is a basis for $\cR_m$ as a free $\cO_C$-module.
Note that if $D$ is a relative $m$-basis type divisor, then $D\sim_{\bQ} L$ and its restriction $D\vert_{X_0}$ is an $m$-basis type divisor of $L_0$ as in Section \ref{ss:stabilitythreshold}

We say that a relative $m$-basis type divisor $D: = \frac{1}{ m N_m}  
\sum_{i=1}^{N_m}\{s_i=0\}$ is \emph{compatible} with a filtration $\cF$ of $\cR$ if the basis  $(s_1 \vert_{X_0},\ldots, s_{N_m}\vert_{X_0})$ for $R_m$ is compatible with $\cF\vert_{X_0}$ and 
\[
\ord_{\cF}(s_i) = \ord_{\cF\vert_{X_0}}(s_i \vert_{X_0})
\]
for each $1\leq i \leq N_m$. 
We can always construct an $m$-basis type divisor compatible with a given filtration $\cF$ by  fixing an $m$-basis type divisor $D_0$ of $L_0$ compatible with $\cF\vert_{X_0}$ and then lifting the basis to a basis for $\cR_m$ that satisfies the previous equation.
Additionally, we say that $D$ is \emph{compatible} with  $v\in \Val_X$, if $D$ is compatible with $\cF_v$.

\begin{lem}\label{l:S(v)basistype}
For any $v\in \Val_X$ and $m>0$ divisible by $r$, 
\[
S_m(v)= \max \{  v(D)\, \vert\, D \text{ is a relative $m$-basis type divisor of $L$} \}.
\] 
and the max is achieved when $D$ is compatible with $v$. 
\end{lem}

\begin{proof}
Similar to the proof in the absolute case, 
we consider the jumping numbers 
\[a_{m,j} := \inf \{ \la \in \bR_{\geq 0} \,\vert\, \codim (\cF_v\vert_{X_0})^\la R_{m}\geq j \}
,\]
where $1\leq j \leq N_m := {\rm rank} f_* \cO_X(mL)$.
Note that $S_m(v) = S_{m}(\cF_v\vert_{X_0})= \frac{1}{mN_m}\sum_j a_{m,j}$.

Now, consider an $m$-basis type divisor 
$D= \tfrac{1}{mN_m}( \{s_1=0 \} + \cdots \{ s_{N_m}=0\})$ of $L$. 
By reordering the basis, we may assume that $v(s_1) \leq \cdots \leq v(s_{N_m})$. 
Since $(s_1,\ldots, s_{N_m})$ restricts to a basis of $R_m$ and $s_j \in \cF_v^{\la}\cR_m$ for $v(s_j)\geq \la$,
 we have $ a_{m,j} \geq v(s_j) $.
 Thus 
\[
 S_m(v)=  \frac{1}{mN_m}\sum_{j=1}^{N_m} a_{m,j}  \geq  \frac{1}{mN_m}\sum_{j=1}^{N_m} v(s_j) = v(D)
.
\] 

 Next, assume that $D $ is compatible with $v$.
Since $D_0$ is compatible with $\cF_{v}\vert_{X_0}$, after possibly reordering the basis,  
\[
a_{m,j} = \ord_{\cF_v\vert_{X_0}}(s_j\vert_{X_0})
.\]
Thus 
\[
v(s_j) = \ord_{\cF_v}(s_j) = a_{m,j}
\]
which implies that  $v(D) =\tfrac{1}{mN_m} a_{m,j} = S_m(v)$. Therefore $S_m(v)$ equals the desired maximum and the maximum is achieved when $D$ is compatible with $v$.
\end{proof}

The following lemma was essentially observed in \cite{BX-separated}. 
    
\begin{lem}\label{lem-dualfiltration}
Let $f':(X',\Delta',L')\to C$ satisfy the same condition as $f:(X,\Delta,L)\to C$ in Section \ref{sss:setup}.
Assume that $X_0$ and $X'_0$ are irreducible and there is an isomorphism 
\[
(X,L)_{C\setminus 0}\simeq (X',L')_{C\setminus 0}
\]
over $C\setminus 0$.
If $D$ is a relative  $m$-basis type divisor of $L$ compatible with $\ord_{X'_0}$, then its strict transform $D'$ on $X'$ is a relative $m$-basis type divisor of $L'$ compatible with $\ord_{X_0}$.
\end{lem}

\begin{proof}
Let $\cR'_m:= H^0(X',mL')$ and $R'_m:= H^0(X'_0,mL'_0)$.
Via the isomorphism $
(X,L)_{C\setminus 0} \simeq (X',L')_{C\setminus0}$,  may view both $\cR_m$ and $\cR'_m$ as sub-modules of $H^0(X_\eta, mL_\eta)$.
We define  $\bZ$-filtrations $\cG$ and $\cG'$  by setting
\[
\cG^{\la}\cR_m := \cR_m \cap \pi^{\lceil \la\rceil } \cR'_m
\quad \text{ and } \quad 
\cG'^{\la} \cR'_m := \cR'_{m} \cap \pi^{\lceil \la \rceil}\cR_m
,\]
where $\pi \in \cO_C(C)$ is a uniformizer of the DVR.
We write $G: = \cG\vert_{X_0}$ and $G':= \cG\vert_{X'_0}$ for their restrictions to filtrations of $R$ and $R'$.

We now relate these filtration to the valuations $\ord_{X'_0}$ and $\ord_{X_0}$. 
Let $Y $ denote the normalization of the graph of $X\dashrightarrow X'$ with morphisms $g: Y\to X$ and $g':Y \to X'$.
Set 
\[
c:=\tfrac{1}{r} {\rm ord}_{X'_0} ( g^*h)_{rL'}
\quad \text{ and }  
\quad c':= \tfrac{1}{r}{\rm ord}_{X_0}(g'^*h')_{rL}
,\]
 where $h$ and $h'$ are local section of $rL$  and $rL'$ that are regular and non-vanishing at $c_X(\ord_{X'_0})$ and $c_{X'}(\ord_{X_0})$, respectively. 
The subscript $rL'$ denotes that  $g^*h$ is viewed as a rational section of $rL'$ before taking an order of vanishing along $\ord_{X'_0}$ in the first equality. 
We use similar notation moving forward.
We claim that 
\[
\cG^\la \cR_m =\{ s\in \cR_m\, \vert\, \ord_{X'_0}(s) \geq \la -mc\}
\quad
\mathrm{and}
\quad
\cG'^\la \cR'_m =\{ s\in \cR'_m\, \vert\, \ord_{X_0}(s) \geq \la -mc'\}
.\]
Indeed, for $s\in \cR_m$
\[
\ord_{X'_0} (s)_{mL} =  \ord_{X'_0} (g^* s)_{h^*mL}
= 
\ord_{X'_0} (g^*s)_{g^*mL'} - mc
=
\ord_{X_0}(s)_{mL'} -mc
.\]
Therefore the formula holds for $\cG$   and, by symmetry, also holds for $\cG'$.

For $\la \in \bZ$, multiplication by $\pi^{-\la}$ induces an isomorphism 
\[
\phi_{\la}:
\cG^\la R_m = \cR_m \cap \pi^{\la} \cR'_m 
\overset{\cdot \pi^{-\la} }{\longrightarrow}
\cR'_{m}\cap \pi^{-\la}\cR_m
= \cG^{-\la }\cR_m
,
\]
which satisfies 
\[
\phi_{\la}(\cG^{\la+1} \cR_m )= \cG'^{-\la}\cR_m \cap \pi \cR'_m
\quad \text{ and } \quad 
\phi_{\la}(\cG^{\la} \cR_m \cap \pi\cR_m )= \cG'^{-\la+1}\cR'_m
.\]
Using that 
\[
{\rm gr}_{G}R_m = G^{\la }R_m / G^{\la+1}R_m \cong \cG^{\la}\cR_m / (\cG^{\la+1}\cR_m + (\cG^\la \cR_m \cap \pi \cR_m))
\]
and  a similar formula  for ${\rm gr}_{G'} R'_m$, we deduce that $\phi_{\la}$ induces an isomorphism 
\[
{\rm gr}_{G}^{\la}R_m \to {\rm gr}_{G'}^{-\la}R_m
\]
defined by sending $\overline{s}$ to $\overline{\pi^{-\la}s}$.

Now,  fix an $m$-basis type divisor 
$
D:= \tfrac{1}{mN_m} \sum_{i=1}^{N_m}\{s_i=0\}
$
of $L$ that is compatible with $\ord_{X'_0}$.
Since $\cF_{\ord_{X'_0}}$ and $\cG$ differ by a shift, 
$D$ is also compatible with $\cG$. 
Let 
\[
\la_i:= \ord_{\cG} (s_i)\quad \text{ and } \quad 
s'_i := \pi^{-\la_i}s_i \in  \cG'^{-\la} \cR'_m.
\]
We claim that $D':= \frac{1}{mN_m}\sum_i \{s'_i=0\}$ is
an $m$-basis type divisor of $L'$ compatible with $\cG'$. 
Indeed, we know that $\ord_{\cG'}(s'_i) \geq -\la_i$.
Furthermore, since $(s_1,\ldots, s_{N_m})$ is compatible with $\cG$, the basis restricts to a basis for ${\rm gr}_{G} R_m$. Using the isomorphism ${\rm gr}_{G}R_m \cong {\rm gr}_{G'}R'_m$, we see that $(s'_1,\ldots, s'_{N_m})$ restricts to a basis for 
${\rm gr}_{G'}R'_m$ with the restriction of $s'_i$ living in degree $-\la'_i$. 
Thus $\ord_{G'}(s'_i)= -\la_i$ and $(s'_1,\ldots, s'_{N_m})$ is a basis for $\cR'_m$.
Therefore $D'$ is an $m$-basis type divisor of $L$ compatible with $G$. 
Furthermore, since $D$ and $D'$ agree over $X_{C\setminus 0}\cong X'_{C\setminus 0}$ and $X'_0 \not\subset \Supp(D')$, $D'$ is the strict transform of $D$ on $X'$.  
\end{proof}

\subsubsection{Stability thresholds}
We now introduce the relative stability threshold.

\begin{defn}
The \emph{relative stability threshold} of $(X,\Delta;L)\to C$ is defined as
\[
\delta (X,\Delta;L): =
\limsup_{m \to \infty} \delta_{mr}(X,\Delta;L)
,\]
where 
\[
\delta_{m}(X,\Delta;L):=\inf\{ \lct(X,\Delta;D) \, \vert\, \text{ $D$ is a relative $m$-basis type divisor of $L$} \}.
\]
To simplify notation,  we will often write $\delta(L)$ (resp., $\delta_m(L)$) for $\delta(X,\Delta;L)$ 
(resp. $\delta_m(X,\Delta;L)$) when the pair $(X,\Delta)$ is clear from context.
\end{defn}

 The next statement is an analogue of \cite[Theorem C]{BJ-delta} in the relatively setting.

\begin{prop} \label{p:delta=inf A/S}
The limsup in the definition of $\delta(X,\Delta;L)$ is a limit.
In addition,
\[
\delta_{m}(X,\Delta;L)=	\inf_{v} \frac{A_{X,\Delta}(v)}{ S_m(v)}	
\quad \text{ and } \quad 
\delta(X,\Delta;L)=	\inf_{v} \frac{A_{X,\Delta}(v)}{ S(v)}, 
\]
where the above infima can be taken over all valuations $v$ in $\Val_{X}^\circ$  or $\DivVal_X$.
\end{prop}

In the above expression, we use the convention that when $S(v)=0$, then $\tfrac{A_{X,\Delta}(v)}{S(v)} := +\infty$. 
By Lemma \ref{l:S>0}, this only occurs when $X_0$ is irreducible and $v= c \cdot \ord_{X_0}$ for some $c\geq0$. We use a similar convention when $S_m(v)=0$.

We say that $\delta(X,\Delta;L)$ is \emph{computed} by $v\in\Val_{X}^\circ$ if $\delta(X,\Delta;L)=	\frac{A_{X,\Delta}(v)}{ S(v)}$.

\begin{proof}
We follow the argument in \cite{BJ-delta}, which proves the absolute case.
The valuative characterization of the lct implies  that
\[
\delta_m(L) = \inf_{\text{$D$: $m$-basis type }} \left( \inf_v \frac{A_{X,\Delta}(v)}{v(D)} \right),  
\]
where the inner infimum runs through all valuations with finite discrepancy or divisorial valuations. 
Switching the order of the infima and applying Lemma \ref{l:S(v)basistype} implies the formula for $\delta_{m}(L)$.

Next, the  formula for $\delta_{m}(L)$ and fact that $\lim_{m\to \infty} S_{mr}(v) = S(v)$  implies  that
\[
\limsup_{m\to\infty} \delta_{mr}(L) 
\leq \inf_{v} \frac{ A_{X,\Delta}(v)}{S(v)} 
.
\]
By Proposition \ref{prop-reducedBJinequality}, given any $\varepsilon>0$ there exists an integer $ m_0(\varepsilon)$ such that $S_{mr}(\cF) \leq (1+\varepsilon)S(\cF)$ for all linearly bounded filtrations $\cF$ of $R$ and $m \geq m_0(\varepsilon)$.
Using that  $S(v): = S(\cF_v \vert_{X_0})$ and Lemma \ref{lem-delta positive}, we deduce that 
$S_{mr}(v) \leq (1+\varepsilon)S(v)$ for all  $v\in \Val_X^\circ$ and $m \geq m_0(\varepsilon)$.
Therefore
\[
\liminf_{m\to\infty} \delta_{mr}(L)
=
\liminf_{m\to\infty} \inf_v \frac{A_{X,\Delta}(v)}{S_{mr}(v)}
\geq (1+\varepsilon )^{-1} \inf_v \frac{A_{X,\Delta}(v)}{S(v)}
.\]
Combining the previous two equations, we deduce that  $\delta(L)$ is a limit and has the desired valuative formula.
\end{proof}

\begin{prop}\label{p:deltaetalebasechange}
If $C'\to C$ be an \'etale morphism from the spectrum of a DVR, then the base change
\[
(X',\Delta',L') := (X,\Delta,L)\times_C C'
\]
satisfies $\delta(X',\Delta';L') \leq \delta(X,\Delta;L)$. Moreover, equality holds when $(X,\Delta)$ is klt and $C'$ has the same residue field as $\bk$.
\end{prop}

\begin{proof}
Since $(X,\Delta)$ is lc and $X'\to X$ is \'etale, $(X',\Delta')$ is lc by \cite[2.14]{Kol13}.
By flat base change, 
\begin{equation} \label{e:section base change}
H^0(X',mL') \cong H^0(X,mL)\otimes_{A} A'
\end{equation}
for $m \in r\bN$, where 
 $A:= \cO_C(C)$ and $A':= \cO_{C'}(C')$.
Therefore if $D$ is a relative $m$-basis type divisor of $L$, then its pullback $D'$ to $X'$ is a relative $m$-basis type divisor of $L'$ and using \cite[2.14]{Kol13} satisfies
\[
\lct(X',\Delta';D') = \lct(X,\Delta;D)
\]
Therefore $\delta(X',\Delta';L')\leq \delta(X,\Delta;L)$. 

For the equality case, if suffices to show that $\delta_m(X,\Delta;L)\le \delta_m(X',\Delta';L')$ for all $m$ when $C'$ is the completion of $C$ at $0$. By \eqref{e:section base change}, for any positive integer $N$ and any basis $s'_1,\dots,s'_{N_m}$ of $H^0(X',mL')$, we can find a basis $s_1,\dots,s_{N_m}$ of $H^0(X,mL)$ such that $s'_i-s_i\in \fm^N \cdot H^0(X',mL')$. By \cite[2.14]{Kol13}, we can compute $\lct(X,\Delta;D)$ by pulling back to $X'$. Combined with Lemma \ref{lem:lct subadditive}, we see that 
\[
\delta_m(X,\Delta;L)\le \lct(X,\Delta;D) \le \lct(X',\Delta;D') + \frac{\lct(X',\Delta';\fm)}{N},
\]
where $D'$ (resp. $D$) is the corresponding relative $m$-basis type divisor. Letting $N\to\infty$, we get $\delta_m (X,\Delta;L)\le \delta_m(X',\Delta;L')$. 
\end{proof}

The following lemma was used in the previous proof.

\begin{lem} \label{lem:lct subadditive}
If $(X,\Delta)$ is a klt pair, $\fa,\fb\subseteq \cO_X$  nonzero ideals on $X$, and $x\in X$, then
\[
\lct_x(X,\Delta;\fa+\fb)\le \lct_x(X,\Delta;\fa)+\lct_x(X,\Delta;\fb).
\]
\end{lem}

\begin{proof}
This follows from \cite{Takagi-multiplier-ideal}*{Theorem 0.1(2)} as in the proof of \cite{XZ-minimizer-unique}*{Theorem 3.11}. We include the details here for the reader's convenience. 
Let $a=\lct_x(X,\Delta;\fa)$ and $b=\lct_x(X,\Delta;\fb)$.
It suffices to show that the multiplier ideal $\cJ(X,\Delta,(\fa+\fb)^c)$ is nontrivial at the point $x\in X$ for any $c>a+b$. By \cite{Takagi-multiplier-ideal}*{Theorem 0.1(2)} or \cite{JM08} (note that although the result in \emph{loc. cit.} is stated in the $\bQ$-Gorenstein case, the proof goes through in general), we have 
\[
\cJ(X,\Delta,(\fa+\fb)^c) = \sum_{\lambda+\mu = c} \cJ(X,\Delta,\fa^\lambda \fb^\mu).
\]
For any $\lambda,\mu$ with $\lambda+\mu=c>a+b$, we have $\lambda>a$ and hence $\cJ(X,\Delta,\fa^\lambda)\subseteq \fm_x$, or $\mu>b$ and $\cJ(X,\Delta,\fb^\mu)\subseteq \fm_x$. It follows that $\cJ( X,\Delta,\fa^\lambda \fb^\mu)\subseteq \cJ(X,\Delta,\fa^\lambda)\cap \cJ(X,\Delta,\fb^\mu)\subseteq \fm_x$ and therefore $\cJ(X,\Delta,(\fa+\fb)^c)\subseteq \fm_x$ as well.
\end{proof}

\begin{prop}\label{p:deltapos}
If $(X,\Delta)$ is klt, then $\delta(X,\Delta;L) >0$. 
\end{prop}

\begin{proof}
We have that 
\[
\delta(X,\Delta;L) = \inf_{v}\frac{A_{X,\Delta}(v)}{S(v)}
\geq 
\inf_{v} \frac{A_{X,\Delta}(v)}{T(v)}
>0
,\]
where the first equality holds by Proposition \ref{p:delta=inf A/S}, the second holds as $S(v) \leq T(v)$, and the third holds by Lemma \ref{lem-delta positive}.
\end{proof}

A case of primary interest is when  $(X,\Delta+X_0)$ is plt, which implies that  $(X,\Delta) \to C$  is a locally KSBA stable family as in \cite[Definition 2.3]{Kol23}).
In this case, $(X_0,\Delta_0)$ is a normal klt pair by \cite[Proposition 5.51]{KM-book} and adjunction.
The following lemma shows that in this case the relative stability threshold can be used to recover the stability threshold of the special fiber.

\begin{lem} \label{lem-tconvergeto0}
If $(X,\Delta+X_0)$ is plt, then 
\[
\delta(X_0,\Delta_0;L_0) =\delta(X,\Delta+X_0;L) = \lim_{t\to 0+} \delta(X,\Delta+(1-t)X_0;L)
.\]
\end{lem}

\begin{proof}
 By inversion of adjunction, $\lct(X,\Delta +X_0;D)=\lct(X_0,\Delta_0;D_0)$ for any relative $m$-basis type divisor $D$ of $L$. Since any $m$-basis type divisor  $D_0$ of $L_0$ can be lifted to a relative $m$-basis type divisor $D$ of $L$, we see that for 
\[
\delta_m(X,\Delta+X_0;L)=\inf_{D} \lct(X,\Delta+X_0;D)=\inf_{D_0} \lct(X_0,\Delta_0;D_0)=\delta_m(X_0,\Delta_0;L_0) 
\]
when  $m$ is divisible by $r$.
Sending $m\to \infty$ gives $\delta(X,\Delta+X_0;L)=\delta(X_0,\Delta_0;L_0)$, which is the first equality. 

For the second equality, note that
\[
\delta(X,\Delta+X_0;L) \leq
\delta(X,\Delta+(1-t)X_0;L)
\]
for  $t\in[0,1]$.
Therefore it suffices to verify that $\limsup_{t\to0 ^+}\delta(X,\Delta+(1-t)X_0;L)$ is at most $ \delta(X_0,\Delta_0;L_0)$.
For any $\varepsilon>0$,   there exists  $v\in \Val_X^\circ$  with $\frac{A_{X,\Delta+X_0}(v)}{S(v)}<\delta(X,\Delta+X_0;L)+\varepsilon$  by Proposition \ref{p:delta=inf A/S}. 
Now
\begin{multline*}
\limsup_{t\to0^+}\delta(X,\Delta+(1-t)X_0;L) \le
\limsup_{t\to0^+} \frac{A_{X,\Delta+(1-t_0)X_0}(v)}{S(v)} \\= \frac{A_{X,\Delta+X_0}(v)}{S(v)} <\delta(X,\Delta+X_0;L)+\varepsilon\, 
\end{multline*}
where the first inequality is by Proposition \ref{p:delta=inf A/S}.
Therefore the second equality holds.
\end{proof}

\subsection{Minimizer}\label{ss-minimizer} 
In this section, we show that the relative stability threshold is computed by a quasi-monomial valuation. We keep the notation and assumptions from Setup \ref{sss:setup}. We also assume that the residue field $\bk$ is algebraically closed and uncountable.

\begin{thm} \label{t:delta computed by qm valuation base complete}
If $(X,\Delta)$ is klt and $C$ is the spectrum of a complete DVR, then there exists a quasi-monomial valuation $v\in \Val_X$ that computes $\delta(X,\Delta;L)$.
\end{thm}

When the DVR $\cO_{C,0}$ is not complete, the relative stability threshold may not be computed by any valuation (see Remark \ref{r:no val compute delta}), but under mild assumptions such valuations still exist.

\begin{thm} \label{t:delta computed by qm valuation}
If $(X,\Delta)$ is klt and $\delta(X,\Delta;L)<\delta(X_K,\Delta_K;L_K)$, where $K=\mathrm{Frac}(\widehat{\cO_{C,0})}$ and \[
(X_K,\Delta_K;L_K) = (X,\Delta;L)\times_C \Spec~K
,\]
then there exists a quasi-monomial valuation $v\in \Val_X$ that computes $\delta(X,\Delta;L)$.
Moreover, the center of any such valuation $v$ is contained in $X_0$.
\end{thm}

As in \cite{BJ-delta}, the proof of Theorem \ref{t:delta computed by qm valuation base complete} relies heavily on generic limit constructions. We first discuss some general properties of this construction.

\subsubsection{Generic limit}\label{ss-genericlimit}
In this Section \ref{ss-genericlimit}, we assume that $C$ is the spectrum of a complete DVR. By Cohen structure theorem, we have $\cO_{C,0}\cong \bk[\![t]\!]$. We fix an isomorphism as such.


Let $\cF_k$ ($k=1,2,\cdots$) be a sequence of $\bN$-filtrations of $\cR$. By restriction, they induce $\bN$-filtrations (still denoted by $\cF_k$) of 
\[
\widetilde{R}_m:=\cR/(\fm^m \cR+ \cR_{>m})\cong \oplus_{l\le m} \cR_l/\fm^m \cR_l
\]
for all  $m \in r\bN$. More precisely,
\[
\cF_k^{\lambda}\widetilde{R}_m={\rm Image} (\oplus_{l\le m} \cF_k^{\lambda}\cR_l\to \oplus_{l\le m}\cR_l\to \widetilde{R}_m) \, .
\]
For each $m \in r\bN$ and $C\in \bN$, the $\bN$-filtrations of $\widetilde{R}_m$ with $\cF^{C}\widetilde{R}_m = 0$ are parameterized by a scheme $H_{m,C}$ that is  a closed subscheme of a nested Hilbert scheme of finite type over $\bk$, since each filtration $\cF$ yields a sequence of ideals
\[
0=\cF^{C}\widetilde{R}_m\subseteq \cF^{C-1}\widetilde{R}_m\subseteq\cdots \subseteq \cF^{0}\widetilde{R}_m=\widetilde{R}_m \,,
\]
and $\widetilde{R}_m$ is a finite $\bk$-algebra. If $C'\ge C$ then $H_{m,C}\subseteq H_{m,C'}$ is a union of connected components. Thus $\bN$-filtrations of $\widetilde{R}_m$ can be parameterized by a scheme $H_m:=\cup_{C} H_{m,C}$ whose connected components are of finite type over $\bk$.

Assume now that the filtrations $\cF_k$ of $\widetilde{R}_m$ are uniformly bounded, in the sense that for each $m\in r\bN$ there exists some constant $C_m$ such that $\cF^{C_m}_k \tR_m = 0$ for all $k$. Then the filtrations $\cF_k$ correspond to points $z_{k,m}$ of $H_{m,C_m}$, and thus it makes sense to take their Zariski closure in $H_m$ which gives a scheme of finite type. 

We have natural truncation maps $H_{m+r}\to H_m$ induced by the restriction of filtrations. By the same argument as \cite{BJ-delta}*{Proposition 6.7}, there exists a sequence of   infinite subsets
\[
\mathbb{N}\supset I_1\supset I_2\supset \cdots 
\]
such that $Z_{m}:=\overline{\bigcup_{k\in I_{m/r}}\{z_{k,m}\}} \subset H_{m}$ satisfies the condition that  any  closed subset $Y\subsetneq Z_m$ contains finitely many $z_{k,m}$. 
In particular, there exists dominant natural morphisms $Z_{m+r}\to Z_m$ as well. By the uncountability assumption of the field $\bk$ and the same argument as  \cite{BJ-delta}*{Lemma 6.6},  for any sequence of non-empty open sets $U_m\subset Z_m$, we can pick   a compatible sequence of closed points  $z_m\in U_m$ such that $z_{m+r}$ maps to $z_m$ under the truncation map. 
Since each point $z_m$ corresponds to a filtration of $\widetilde{R}_m$, and since $\cO_{C,0}$ is complete, the entire sequence $\{z_m\}_{m\in r\bN}$ determines a filtration $\cF$ of $\cR$. 
We call $\cF$ a \emph{generic limit} of the $\cF_k$.

We shall analyze the behavior of several invariants of filtrations under the generic limit construction. 
Let $\fb_{m,\la}(\cF)$ be the base ideal of $\cF^\la \cR_m$, i.e.
\[
\fb_{m,\la}(\cF):= {\rm Image}(\cF^\la \cR_m \otimes \cO_{X} \to \cO_{X}(mL) )\otimes  \cO_{X}(-mL) \subseteq  \cO_{X} \, .
\]
Then $\fb_{m,\la}(\cF)\subseteq \fb_{m+m_0,\la}(\cF)$ for any integer $m_0 \in r\bN$ such that $m_0 L$ is base point free. As $L$ is ample, this implies that for any fixed $\la$ the sequence $\fb_{m,\la}(\cF)$ of ideals stabilizes when $m \in r\bN$ is sufficiently large. We can thus set 
\begin{equation} \label{eq:stable base ideal}
\fb_\la(\cF):=\fb_{mr,\la}(\cF)
\end{equation}
for $m\gg 0$. Then $\{\fb_\la(\cF)\}_{\la \in \bN}$ form a multiplicative graded sequence of ideals on $X$. The first property we need is that the log canonical threshold of base ideals does not increase under generic limit.

\begin{lem} \label{l:lct under generic limit}
With the above setup, we have
\[
\lct(X,\Delta;\fbb(\cF))\le \limsup_{k\to\infty} \lct(X,\Delta;\fbb(\cF_k)).
\]
\end{lem}

\begin{proof}
Fix integers $m\geq N \geq 0$ both divisible by $r$ and  $\la \in \bR$.
By abuse of notation, we denote the ideal $\fm_{C,0} \cdot \cO_{X}$ also by $\fm$.
From the generic limit construction above, there is an ideal $\mathcal{B}_{m,\la} \subset \cO_{X \times H_m}$, which we view as a family of ideals on $X$ parametrized by $H_m$, 
with the property that
$\cB_{m,\la}\vert_{X\times \{z_{k,N}\} }$ and $\cB_{m,\la}\vert_{X\times\{z_N\}}$ agree with $\fb_{m,\la}(\cF_k) + \fm^N$
and $\fb_{m,\la}(\cF)+\fm^N$, respectively.
The function 
\[
H_N\ni h \mapsto \lct(X,\Delta; \cB\vert_{X\times \{h\}})
\]
is  constructible and lower semicontinuous by inversion of adjunction. 
Therefore
\begin{multline*}
\lct(\fb_{m,\la}(\cF))
\leq 
\lct(\fb_{m,\la}(\cF)+\fm^N)\\
\leq \lim_{I_{N/r}\ni k\to\infty} \lct(\fb_{m,\la}(\cF_k)+\fm^N)
\leq 
\limsup_{k\to\infty} \lct(\fb_{m,\la}(\cF_k) +\fm^N),
\end{multline*}
where the second inequality uses that $Z_{N}$ is irreducible, contains $z_N$, and is the closure of  $\{z_{N,k} \, \vert\, k\in I_{N/r} \}$.
By Lemma \ref{lem:lct subadditive}, we see that 
\[
\lct(\fb_{m,\la}(\cF_k)+\fm^N)\le \lct(\fb_{m,\la}(\cF_k))+\frac{\lct(\fm)}{N}
.\]
Combining the previous two inequalities and letting $N\to \infty$, we deduce that 
\[
\lct(\fb_{m,\la}(\cF))
\leq 
\limsup_{k\to \infty} \lct(\fb_{m,\la}(\cF_k))
.
\]
For $m>0$ sufficiently large and divisible, we obtain
\[
\lct(\fb_{\la}(\cF))
=
\lct(\fb_{m,\la}(\cF))
\leq 
\limsup_{k\to \infty} \lct(\fb_{m,\la}(\cF_k))
\leq 
\limsup_{k\to \infty} \lct(\fb_{\la}(\cF_k))
,\]
where the last equality uses that $\fb_{m,\la}(\cF_k) \subset \fb_{\la}(\cF_k)$. 
By multiplying by $\la$ and sending $\la\to \infty$, we deduce the desired result.
\end{proof}




Next, we analyze the behavior of the $S$-invariant under the generic limit. Before we do this, we need to generalize a uniform approximation result from \cite{BJ-delta} to our relative setting.

For any filtration $\cF$ of $\cR$, positive numbers $m$ and $\ell$ with $m$ divisible by $r$,  and $\la \in \bR$, similar to \cite[pg. 32]{BJ-delta}, we define a graded linear series of $mL_0$ by 
\begin{equation} \label{eq:tF}
V^{\lambda}_{m,\ell}:=H^0(X_0,m\ell L_0\otimes \overline{\fb_{m,\lambda m}(\cF)|_{X_0}^\ell}) \, ,
\end{equation}
where the notation $\overline{\fb}$ denotes the integral closure of the ideal $\fb$.
We define an invariant
\[
\tS_m(\cF)=\frac{1}{\vol(mL_0)}\int_0^{T(\cF)}\vol(V^{\lambda}_{m,\bullet})\,  {\rm d}\lambda\, . 
\]As before, we denote by $\tS_m(v):=\tS_m(\cF_v)$ for any valuation $v$.


\begin{lem}[{\it cf.} \cite{BJ-delta}*{Theorem 5.3}] \label{l:S uniform convergence}
If $(X,\Delta)$ is klt, then 
there exists a constant $C>0$ such that
\begin{equation} \label{eq:uniform S estimate}
    0\le S(v)-\tS_m(v)\le C\cdot\frac{A_{X,\Delta}(v)}{m}
\end{equation}
for all $m> 0$ divisible by $r$ with $h^0(X,mL)\neq 0$ and  $v\in \Val_X^\circ$. 
\end{lem}

\begin{proof}
We first consider the case when $\Delta$ does not contain any component of $X_0$. We follow the same multiplier ideal argument as in \cite{BJ-delta}*{Section 5}. For any graded linear series $V_\bullet$ of $L$ and any $c\ge 0$, we denote by $\cJ(c\cdot \norm{V_\bullet})$ the corresponding asymptotic multiplier ideals with respect to $(X,\Delta)$ (see e.g. \cite[Definition 11.1.2]{Laz-positivity-2}). By \cite{Tak-subadditive-pairs}*{Theorem} and the basic property $\fa\cdot \cJ(X,\Delta,\fb)\subseteq \cJ(X,\Delta,\fa \fb)$ of multiplier ideals on a klt pair, we know that there exists a nonzero ideal $\fb\subseteq \cO_X$ depending only on the pair $(X,\Delta)$ such that $\fb_0:=\fb|_{X_0} $ is not zero along any component of $X_0$ and 
\begin{equation} \label{eq:multiplier ideal uniform inclusion}
    \fb^{\ell} \cdot \cJ(m\ell \cdot \norm{V_\bullet}) \subseteq \cJ(m \cdot \norm{V_\bullet})^\ell
\end{equation}
for all graded linear series $V_\bullet$. In fact by {\it loc. cit.} we can take $\fb=\mathrm{Jac}_X \cO_X(-s\Delta)$ where $\mathrm{Jac}_X$ is the Jacobian ideal and $s>0$ is an integer such that $s(K_X+\Delta)$ is Cartier, and the assumption that ${\rm Supp}(\Delta)$ does not contain any component of $X_0$ is used to show that $\fb_0$ is not zero along any component of $X_0$. By \cite{BJ-delta}*{Corollary 5.10}, there also exists an integer $a=a(L)$  divisible by $r$ such that for any graded linear series $V_\bullet$ of $L$, the sheaf 
\[
(a+m)L\otimes \cJ(m\cdot \norm{V_\bullet})
\]
is globally generated for all integer $m\in r\bN$ (note that our situation is slightly different from \emph{loc. cit. } as $X$ is only projective over a DVR and we have an additional boundary divisor, but the same argument in \emph{loc. cit.}., which essentially relies on Nadel vanishing and Castelnuovo-Mumford regularity, goes through in our setting). Applying this to the graded linear series $V^\lambda_\bullet$ with $V^\lambda_m:=\cF_v^{\lambda m} \cR_m$, and noting that (\cite{BJ-delta}*{Proposition 5.11})
\[
\cJ(m\cdot \norm{V^\lambda_\bullet})\subseteq \fa_{\lambda m-A_{X,\Delta}(v)}(v),
\] 
we see that $\cJ(m\cdot \norm{V^\lambda_\bullet})$ is contained in the base ideal $\fb_{m+a,\lambda m-A}:=\fb_{m+a,\lambda m-A}(\cF_v)$ of $\cF_v^{\lambda m-A} \cR_{m+a}$, where $A:=A_{X,\Delta}(v)$. In other words,
\[
\cJ(m\cdot \norm{V^\lambda_\bullet})\subseteq \fb_{m+a,\lambda m-A}.
\]
Moreover, we can choose $a$ so that $H^0(X_0,aL_0\otimes \fb_0)$ contains an element $s$ which does not vanish along any component of $X_0$. There is also an obvious inclusion $\fb_{m,\lambda m}(\cF_v)\subseteq \cJ(m\cdot \norm{V^\lambda_\bullet})$. Combined with \eqref{eq:multiplier ideal uniform inclusion}, we deduce a chain of inclusions (where we write $\cJ_{m,\lambda}:=\cJ(m\cdot \norm{V^\lambda_\bullet})|_{X_0}$ for ease of notation)
\begin{align*}
    \cF_v^{\lambda m\ell} R_{m\ell}\hookrightarrow H^0(m\ell L_0\otimes \cJ_{m\ell,\lambda}) & \stackrel{\cdot s^\ell}{\rightarrow} H^0((m+a)\ell L_0\otimes (\fb^\ell_0\cdot \cJ_{m\ell,\lambda})) \\
    & \hookrightarrow H^0((m+a)\ell L_0 \otimes \cJ_{m,\lambda}^\ell) \\
    & \hookrightarrow H^0((m+a)\ell L_0 \otimes \fb_{m+a,\lambda m-A}|_{X_0}^\ell) \\
    & \hookrightarrow H^0((m+a)\ell L_0 \otimes \overline{\fb_{m+a,\lambda m-A}|_{X_0}^\ell}) = V^{\lambda'}_{m+a,\ell} \, , 
\end{align*}
where $\lambda'=\frac{\lambda m-A}{m+a}$.
The rest of the proof of \eqref{eq:uniform S estimate} then proceeds verbatim as in \cite{BJ-delta}*{Propositions 5.13 and 5.15} using Lemmas \ref{lem-delta positive} and \ref{l:Svintresvol}.

In general, $\Delta$ may contain some components of $X_0$, but for some sufficiently large integer $a$ there exists some effective  $\bQ$-divisor $\Delta'$ such that ${\rm Supp}(\Delta')$ does not contain any component of $X_0$, $aL+\Delta\sim_{\bQ}\Delta'$ and $(X,\Delta')$ is klt. Moreover, since $(X,\Delta)$ is klt, we may also choose some small constant $0<c_0\ll 1$ such that $c_0\cdot A_{X,\Delta'}(v)\le A_{X,\Delta}(v)$ for all $v\in \Val_X$; indeed this is equivalent to the condition that $(X,\Delta+\frac{c_0}{1-c_0}(\Delta-\Delta'))$ is a klt sub-pair. Thus the general case of \eqref{eq:uniform S estimate} follows from the special case treated above.
\end{proof}

For a filtration $\cF$, we denote by $\cF_{\bN}$  the associated $\bN$-filtration defined by  $\cF_{\bN}^{\lambda}R_m=\cF^{\lceil \lambda \rceil}R_m$.
Below, we write $\cF_k:=(\cF_{v_k})_{\bN}$.

\begin{lem} \label{l:S under generic limit}
Let $A,T>0$ be two constants. 
If $(v_k)_{k\in \bN}$  is a sequence of valuations on $X$  such that $A_{X,\Delta}(v_k)<A$ and $T(v_k)<T$ for all $k$, then the filtrations $\cF_{k}$ admit a generic limit $\cF$ satisfying
\[
S(\cF)\ge \liminf_{k\to\infty} S(v_k).
\]
\end{lem}

\begin{proof}
Since $(X,\Delta)$ is klt, there exists some constant $\varepsilon>0$ such that $(X,\Delta+\varepsilon X_0)$ is still klt. It follows that $A_{X,\Delta}(v_k)\ge \varepsilon v_k(\fm)$ and hence $v_k(\fm)<\frac{A}{\varepsilon}$. By Lemma \ref{l:Afinite->linbounded}, we see that the filtrations $\cF_{v_k}$ of $\tR_m$ are uniformly bounded. 

For each $m$, there exists a nonempty open set $U_m\subset Z_m$ such that
the function 
\[
U_m \ni z\to \tS_m(\cF_z)
\]is a constant (see \cite{BJ-delta}*{Proposition 6.3}).
By the discussion at the beginning of this subsection, there exists a generic limit $\cF$  of the $\cF_{k}$ given by a compatible set of points $z_m\in U_m$.
Note that $\tS_m(\cF)= \tS_m(\cF_k) $ for $k \in I_{m/r}$ with $z_{m,k}\in U_m$.
Since $Z_{m}\setminus U_m$ contains finitely many $z_{m,k} $ with $k\in I_{m/r}$, we see that 
\[
\tS_{m}(\cF) = \lim_{I_{m/r} \ni k \to \infty } \tS_{m}(\cF_k) 
\geq 
\liminf_{k\to \infty} \tS_m(\cF_k)
.
\]
Now we compute that 
\[
S(\cF)=\lim_{m\to \infty}\tS_m(\cF)
\geq\lim_{m\to\infty}  
\liminf_{k\to\infty} \tS_m(\cF_{k})  
= \lim_{m\to\infty} 
\liminf_{k\to\infty} \tS_m(\cF_{v_k}) 
=
\liminf_{k\to\infty}S(v_k)
,
\]
where the first relation holds by  the discussion in \cite{BJ-delta}*{Section 5.1},
the second relation by the previous inequality,
the third by the inequality $0<\tS_m(v_k)-\tS_m(\cF_k) \le \frac{1}{m}$  which holds by the proof of  \cite[Lemma 6.9]{BJ-delta}, and the final by  Lemma \ref{l:S uniform convergence}.
\end{proof}

\subsubsection{Existence of minimizer}

We  now return to prove the existence of quasi-monomial valuations that compute the relative stability threshold $\delta(L)$.

\begin{proof}[Proof of Theorem \ref{t:delta computed by qm valuation base complete}]
By Proposition \ref{p:deltapos}, $\delta(L)>0$. 
By Proposition \ref{p:delta=inf A/S}, there exists 
 a sequence of valuations $(v_k)_{k \in \bN}$ on $X$ such that $\lim_{k\to \infty} \frac{A_{X,\Delta}(v_k)}{S(v_k)} = \delta(L)$.
 By scaling the valuations, we may assume that  $A_{X,\Delta}(v_k)=1$ for all $k\in \bN$. Then we have 
\[
\lct(\fbb(\cF_{v_k}))
\le 
\lct(\fa_\bullet(v_k))
\le \frac{A_{X,\Delta}(v_k)}{v_k(\fa_\bullet(v_k))} \le 1
\]
(where $\fbb(\cF_{v_k})$ is the sequence of stable base ideals defined in \eqref{eq:stable base ideal}) and
\[
S:=\lim_{k\to\infty} S(v_k) = \delta(L)^{-1}\in (0,+\infty).
\]
In particular, the $S$-invariants $S(v_k)$ are uniformly bounded from above. 
Thus, Proposition \ref{p:reducedSinvproperties} implies that the $T(v_k)$ are also uniformly bounded from above. 
Since $C$ is the spectrum of a complete DVR, the filtrations $\cF_k:=(\cF_{v_k})_{\bN}$ admit a generic limit $\cF$, which satisfies $S(\cF)\ge S$ by Lemma \ref{l:S under generic limit}. 
By Lemma \ref{l:lct under generic limit}, we also have $\lct(\fbb(\cF))\le 1$. 

By \cite[Theorem 1.1]{Xu-quasimonomial} (and  its extension to the case of excellent schemes  relying on \cite{Lyu-Murayama}), 
there exists a quasi-monomial valuation $v\in \Val_{X}$ that computes $\lct(\fbb(\cF))$. After rescaling, we may assume that $A_{X,\Delta}(v)=1$. As $\lct(\fbb(\cF))\le 1$, this gives $v(\fbb(\cF))\ge 1$ and hence $S(v)\ge S(\cF)\ge S$ (see \cite[Corollary 3.21]{BJ-delta}). It follows that 
\[\frac{A_{X,\Delta}(v)}{S(v)}\le S^{-1} = \delta(L)
\]
and hence $v$ computes $\delta(L)$.
\end{proof}

\begin{lem} \label{l:center in X0}
Assume that $\delta(X,\Delta;L)<\delta(X_\eta,\Delta_\eta;L_\eta)$. Then the center of any valuation $v\in \Val_X$ that computes $\delta(X,\Delta;L)$ is contained in $X_0$.
\end{lem}

\begin{proof}
Suppose that $C_X(v)\not\subset X_0$, then it induces a valuation on the generic fiber $X_\eta$ with $S_{X_\eta,\Delta_\eta}(v)=S(v)$. It follows that 
\[
\delta(X_\eta,\Delta_\eta;L_\eta)\le \frac{A_{X_\eta,\Delta_\eta}(v)}{S_{X_\eta,\Delta_\eta}(v)} = \frac{A_{X,\Delta}(v)}{S(v)} = \delta(X,\Delta;L),
\]
which contradicts our assumption that $\delta(X,\Delta;L)<\delta(X_\eta,\Delta_\eta;L_\eta)$.
\end{proof}

\begin{proof}[Proof of Theorem \ref{t:delta computed by qm valuation}]
Let $\hat{C}=\Spec(\widehat{\cO_{C,0}})$, $(\hX,\hat{\Delta})=(X,\Delta)\times_C \hat{C}$, and let $\hL$ be the pullback of $L$ to $\hX$. By Theorem \ref{t:delta computed by qm valuation base complete}, there exists some quasi-monomial valuation $\hv\in \Val_{\hX}$ that computes $\delta(\hX,\hat{\Delta};\hL)$.  
By Lemma \ref{l:center in X0}, the center of $\hv$ is contained in $\hat{X}_0$, hence by \cite[Lemma 3.10]{JM-val-ideal-seq} (and its proof) it restricts to a quasi-monomial valuation $v$ on $X$ such that $A_{X,\Delta}(v)=A_{\hX,\hat{\Delta}}(\hv)$ and $S(v)=S(\hv)$ (in fact $v$ and $\hv$ induce the same filtration of $R$). By Proposition \ref{p:deltaetalebasechange}, we have $\delta(X,\Delta;L)=\delta(\hX,\hat{\Delta};\hL)$. Therefore $v$ computes $\delta(X,\Delta;L)$. Since $\delta(X_K,\Delta_K;L_K)\le \delta(X_\eta,\Delta_\eta;L_\eta)$, the center of any valuation $v\in \Val_X$ that computes $\delta(X,\Delta;L)$ is contained in $X_0$ by Lemma \ref{l:center in X0}.
\end{proof}

\begin{rem} \label{r:no val compute delta}
Without the hypothesis in Theorems \ref{t:delta computed by qm valuation base complete} and \ref{t:delta computed by qm valuation}, the relative stability threshold may not be computed by any valuation. For example, consider a smooth projective morphism of relative dimension one $X\to C$ (such as an elliptic fibration). Then $\delta(L)=\delta(L_\eta)=\frac{1}{\deg L_\eta}$ and every valuation computing $\delta(L)$ comes from a section of the family $X\to C$. In general, such a section may not exist, unless $C$ is the spectrum of a complete DVR.
\end{rem}

\subsection{\texorpdfstring{$S$}{S}-invariant revisited} \label{ss:revisit S-inv}

To prove further properties of the minimizer, we need to give a different interpretation of the $S$-invariant by working with a thickening of the central fiber $X_0$. To this end, fix some positive real number $c$. For any $m\in r\bN$ we set
\[
V_{c,m}:=\cR_m/\fm^{\lceil cm \rceil}\cR_m,
.\]
which is an Artinian $\cO_C$-module.
As before, any valuation $v$ on $X$ induces a filtration  of $V_{c,m}$ denoted also by $\cF_v$, and we define
\[
S_{c,m}(v)
= 
\frac{1}{m\dim V_{c,m}}
    \sum_{\la \in \bR} \la\cdot  \mathrm{length} ({\rm gr}_{\cF_v}^\la V_{c,m})
\]

\begin{lem} \label{lem:S_c,m}
We have $S_{c,m}(v)=S_m(v)+\frac{\lceil cm \rceil-1}{2m} v(\fm)$. In particular, 
\[\lim_{m\to \infty} S_{c,mr}(v)= S(v)+\frac{c}{2}\cdot v(\fm).\]
\end{lem}

\begin{proof}
Let $\pi\in \fm$ be a uniformizer of $\cO_{C,0}$. Consider the $\fm$-adic filtration $\cF_{\fm}$ of $\cR$ given by $\cF_{\fm}^k \cR_m :=\fm^k \cR_m = \pi^k \cR_m$. It also induces a filtration on each $V_{c,m}$. By \cite{AZ-K-adjunction}*{Lemma 3.1}, the average jumping number can be computed from a basis of $V_{c,m}$ that is compatible with both filtrations $\cF_{\fm}$ and $\cF_v$. Note that $\cF_{\fm}^k V_{m,c} = 0$ for all $k\ge \lceil cm \rceil$. When $0\le k<\lceil cm \rceil$, multiplication by $\pi^k$ induces a filtered isomorphism up to a degree shift
\[
R_m\to \pi^k\cR_m/\pi^{k+1}\cR_m\cong \cF_{\fm}^k V_{c,m}/\cF_{\fm}^{k+1} V_{c,m},
\]
where the filtrations are all induced by $\cF_v$.
The compatible basis we get therefore consists of elements of the form $\pi^i f_{i,j}$, where $0\le i<\lceil cm \rceil, 1\le j\le \dim R_m$, such that for each fixed $i$, the elements $f_{i,j}$ restrict to a basis of $R_m$ (through the above isomorphism) that is compatible with $\cF_v$. The average $v$-value of the $f_{i,j}$ is $m\cdot S_m(v)$ by definition. The average $\fm$-adic value of the compatible basis is $\frac{\lceil cm \rceil-1}{2}$, thus their contribution to $S_{c,m}(v)$ is given by $\frac{\lceil cm \rceil-1}{2m} v(\fm)$.
\end{proof}

Next, suppose that $\Delta\ge \frac{c\delta}{2}X_0$ where $\delta=\delta(L)$ (in the applications, this often means that $X_0\subseteq \Supp(\Delta)$ and we choose some sufficiently small $c$). In this case we can use the above calculation to reinterpret the relative stability threshold as follows. We call any lift $s_1,\dots,s_{N_{c,m}}\in \cR_m$, where $N_{c,m}=\dim V_{c,m}$, of a basis of $V_{c,m}$ a level-$c$ $m$-basis. The corresponding divisor
\[
D=\frac{1}{mN_{c,m}} \sum_{i=1}^{N_{c,m}} \{s_i=0\}
\]
is called a level-$c$ $m$-basis type divisor. Denote by 
\begin{eqnarray}\label{defn-delta cm}
\delta_{c,m}(L)=\min_D\lct(X,\Delta-\frac{c\delta}{2}X_0;D)
\end{eqnarray} as $D$ varies among the level-$c$ $m$-basis type divisors.

\begin{lem} \label{l:delta_c,m}
 We have 
\[
\lim_{m\to\infty} \delta_{c,m}(L) = \inf_v \frac{A_{X,\Delta}(v)+\frac{c\delta}{2}v(\fm)}{S(v)+\frac{c}{2}v(\fm)}=\delta(L).
\]
Moreover, if $v$ is a valuation that computes $\delta(L)$, then it also computes the above infimum.
\end{lem}

\begin{proof}
We have
\[
\delta_{c,m}(L)
=
\inf_{v} \inf_{D} 
\frac{A_{X,\Delta-\frac{c\delta}{2}X_0}(v)}{v(D)} 
=
\inf_v \frac{A_{X,\Delta-\frac{c\delta}{2}X_0}(v)}{S_{c,m}(v)} = \inf_v \frac{A_{X,\Delta}(v)+\frac{c\delta}{2}v(t)}{S_m(v)+\frac{\lceil cm \rceil-1}{2m} v(t)},
\]
where the last equality holds by Lemma \ref{lem:S_c,m} and the infima run through all level-$c$ $m$-basis type  divisors $D$ and $v\in \Val_{X}^\circ$.
The first equality in the lemma then follows from the same argument as in the proof of Proposition \ref{p:delta=inf A/S}. The remaining statement is then a consequence Proposition \ref{p:delta=inf A/S}.
\end{proof}

The main advantage of the level-$c$ basis type divisors is that they can be arranged to contain a positive multiple of an arbitrary divisor. The precise statement is the following. It will play an important role when we show that special complements exist in the log Fano case.

\begin{lem} \label{l:level-c basis type >= any G}
Let $v\in \Val_X$ be a valuation with $A_{X,\Delta}(v)<\infty$ and $C_X(v)\subseteq X_0$. Let $G$ be an effective divisor on $X$ whose support does not contain any component of $X_0$. Then there exists  $\varepsilon>0$ such that for every sufficiently large integer $m\in r\bN$, we can find a level-$c$ $m$-basis type divisor $\Gamma_m$ such that $v(\Gamma_m)=S_{c,m}(v)$ and $\Gamma_m\ge \varepsilon G$.
\end{lem}

\begin{proof}
After rescaling $v$ we may assume that $v(\fm)=1$. By adding components to $G$ and rescaling we may also assume that $G$ is reduced and $G\in |kL|$ for some integer $k>0$. The divisor $G$ induces a filtration $\cG$ on $\cR$ and hence also on $V_{c,m}$. Choose a basis $\bar{s}_1,\dots,\bar{s}_{N_{c,m}}$ of $V_{c,m}$ that is compatible with both filtrations $\cF_v$ and $\cG$. Our plan is to lift it to a level-$c$ $m$-basis that preserves the $v$-value while maximizing the multiplicity along $G$. 

By definition, if $\bar{s}\in \cF^\lambda_v V_{c,m}\setminus \cF^{>\lambda}_v V_{c,m}$, then we can lift $\bar{s}$ to some $s\in \cR_m$ with $v(s)=\lambda$. If in addition $\bar{s}\in \cG^\mu V_{c,m}$ and $\lambda\le cm$, then we can further ensure that $s\in \cR_m$ satisfies $\mult_G (s)\ge \mu$: this is because we can first lift $\bar{s}$ to some $s\in \cG^\mu \cR_m$; but as $\bar{s}\in \cF^\lambda_v V_{c,m}$ and  $v(\fm)=1\ge \frac{\lambda}{cm}$, we  get 
\[
s\in \cF^\lambda_v \cR_m + \fm^{\lceil cm \rceil}\cR_m \subseteq \cF^\lambda_v \cR_m \, ,
\]
i.e. $s\in \cF^\lambda_v \cR_m \cap \cG^\mu \cR_m$.
 Once we lift all the basis elements $\bar{s}_1,\dots,\bar{s}_{N_{c,m}}$ this way, we get a level-$c$ $m$-basis $s_1,\dots,s_{N_{c,m}}$ together with the associated basis type divisor $\Gamma_m$ such that $v(\Gamma_m)=S_{c,m}(v)$. We claim that 
\[
\liminf_{m\to \infty}\mult_G (\Gamma_m) >0,
\]
from which the lemma follows.

To see this, choose some small constants 
\[
0<\varepsilon_1\ll \varepsilon_2\ll c
\] and consider those basis elements $\bar{s}_j$ with
\[
\bar{s}_j\in \cG^{\varepsilon_1 m} V_{c,m} \setminus (\cF_v^{>\varepsilon_2 m} V_{c,m}\cap \cG^{\varepsilon_1 m} V_{c,m}).
\]
By the previous discussion, $\bar{s}_j$ can be lifted to $s_j\in  \cG^{\varepsilon_1 m} \cR_{m}$. Hence $\mult_G (\Gamma_m)\ge \frac{\varepsilon_1 a_m}{N_{c,m}}$ where
\[
a_m:=\dim \left( \cG^{\varepsilon_1 m} V_{c,m} / (\cF_v^{>\varepsilon_2 m} V_{c,m}\cap \cG^{\varepsilon_1 m} V_{c,m}) \right) \, .
\]

Since $G$ does not contain any component of $X_0$, multiplication with powers of the defining equation of $G$ identifies $\cR_{m- k\lceil\mu\rceil}/\fm^{\lceil cm\rceil} \cR_{m-k\lceil\mu\rceil}$ with $\cG^{\mu} V_{c,m}$. It induces an isomorphism 
\[
 \cR_{m- k \lceil \varepsilon_1 m\rceil} / \cF_v^{>(\varepsilon_2 m- \lceil\varepsilon_1 m\rceil \cdot v(G))} \cR_{m- k \lceil \varepsilon_1 m\rceil} 
\cong
\cG^{\varepsilon_1 m} V_{c,m} / (\cF_v^{>\varepsilon_2 m} V_{c,m}\cap \cG^{\varepsilon_1 m} V_{c,m}) 
\, ,
\]
where we use $\varepsilon_1\ll\varepsilon_2\ll c$ for 
\[
 \cF_v^{>(\varepsilon_2 m- \lceil\varepsilon_1 m\rceil \cdot v(G))} \cR_{m- k \lceil \varepsilon_1 m\rceil}  \supseteq \fm^{\lceil cm\rceil} \cR_{m-k \lceil \varepsilon_1 m\rceil} \, .
\]
Since $v(\fm)=1$, by the same argument as in the proof of Lemma \ref{lem:S_c,m} this gives
\begin{eqnarray*}
 {a_m} &=&\sum_{j\ge 0}\dim (R_{m- k \lceil \varepsilon_1 m\rceil}/\cF^{>(\varepsilon_2 m- \lceil\varepsilon_1 m\rceil \cdot v(G)-j)}_{v}|R_{m- k \lceil \varepsilon_1 m\rceil})
\end{eqnarray*}
Combined with Lemma \ref{l:volume drop}, this gives 
$\liminf_{m\to\infty} \frac{a_m}{m^n}>0$ with $n=\dim X$. It is also clear that $N_{c,m}=O(m^n)$. Thus
\[
\liminf_{m\to\infty} \frac{a_m}{N_{c,m}} > 0
\]
and this proves the claim.
\end{proof}

\section{Divisorial minimizer}\label{ss-finite generation}

In this section we prove Theorem \ref{t:divisorial minimizer} on the existence of divisorial minimizers of the relative stability threshold of a log Fano pair over a DVR.

\subsection{Statement of main result}
Let $(X,\Delta)\to C$ be a relative log Fano pair over the spectrum of a DVR $C$ such that $X\to C$ has connected and geometrically connected and reduced fibers. 
We set
\[
\delta(X,\Delta):= \delta(X,\Delta;-K_X-\Delta)
,\]
where we are viewing $L:=-K_X-\Delta$ as $\bQ$-line bundle and fixing $r>0$ such that $-r(K_X+\Delta)$ is a Cartier divisor.
In addition, for $v\in \Val_X$, the  invariants $S(v)$ and $T(v)$ denote the $S$ and $T$ invariants for $v$ with respect to $L$.
The above values are independent of the choice of $r$ by  Proposition \ref{p:reducedSinvproperties}.1 and Proposition \ref{p:delta=inf A/S}.

The goal of this section is to prove the following enhancement of Theorem \ref{t:delta computed by qm valuation}.

\begin{thm} \label{t:divisorial minimizer}
Let $f\colon (X,\Delta) \to C$ be a log Fano pair over the spectrum of a DVR with algebraically closed residue field.  
If $X_0$ is integral, $X_0\subseteq \Supp(\Delta)$, and 
\[
\delta(X,\Delta)<\min\{1,\delta(X_{\eta},\Delta_{\eta})\} \ \ \ \mbox{for }\eta={\rm Spec}\,K(C),
\]
then there exists a divisorial valuation $v\in \Val_X$ with $C_X(v)\subsetneq X_0$ that computes $\delta(X,\Delta)$.
\end{thm}

The key point here is that the minimizer can be made divisorial. This is a relative version of the optimal destabilization theorem \cite[Theorem 1.2]{LXZ-HRFG}. In the absolute case, one shows minimizing valuations induce finitely generated filtrations on the section ring, and this implies that any small perturbation of the minimizer is still a minimizer. Hence we get a divisorial minimizer by choosing a rational perturbation. To deduce the finite generation, we show that any quasi-monomial minimizer of the stability threshold is a monomial log canonical place of a special complement (to be defined below), and then show that this geometric property yields finite generation even for valuations with higher rational rank (than one).  We follow a similar strategy in our relative setting, though each step requires some extra work. In particular, to construct a special complement such that the minimizer is an lc place is tricky.  Once we show a minimizer is a monomial log canonical place of a special complement
in the relative setting, the finite generation can be deduced from \cite{XZ-SDC} using a cone construction.

\subsection{Special complements}

In this subsection, we show that the quasi-monomial minimizers are monomial lc places of special $\bQ$-complements.

\begin{defn}[{\it cf.} {\cite[Definition 3.1]{XZ-SDC}}]\label{d:specialcomp}
A \emph{special $\bQ$-complement} of $(X,\Delta)$ with respect to a log smooth model $\pi\colon(Y,E)\to (X,\Delta)$ is a $\bQ$-complement $\Gamma$ such that $\pi_*^{-1}\Gamma\ge G$ for some ample effective $\bQ$-divisor $G$ on $Y$ whose support does not contain any stratum of $(Y,E)$. Any valuation $v\in \QM(Y,E)$ that is an lc place of $(X,\Delta+\Gamma)$ is called a \emph{monomial lc place} of the special $\bQ$-complement $\Gamma$ with respect to $(Y,E)$.
\end{defn}

For any $\bQ$-divisor $D$ on $X$, we denote its horizontal (resp. vertical) part (over $C$) by $D^h$ (resp. $D^v$).

\begin{prop} \label{p:special complement}
In the setting of Theorem \ref{t:divisorial minimizer}, let $v$ be a quasi-monomial valuation that computes $\delta(X,\Delta)$. 
Then $v$ is a monomial lc place of some special $\bQ$-complement of $(X,\Delta-tX_0)$ for any $t>0$ with $tX_0\le \Delta$.
\end{prop}

Before we prove this general statement, we construct some $\bQ$-complements that realize the minimizer $v$ as a log canonical place. A useful tool is the following statement, which we extract from \cite{LXZ-HRFG}.

\begin{lem} \label{l:limit of complement}
Let $(X,\Delta)$ be a log Fano pair over $C$, i.e. $(X,\Delta)$ is klt and $-K_X-\Delta$ is ample over $C$. 
Fix $\varepsilon>0$, $v\in \Val_X$  a quasi-monomial valuation, and  $G$  an effective divisor on $X$. Assume that there exists a sequence of $\bQ$-complements $\Gamma_m$ $(m=1,2,\dots)$ of $(X,\Delta)$ such that
\[
\lim_{m\to\infty} v(\Gamma_m) = A_{X,\Delta}(v)\quad \text{and}\quad\Gamma_m \ge \varepsilon G.
\]
Then there exists a $\bQ$-complement $\Gamma$ such that $\Gamma\ge \varepsilon G$ and $v$ is an lc place of $(X,\Delta+\Gamma)$.
\end{lem}

\begin{proof}
This follows from the exact same proof of \cite[Lemma 3.1]{LXZ-HRFG}, so we only sketch the main steps. After rescaling we assume $A_{X,\Delta}(v)=1$. Let $\fab=\fab(v)$. By (3.1) in {\it loc. cit.}, for any sufficiently small $\varepsilon_1>0$, there exist some divisors $F_1,\dots,F_r$ over $X$ and some $\varepsilon_0>0$ such that 
\begin{equation} \label{e:A<<1}
    A_{X,\Delta+\fab^{1-\varepsilon_0}}(F_i)<\varepsilon_1
\end{equation}
for all $i=1,\dots,r$.  Moreover, the construction in \cite{LXZ-HRFG} ensures that if $\Gamma$ is a $\bQ$-complement that has all the $F_i$ as lc places, then $v$ is also an lc place. By \cite[Corollary 1.4.3]{BCHM}, there exists a proper birational morphism $p\colon \tX\to X$ that extracts exactly the divisors $F_i$. Hence to prove the lemma it is enough to find a $\bQ$-complement of $(\tX,\tDelta+\varepsilon \tG+\sum F_i)$, where $\tDelta$, $\tG$ denotes the strict transforms of $\Delta$, $G$. Now the assumptions on $\Gamma_m$ together with \eqref{e:A<<1} imply that $\tX$ is of Fano type over $C$ (by Lemma \ref{l:Fanotype}) and $(\tX,\tDelta+\varepsilon \tG+(1-\varepsilon_1)\sum F_i)$ has a $\bQ$-complement, thus we conclude by applying Lemma \ref{lem:ACC} as $0<\varepsilon_1\ll 1$.
\end{proof}

The following is essentially \cite[Lemma 3.2]{LXZ-HRFG}.
\begin{lem} \label{lem:ACC}
Let $(X,\Delta)$ be a pair and let $G$ be an effective $\bQ$-Cartier $\bQ$-divisor on $X$. Assume that $(X,\Delta)$ is projective and of Fano type over $C$. Then there exists some $\varepsilon>0$ depending only on the dimension of $X$ and the coefficients of $\Delta$ and $G$ such that: if $(X,\Delta+(1-\varepsilon)G)$ has a $\bQ$-complement over $C$, then the same is true for $(X,\Delta+G)$. 
\end{lem}

\begin{proof}
By replacing $X$ by a small $\bQ$-factorial modification, we may assume that $X$ itself is $\bQ$-factorial. Let $n=\dim X$ and let $I\subseteq \bQ_+$ be the coefficient set of $\Delta$ and $G$. By \cite{HMX-ACC}*{Theorems 1.1 and 1.5}, we know that there exists some rational constant $\varepsilon>0$ depending only on $n,I$ which satisfies the following property: for any pair $(X,\Delta)$ of dimension at most $n$ and any $\bQ$-Cartier divisor $G$ on $X$ such that the coefficients of $\Delta$ and $G$ belong to $I$, we have $(X,\Delta+G)$ is lc as long as $(X,\Delta+(1-\varepsilon)G)$ is lc; if in addition $X$ is projective and there exists some divisor $D$ with $(1-\varepsilon)G\le D\le G$ such that $K_X+\Delta+D\sim_{\bQ} 0$, then $D=G$. 

Now let $(X,\Delta+(1-\varepsilon)G)$ be a pair with a $\bQ$-complement $\Gamma$. As $X$ is of Fano type, we may run the $-(K_X+\Delta+G)$-MMP $f\colon X\dashrightarrow X'$ over $C$. Let $\Delta'$, $G'$, and $\Gamma'$ denote the strict transforms of $\Delta$, $G$, and $\Gamma$. Note that by construction
\[
K_X+\Delta+G \le f^*(K_{X'}+\Delta'+G').
\]
Hence $(X,\Delta+G)$ has a $\bQ$-complement over $C$ if and only if $(X',\Delta'+G')$ has one. Since 
\[
K_X+\Delta+(1-\varepsilon)G+\Gamma\sim_{\bQ} 0,
\]
the MMP is crepant for the lc pair $(X,\Delta+(1-\varepsilon)G+\Gamma)$. Hence $(X',\Delta'+(1-\varepsilon)G'+\Gamma')$ is also lc. It follows that $(X',\Delta'+(1-\varepsilon)G')$ is lc. Thus by our choice of $\varepsilon$, $(X',\Delta'+G')$ is lc as well. Suppose that $X'$ is Mori fiber space $g\colon X'\to S$ over $C$. Then $K_{X'}+\Delta'+G'$ is $g$-ample. Since $K_{X'}+\Delta'+(1-\varepsilon)G'\sim_\bQ -\Gamma'\le 0$ and 
$\rho(X'/S)=1$, there exists some $\varepsilon'\in (0,\varepsilon]$ such that $K_{X'}+\Delta'+(1-\varepsilon')G'\sim_{\bQ,S} 0$. But if we restrict the pair to the generic fiber of $X'\to S$ we would get a contradiction to our choice of $\varepsilon$. 

Thus $X'$ is a minimal model and $-(K_{X'}+\Delta'+G')$ is nef. As $X'$ is also of Fano type, we see that $-(K_{X'}+\Delta'+G')$ is semiample over $C$, hence  $(X',\Delta'+G')$ has a $\bQ$-complement. By the previous discussion, this implies that $(X,\Delta+G)$ has a $\bQ$-complement as well.
\end{proof}

\begin{lem} \label{l:complement>=given divisor}
Under the assumptions of Theorem \ref{t:divisorial minimizer}, let $v$ be a quasi-monomial valuation that computes $\delta(X,\Delta)$. Then for any positive $t$ with $\Delta\ge tX_0$ and any effective divisor $G$ on $X$, there exists some $\bQ$-complement $\Gamma$ of $(X,\Delta-tX_0)$ such that $\Gamma\ge \varepsilon G$ for some $\varepsilon>0$ and $v$ is an lc place of $(X,\Delta-tX_0+\Gamma)$.
\end{lem}

\begin{proof}
We will see that this is a formal consequence of Lemmas \ref{l:level-c basis type >= any G} and \ref{l:limit of complement}. By Lemma \ref{l:center in X0}, the center of $v$ is contained in $X_0$. We may assume that $G^v=0$ as it will be clear from the proof that the complement we obtain automatically contains some positive multiple of $X_0$. 

Choose  $c>0$ such that $t>\frac{c\delta}{2}$, where $\delta=\delta(X,\Delta)$. Let 
\[
\Delta':=\Delta-\frac{c\delta}{2}X_0\ge \Delta-tX_0\,.
\]
By Lemma \ref{l:level-c basis type >= any G}, there exist some constant $\varepsilon_0>0$ and some level-$c$ $m$-basis type divisors $\Gamma_m\sim_\bQ L$ such that $v(\Gamma_m)=S_{c,m}(v)$ and $\Gamma_m\ge \varepsilon_0 G$ for each  $m\in r \bN$ with $m\gg0$. Let $\delta_{c,m}=\delta_{c,m}(L)$ be as in \eqref{defn-delta cm}. Thus $(X,\Delta'+\delta_{c,m} \Gamma_m)$ is log canonical. By Lemma \ref{l:delta_c,m} and our assumption that $\delta(X,\Delta)<1$, we also see that $-(K_X+\Delta'+\delta_{c,m} \Gamma_m)$ is ample over $C$. Thus by Bertini's theorem we get a $\bQ$-complement $\Gamma'_m$ of $(X,\Delta')$  such that $\Gamma'_m\ge \delta_{c,m}\Gamma_m$. 

Since $\lim_{m\to\infty}\delta_{c,m}=\delta(L)> 0$ by Lemma \ref{l:delta_c,m}, 
\[
\Gamma'_m\ge \delta_{c,m}\Gamma_m\ge \varepsilon G
\]
 for some constant $\varepsilon>0$ that is independent of $m$.
Moreover, we deduce that 
\begin{eqnarray*}
\lim_{m\to\infty} v(\Gamma'_m)&=&\lim_{m\to\infty} v(\delta_{c,m}\Gamma_m) \\
&= &\lim_{m\to\infty}\delta_{c,m}S_{c,m}(v)  \ \ \ \ \mbox{(by Lemmas \ref{lem:S_c,m})}\\
&= &\delta(L)(S(v)-\frac{c}2 v(X_0) )= A_{X,\Delta'}(v) \, .
\end{eqnarray*}
 Thus we can apply Lemma \ref{l:limit of complement} to get a $\bQ$-complement $\Gamma'\ge \varepsilon G$ of $(X,\Delta')$ that has $v$ as an lc place. Now $\Gamma=\Gamma'+(t-\frac{c\delta}{2})X_0$ is a $\bQ$-complement of $(X,\Delta-tX_0)$ and all the conditions we want are satisfied.
\end{proof}

\begin{proof}[Proof of Proposition \ref{p:special complement}]
This is proved as in \cite[Corollary 3.4]{LXZ-HRFG}. Let $\pi\colon (Y,E)\to (X,\Delta)$ be a log smooth model such that $v\in \QM(Y,E)$, and let $G$ be a general ample effective divisor on $Y$. By Lemma \ref{l:complement>=given divisor}, $v$ is an lc place of some $\bQ$-complement $\Gamma$ of $(X,\Delta-tX_0)$ whose support contains $\pi_* G$. In particular $G\subset \Supp(\pi_*^{-1} \Gamma)$, hence $\Gamma$ is necessarily special.
\end{proof}

\begin{cor}\label{cor-minimizer lc place}
Under the assumptions of Theorem \ref{t:divisorial minimizer}, any quasi-monomial valuation that computes $\delta(X,\Delta)$ is an lc place of complement. 
\end{cor}

\begin{proof}
This follows  from combining Lemma \ref{l:limit of complement} and  Proposition \ref{p:special complement}.
\end{proof}

\subsection{Finite generation}

Let $f\colon (X,\Delta) \to C$ be a log Fano pair over the spectrum of a DVR, $L=-(K_X+\Delta)$, and $r>0$ a positive integer such that $rL$ is a Cartier divisor. For any valuation $v\in \Val_X$, we denote by 
\[
\gr_v \cR:= \oplus_{m\in r\bN}\oplus_{\lambda\in \bR} \cF_v^\lambda \cR_m / \cF_v^{>\lambda}\cR_m
\]
the associated graded ring of the induced filtration $\cF_v$ on the relative section ring $\cR$. The next step is to study the finite generation of this graded ring (as a $\bk$-algebra). 

\begin{prop} \label{p:finite generation}
Assume that $v$ is a monomial lc place of some special $\bQ$-complement $\Gamma$ of $(X,\Delta)$ with respect to a log resolution $(X',E')\to (X,\Delta)$, and its center is contained in $X_0$. Then $\gr_v \cR$ is finitely generated.
\end{prop}

\begin{proof}
The plan is to reduce this to the finite generation results of \cite{XZ-SDC} through a cone construction. Consider the relative cone $(Y,D)$ of the family $(X,\Delta)\to C$ with respect to the polarization $L$. In particular, $Y=\Spec_C(\cR)$, it has a fiberwise $\bG_m$-action over $C$, and $(Y,D)$ is klt. The cone vertices form a section $\sigma\colon C\to Y$ of the family $Y\to C$. Let $\pi\colon \tY:=\Bl_{\sigma(C)}Y\to Y$ be the blowup of $Y$ at $\sigma(C)$ and let $E$ be the exceptional divisor. Then $\tY$ is also the total space of the line bundle $\cO_X(-rL)$ over $X$ and $D=\pi_*p^*\Delta$, where $p\colon \tY\to X$ denotes the projection.

Define a $\bG_m$-invariant valuation $v_1$ on $Y$ by $v_1(f):=m+v(f)$ for any $0\neq f\in \cR_m$. Geometrically, $v_1$ is the quasi-monomial combination of $\ord_E$ and $v$. Since $C_X(v)\subseteq X_0$ by assumption, and $v_1(\cR_m)>0$ for all $m>0$, we see that $v_1$ is centered at the closed point $\sigma(0)\in Y$. Clearly, the associated graded ring $\gr_{v_1}\cR$ is isomorphic to $\gr_v \cR$. Thus, by \cite[Theorem 1.4]{XZ-SDC}, it suffices to show that $v_1$ is a monomial lc place of some special $\bQ$-complement of $(Y,D)$.

We claim that $G:=\pi_*p^*\Gamma$ is one such special $\bQ$-complement. Indeed, since $(Y,D+G)$ is the relative cone over $(X,\Delta+\Gamma)$ and $\Gamma$ is a $\bQ$-complement of $(X,\Delta)$, we see that $G$ is a $\bQ$-complement of $(Y,D)$ and $E$ is an lc place of $(Y,D+G)$. Moreover, as $v$ is an lc place of $(X,\Delta+\Gamma)$ by assumption, the quasi-monomial combination $v_1$ is also an lc place of $(Y,D+G)$. It remains to show that the $\bQ$-complement $G$ is special. 

Let $Y'=\tY\times_X X'$ and $q\colon Y'\to \tY$ the induced map. Denote the induced map $Y'\to X'$ by $p'$. Note that $Y'$ is the total space of the line bundle $g^*\cO_X(-rL)$ and $q^*E$ is the zero section.  
\[
\begin{tikzcd}
    Y' \arrow[r,"q"] \arrow[d,"p'"'] & \tY \arrow[r,"\pi"] \arrow[d,"p"] & Y \\
    X' \arrow[r, "g"] & X & \\
\end{tikzcd}
\]
Let $G'$ be the strict transform of $G$ on $Y'$ and $\Gamma'$ the strict transform of $\Gamma$ on $X'$. Then $G'=p'^*\Gamma'$. Since $\Gamma$ is a special $\bQ$-complement, we have $\Gamma'\ge \Gamma_0$ for some effective ample divisor $\Gamma_0$ whose support does not contain any stratum of $(X',E')$. It follows that $G'\ge p'^*\Gamma_0$ and the support of $p^*\Gamma_0$ does not contain any stratum of $(Y',p^*E'+q^*E)$. By construction, we have $v_1\in \QM(Y',p'^*E'+q^*E)$. Thus to finish the proof it suffices to show that $p^*\Gamma_0$ is ample on $Y'$. As $Y$ is affine, it is enough to show that $p'^*\Gamma_0$ is ample over $Y$. Now away from $\sigma(C)\subseteq Y$ the map $\pi$ is an isomorphism, hence the fibers of $Y'\to Y$ projects isomorphically to the fibers of $X'\to X$, while the preimage of $\sigma(C)$ in $Y'$ is the zero section $q^*E$, which projects isomorphically to $X'$. Thus as $\Gamma_0$ is ample on $X'$ we see that $p'^*\Gamma_0$ is ample on every fiber of $Y'\to Y$. It follows that $p'^*\Gamma_0$ is ample as desired.
\end{proof}

\begin{cor} \label{c:delta minimizer f.g.}
Assume we are in the setting of Theorem \ref{t:divisorial minimizer} and let $v$ be a quasi-monomial valuation that computes $\delta(X,\Delta)$. Then $\gr_v \cR$ is finitely generated.
\end{cor}

\begin{proof}
Choose some $0<t\ll 1$ such that $\Delta\ge t X_0$. Since  
\[
K_X+\Delta\sim_{\bQ,C} K_X+\Delta-tX_0.
\]
It follows that $(X,\Delta-tX_0)$ is also log Fano over $C$. By Proposition \ref{p:special complement}, the valuation $v$ is a monomial lc place of some special $\bQ$-complement of $(X,\Delta-tX_0)$. 
Thus, by Proposition \ref{p:finite generation}, the associated graded ring $\gr_v \cR$ is finitely generated.
\end{proof}

\subsection{Existence of a divisorial minimizer} We are now ready to prove Theorem \ref{t:divisorial minimizer}.

\begin{proof}[Proof of Theorem \ref{t:divisorial minimizer}]
We will prove the result in increasing levels of generality over three steps as we would like to apply results from Section \ref{ss-minimizer} where we assume the residue field of $\cO_{C,0}$ is uncountable and algebraically closed.
\medskip

 \noindent \emph{Step 1: We  will prove the result when the residue field of $\cO_{C,0}$ is uncountable and algebraically closed.}

The strategy is similar in spirit to the proof of \cite[Theorem 5.1]{LXZ-HRFG}. By Proposition \ref{p:deltafieldext} and our assumption on the residue field, Theorem \ref{t:delta computed by qm valuation} implies that there exists a quasi-monomial valuation $w$  that computes $\delta(X,\Delta)$ and has $C_X(w)\subsetneq X_0$. By Corollary \ref{c:delta minimizer f.g.}, the associated graded ring $\gr_w \cR$ is finitely generated. Since $w$ is quasi-monomial, we can find a log smooth model $(Y,E)$ of $(X,\Delta)$ such that $w\in \QM(Y,E)$. Let $\Sigma\subseteq \QM(Y,E)$ be the smallest rational PL subspace containing $w$. We claim that the $S$-invariant function $v\mapsto S(v)$ on $\Sigma$ is linear in a neighborhood of $w$. To see this, first notice that $\gr_v R \cong \gr_v \cR/ (\fm\cdot\gr_v \cR)$.
Hence
\begin{equation} \label{e:formula for S_m}
    mN_m\cdot S_m(v) = \sum_{\lambda \in \bR} \lambda\cdot \dim \gr_v^\lambda R_m = \sum_{\lambda\in \bR} \lambda\cdot (\dim \gr_v^\lambda \cR_m - \dim \gr_v^{\lambda-v(\fm)} \cR_m ).
\end{equation}
Fix a finite set $s_1,\dots,s_r$ of  homogeneous generators of $\gr_w \cR$. From the way quasi-monomial valuations are defined, we see that there exists a small neighborhood $U$ of $w$ in $\Sigma$, such that the function $v\mapsto v(s_i)$ is linear for every $1\le i\le r$. By \cite[Lemma 2.10]{LX-higher-rank}, after possibly shrinking $U$, for all divisorial valuation $v\in U$ there is an isomorphism $\varphi_{w,v}\colon \gr_w \cR\cong \gr_v \cR$ sending the image of $s_i$ in $\gr_w \cR$ to the corresponding one in $\gr_v \cR$. In particular, for any homogeneous element $s\in \gr_w \cR$, the function $v\mapsto v(\varphi_{w,v}(s))$ on $U(\bQ)$ is linear. Combined with \eqref{e:formula for S_m}, this implies that $v\mapsto S_m(v)$ is linear on $U(\bQ)$. Hence it extends to a linear function on $U$ and by taking the limit we deduce that the $S$-invariant function $v\mapsto S(v)$ is linear on $U$ as well, which proves the claim.

Note that the log discrepancy function $v\mapsto A_{X,\Delta}(v)$ is also linear in $v$ after possibly shrinking the neighborhood $U$. Since $v$ computes $\delta(X,\Delta)$, this implies that
\[
A_{X,\Delta}(v)-\delta(X,\Delta)S(v)\geq 0
\]
for all $v\in U$, with equality when $v=w$. Since the left hand side is linear on $U$, equality holds on $U$. It follows that any divisorial valuation $v\in U$ also computes $\delta(X,\Delta)$.

\medskip

\noindent \emph{Step 2: We will prove the result in the case when  $C= \Spec (\bk[[t]])$, where $\bk$ is an algebraically closed field of characteristic $0$.}

In this case, we will deduce the result from Step 1 using a spreading out argument.
Fix a field extension $\bk\subset \bk' $ such that $\bk'$ is algebraically closed and uncountable. 
Let 
\[
A:= \bk[[t]]\quad \text{ and } \quad A':= \bk'[[t]].
\]
Thus $C= \Spec(A)$. Let $C' := \Spec(A')$. 
Write  $0'$ and $\eta'$ for the closed  and generic points of $C'$.
Consider the relative log Fano pair
\[
 (X',\Delta'):= (X,\Delta)\times_C C' \to C'
\]
with morphism $\rho: X'\to X$.
By Propositions \ref{p:deltafieldext} and \ref{p:deltaetalebasechange}, our assumptions on $(X,\Delta)\to C$ imply that
\[
\delta(X',\Delta') \leq \min \{1, \delta(X'_{\eta'},\Delta'_{\eta'})\}
.\]
Therefore Step 1 implies that there exists a divisor $E'$ over $X'$ such that $\ord_{E'}$ computes $\delta(X',\Delta')$. 
Furthermore, by Corollary \ref{cor-minimizer lc place}, $\ord_{E'}$ is the lc place of a complement of $(X',\Delta')$.
By \cite[Corollary 1.68]{Xu-Kbook}, there exists a proper birational morphism 
\[g':Y' \to X'
\]
such that $Y'$ is normal, $E'$ appears as a prime divisor on $Y'$, and $-E'$ is $g$-ample. 
Therefore, there exists  $\ell>0$ such that $g'$ is the blowup of the ideal
\[
\fb':= g_* \cO_{Y'}(-\ell E) \subset \cO_X
.\]
Note that $\fb'$ contains the ideal $\cI_{X'_0}^{\ell} = (t^\ell) \subset  \cO_{X'} $, since $\ord_{E'}(t) \geq 1$. 

Now consider the Hilbert scheme $H$ that  parameterizes closed subschemes of  $V(\cI_{X_0}^\ell)$ or, equivalently, quasi-coherent ideal sheaves $\fb \subset \cO_{X}$ that  contain $\cI_{X_0}^\ell$.
Since $V(\cI_{X_0}^\ell)$ is a projective finite type $\bk$-scheme, $H$ has the structure of a $\bk$-scheme and its connected components are projective over $\bk$.
Write  $\cJ\subset \cO_{X\times H}$ for universal family of ideal sheaves.
The ideal $\fb'$ induces a morphism $h':\Spec(\bk') \to H$
such that $\fb' = \cJ \cdot \cO_{X_{k(h')}}$. 
Let $Z \subset H$ denotes the closure of the image of $h'$ with reduced scheme structure. 

Consider the blowup
\[
Y:= {\rm Bl}_{\cJ \cdot \cO_{X\times Z}}  X\times Z \to X\times Z
.\]
Since $\fb' = \cJ \cdot \cO_{X_{k(h')}}$,
the base change $Y_{h'} \to X_{k(h')}$ is the blowup of $X'$ along $\fb'$ and, hence, agrees with the  morphism $Y' \to X'$. 
Since $Y'\to X'$ is the base change of $Y_{K(Z)} \to X_{K(Z)}$ via the extension of fields $k(Z) \subset \bk'$,  $Y_{K(Z)}$ is normal and there exists a unique prime exceptional divisor $E_{K(Z)}$ on $Y_{K(Z)}$ whose base change via $K(Z) \subset \bk'$  is $E'$. 
Let $E \subset Y$ denote  closure of $E_{K(Z)}$ in $Y$. 

We now construct an open set of $Z$ over which $E\subset Y \to X\times Z$ shares properties with  $E'\subset Y'\to X'$.
Since $Y_{K(Z)}$ is normal, there exists a non-empty open set $U\subset Z$ such that the fibers of $Y_{U} \to U$ are normal. 
Since $E_{K(Z)}$ is the unique exceptional divisor of $Y_{K(Z)}\to X_{K(Z)}$,
by shrinking $U$, we may assume that $E_z$ is integral and the unique exceptional divisor of $Y_z \to X_z$ for all $z\in U$. 
By  shrinking $U$, we may assume that 
\[
A_{X',\Delta'}(E') = A_{X_{K(Z)},\Delta_{K(Z)}} (E_{K(Z)} ) 
= 
A_{X_{k(z)},\Delta_{k(z)}}(E_z)
\]
for all $z\in U$. 

We now show that $S(E_{K(Z)} ) \leq S(E_z)$ for all $z\in U$. 
Fix a point $z\in Z$. By constructing a divisor over $Z$ with center $z$, we may construct a DVR $B$ with fraction field $K$ and dominant morphism
$T:= \Spec(B)  \to Z$ such that   the closed point $b\in B$ maps to $z$ and $K(Z) \subset K$ is an equality. 
For a $\bk$-algebra $B$ and $m\in r\bN^+$, set 
\[
\cR^B_m  := \cR_m \otimes_{\bk} B \quad \text{ and } \quad R_m^B:= R_m\otimes_{\bk} B
.\]
For $\la \in \bR$, set
$
\cG_K^\la R_m^K := \im ( \cF_{E_K}^\la \cR_m^K \to R_m^K )
$ and consider the filtration by sub $K$-vector spaces 
\[
R_m^K = \cG^0_K R_m^K \supset \cG^1_K R_m^K \supset \cdots 
\]
By the properness of the flag variety (or by explicitly setting $\cG^\la R_m^B  := \cG_K^\la R_m^K \cap R_m^B$), there exists a filtration  of $R_m^B$ by free $B$-modules 
\[
R_m^B = \cG^0 R_m^B \supset \cG^1 R_m^B \supset\cdots 
\]
such that the quotients $R_m^B/\cG^\la R_m^B$ are free and the filtration restricts to the original filtration after applying  $\otimes_B K$.
Now choose a $B$-module basis $(s_1,\ldots, s_{N_m})$ for $R_m^B$  such that each $\cG^\la R_m^B$ is spanned by a subset of the basis elements. 
Next, lift  $(s_1,\ldots, s_{N_m})$  to elements $\widetilde{s}_1 , \ldots, \widetilde{s}_{N_m}$ of $\cR_m^B$ such that $\ord_{E_K}(\tilde{s}_i)  = \ord_{\cG}(s_i)$. Let
\[
D:= \frac{1}{mN_M} \left( \{\widetilde{s}_1=0 \} + \cdots + \{\widetilde{s}_{N_m}=0\}\right)
.\]
Using that  that $D_K$ is a relative $m$-basis type divisor of $-K_{X_K}-\Delta_K$ compatible with $E_K$, we compute that
\[
S_m(E_{K(Z)}) = S_m(E_K) = \ord_{E_K}(D_K) \leq \ord_{E_z}(D_z) \leq S_m(E_z)
,\]
where the second equality and the last inequality hold by Lemma \ref{l:S(v)basistype}. 
By sending $m \to \infty$, we deduce that $S(E_{K(Z)}) \leq S(E_z)$ as desired.

Now observe that for $z\in U(\bk)$, we have that  
\[
\delta(X,\Delta) \leq \frac{A_{X,\Delta}(E_z)}{S(E_z)}
\leq 
\frac{A_{X_{K(Z)},\Delta_{K(Z)}}(E_{K(Z)})}{S(E_{K(Z)})}
\leq 
\frac{A_{X',\Delta'}(E')}{S(E')}
=
\delta(X',\Delta').
\]
Since we also have that $\delta(X',\Delta') \leq \delta(X,\Delta)$ by Proposition \ref{p:deltaetalebasechange}, the above inequalities are all equalities. 
Therefore $\ord_{E_z}$ computes $\delta(X,\Delta)$ for all $z\in U(\bk)$. In particular, there exists a divisorial valuation that computes $\delta(X,\Delta)$.\medskip

\noindent \emph{Step 3: We will now prove the theorem in full generality.} 

Let $\widehat{C} := \Spec( \widehat{\cO_{C,0}})$. Write $\hat{0}$ and $\hat{\eta}$ for the closed and generic points of $\widehat{C}$. Consider the relative log Fano pair 
\[
(\widehat{X}, \widehat{\Delta}) := (X,\Delta)\times_C \widehat{C}  \to \widehat{C}
.\]
By combining Propositions \ref{p:deltafieldext} and \ref{p:deltaetalebasechange} with our  assumptions on  $(X,\Delta) \to C$, we see that
\[
\delta( \widehat{X}, \widehat{\Delta})
< \min\{1, \delta( X_{\hat{\eta}}, \Delta_{\hat{\eta}} ) \} .
\]
Since $\widehat{\cO_{C,0}} \cong \bk[\![t]\!]$, where $\bk$ is the residue field of $\cO_{C,0}$, Step 2 implies that there exists a divisorial valuation $\widehat{v}$ on $\widehat{X}$ that computes $\delta(\widehat{X},\widehat{\Delta})$.
Arguing as in Proof of Theorem \ref{t:delta computed by qm valuation}, $v:= \widehat{v}\vert_{K(X)}$ is a divisorial valuation of $X$ satisfying
$A_{X,\Delta}(v) = A_{\widehat{X},\widehat{\Delta}}(\widehat{v})$
and 
$S_{X,\Delta}(v) = S_{\widehat{X},\widehat{\Delta}}(\widehat{v})$. As $\delta(\widehat{X},\widehat{\Delta}) = \delta(X,\Delta) $ by Proposition \ref{p:deltaetalebasechange}, this implies that $v$ computes $\delta(X,\Delta)$. 
\end{proof}

\subsection{Families of log Fano pairs}
We now apply Theorem \ref{t:divisorial minimizer} to  families of log Fano pairs over a DVR. 

\begin{cor}\label{cor-minimizer complement}
If $(X,B)\to C$ is a family of log Fano pairs over the spectrum of a DVR with
\[
\delta(X_0,B_0)<\min\{\delta(X_\eta, B_\eta),1\}
,
\]
then there exists $\varepsilon>0$ such that for any $t\in (0,\varepsilon)\cap \bQ$ the following holds:
\begin{enumerate}
    \item $\delta(X_0,B_0)< \delta(X,B+(1-t)X_0) <1$;
    \item there exist a divisorial valuation $v= \ord_E$ with $C_X(v) \subsetneq X_0$ computing 
    \[\delta(X,B+(1-t)X_0);\]
    and  $v$ is an lc place of a $\bQ$-complement $B^+$ of $(X,B)$ with $B^+ \geq X_0$.
\end{enumerate}
\end{cor}

In the above statement, note that  $\delta(X,B+(1-t)X_0)$ denotes the stability threshold of the relative log Fano pair $(X,B+(1-t)X_0)\to C$.

\begin{proof}
By Lemma \ref{lem-tconvergeto0} and our assumption on $\delta(X_0,B_0)$, there exists $\varepsilon>0$ such that 
\[
\delta(X,B+(1-t)X_0) < \min \{ \delta(X_\eta,B_\eta),1\}
\]
for all $t\in (0,\varepsilon)\cap \bQ$.
For such $t$, Proposition \ref{p:deltafieldext} implies that $(X,B+(1-t)X_0)$ satisfies the assumptions of  Theorem \ref{t:divisorial minimizer}. 
Thus   there exists a divisorial valuation $v=\ord_E$ such that 
 $C_X(v)\subsetneq X_0$ and $v$  computes $\delta(X,B+(1-t)X_0)$.
Observe that  
\begin{multline*}
\delta(X_0, B_0)  =   \delta(X,B+X_0)
\leq   \frac{A_{X,B+X_0}(v)}{S(v)}
<  \frac{A_{X,B+(1-t)X_0}(v)}{S(v)}\\
= \delta(X,B+(1-t)X_0) < 1 \, ,
\end{multline*}
where the first relation holds by Lemma \ref{lem-tconvergeto0} and the third uses that $C_X(v) \subseteq X_0$. 
Therefore  condition (1) holds. 

To verify (2) holds, it remains to construct the complement $B^+$.
By Corollary \ref{cor-minimizer lc place}, $E$ is an lc place of $(X, B'^+)$ of a $\bQ$-complement $B'^+\ge B+(1-t)X_0$.
We seek to replace $B'^+$ with a new complement so that $B^+ \geq X_0$.
By \cite[Corollary 1.4.3]{BCHM}, there exists a proper birational morphism $g: Z\to X$  such that $Z$ is normal and $\bQ$-factorial and $E\subset Z$ is the sole exceptional divisor. 
By Lemma \ref{l:Fanotype}, $Z$ is Fano type over $C$.
Since
\[
(Z, g_*^{-1}(B +(1-t)X_0)+E)
\]
admits a $\bQ$-complement given by $g_*^{-1}(B'^+)+E$,
Lemma \ref{lem:ACC} implies that, assuming $\varepsilon$ was chosen sufficiently small,
there exists a $\bQ$-complement $B_Z^+$ of 
\[
(Z, g_*^{-1} (B +X_0)+E)
.\]
Now  $B^+:= g_* B^+_Z$ is a $\bQ$-complement of $(X,B)$ with lc place along $E$ and $B^+\geq X_0$.
\end{proof}

\section{Log Fano family induced by minimizer}\label{s-increasestabilitythreshold}

In this section, we will complete the proof of Theorem \ref{thm-main}. 
In light of Corollary \ref{cor-minimizer complement},  the remaining step is to use the divisorial minimizer of the relative stability threshold to construct  a new family of log Fano pairs and bound the stability threshold of the special fiber.

\subsection{Setup and main result}

Throughout this section, we fix the following setup. 
\begin{setup}\label{assu:logFano}
Let $f:(X,\Delta) \to C$ be a relative log Fano pair  over the the spectrum of a DVR $C$ essentially of finite type over $\bk$ and 
 $v= \ord_E$ be a divisorial minimizer of the relative stability threshold $\delta(X,\Delta)$  with $C_X(E) \subsetneq X_0$ and $X_0\subseteq{\rm Supp}(\Delta)$. 
 Furthermore, we assume that 
\begin{enumerate}
	\item $\delta(X,\Delta) <1$, 
\item 	$X_0$ is irreducible and reduced, and
\item $E$ is an lc place of a complement $\Delta^+$ of  $(X,\Delta)$ with $\Delta^+\geq X_0$.
\end{enumerate}
\end{setup}

\begin{rem}
This setup arises naturally in the study of families of log Fano pairs over DVRs; see Corollary \ref{cor-minimizer complement}.
As well, condition (2) will be needed in Lemma \ref{lem-dualfiltration}
and (3) will be needed in Section \ref{ss:basechange}.
\end{rem}

The goal of this section is to prove the following result.

\begin{thm}\label{t:increasedelta}
Under the above assumptions in Setup \ref{assu:logFano}, there exists a finite dominant morphism $\oC \to C$ such that $\oC$ is the spectrum of a DVR with closed point $0\in \oC$ and $\cO(C) \to \cO(\oC)$ is an extension of DVRs, and  a family of log Fano pairs $(\overline{X}',\overline{\Delta}')\to \overline{C}$
such that 
\begin{enumerate}
	\item there exists an isomorphism $(\overline{X}',\overline{\Delta}')_{\overline{C}\setminus 0}\simeq (X,\Delta)_{\oC\setminus 0}$ over $\oC\setminus 0$,
\item the restriction of  $\ord_{\oX'_0}$  to a valuation of  $K(X)$ via the inclusion
$K(X) \subset K(\oX')$ agrees with $\ord_E$,
	 and 
\item $\delta(\overline{X}'_0 ,\overline{\Delta}'_{0})\geq \delta(X,\Delta) $.
\end{enumerate}
\end{thm}

To prove the theorem, we will use straightforward  MMP argument to construct  $(\overline{X}',\overline{\Delta}')\to \overline{C}$ satisfying  (1) and (2). 
Verifying that condition (3) holds  will require a lengthy and subtle analysis.

\subsection{Birational models}\label{ss:birmodels}

In this subsection, we will construct two birational models  of $X$ related to the minimizer $v=\ord_E$.
The first birational model  replaces $X_0$ with $E$ and the second extracts $E$ over $X$. These are well-known results following \cite{BCHM}, but we include an argument for reader's convenience.

\begin{prop}\label{p:Fanoreplacement}
There exists a relative log Fano pair $(X',\Delta')\to C$ with an isomorphism $f':(X,\Delta) \vert_{C\setminus 0}  \to (X', \Delta')\vert_{C\setminus 0} 
$ over $C\setminus 0$ such that
	\begin{enumerate}
		\item   the divisor $X'^{\rm red}_{0}$ is  prime  and the birational transform of $E$ on $X'$, and 
		\item the pair $(X',\Delta'+X'^{\rm red}_0)$ is plt.
\end{enumerate} 
\end{prop}

Note that $X'_0$ is not necessarily reduced, since 
$X'_0 : = d X'^{\rm red}_{0}$, where $d:= \ord_E(X_0) = \ord_E(\pi)$ and $\pi \in \cO_C(C) $ is a uniformizer of the DVR.

\begin{proof}
We first construct a proper birational morphism  $h:Z\to X$ such that $Z$ is a normal $\bQ$-factorial variety, $Z$ is Fano type over $C$, and $E_Z$, which denotes the birational transform of $E$ on $Z$, is the sole exceptional divisor of $h$.
By Corollary \ref{cor-minimizer lc place}, there exists a complement $\Delta^+$ of $(X,\Delta)$ such that $A_{X,\Delta^+}(E)=0$. Then
$G:= \Delta+(1-\varepsilon)(\Delta^+-\Delta)$ satisfies $A_{X,G}(E)<1$ when $0<\varepsilon\ll1$. 
By \cite[Corollary 1.39]{Kol13}, there exists a proper birational morphism $g:Z\to X$ of normal varieties such that $Z$ is $\bQ$-factorial, the birational transform of $E$ on $Z$, denoted $E_Z$, is the sole exceptional divisor, and $-E_Z$ is nef over $X$. 
By Lemma \ref{l:Fanotype}, $Z$ is Fano type over $C$.
 
We now proceed to contract the birational transform of $X_0$ on $Z$.
Since $Z$ is Fano type over $C$,    \cite[Corollary 1.3.1]{BCHM} implies that we may run a $-E_Z$ MMP over $C$ that terminates.
Since
\[
0\sim
h^*X_0 = \widetilde{X}_0 + dE_Z
,\]
where $\tilde{X}_0 := h_*^{-1}(X_0)$ and  $d:= \ord_E(\pi) >1$, 
$\widetilde{X}_0 \sim - d E_Z$.
Therefore the MMP process only contracts curves contained in $\widetilde{X}_0$ and, hence, outputs a birational contraction $h:Z\dashrightarrow W$
such that $-E_W:= -h_*  E_Z$ is nef. 
As $-E_W $ is nef over $C$,  $Z\dashrightarrow W$ must contract $\tilde{
X}_0$.
Therefore $W_0 = d E_W$.

Next, we construct the relative log Fano pair $f':(X',\Delta')\to C$.
Write $\Delta_{W}$ for the birational transform of $\Delta$ on $W$.
Since $Z\dashrightarrow W$ is a birational contraction and $Z$ is of Fano type over $C$, $W$ is of Fano type over $C$.
Thus, by \cite[Corollary 1.3.1]{BCHM}, we  may run a $-K_{W}-\Delta_{W}$ MMP over $C$. 
As $-K_W-\Delta_W$ is big over $C$, the MMP terminates with a birational contraction
$h:W \dashrightarrow W'$ such that $-K_{W'} - \Delta_{W'}$ is nef and semi-ample. 
Let $W' \to X'$ denote the morphism to the ample model. 
Write $\Delta'$ for the birational transform of $\Delta_{W'}$ on $X'$, which is also the birational transform of $\Delta$ on $W$.
Note that both $X\setminus X_0$ and $X' \setminus X'
_0$ are the ample model of $-K_{W}-\Delta_W$ over $C\setminus 0$.
As the ample model is unique,  $X\dashrightarrow X'$ extends to an isomorphism over $X_{C\setminus 0}\simeq  X'_{C\setminus 0}$ over $C\setminus 0$.

We will now verify that $(X',\Delta')\to C$ is a log Fano pair over $C$.
By construction,   $-K_{X'}-\Delta'$ is ample over $C$. 
To analyze the singularities of $(X',\Delta')$, let $\Delta'^+$ denote the birational transform of $\Delta^+ $ on $X'$.
Since $A_{X,\Delta^+}(E)=0$ and $K_{X}+\Delta^+ \sim_{\bQ}0$, the pairs 
\[
(X,\Delta^+) \quad \text{ and } \quad (X', \Delta'^+ +X'^{\rm red}_0)
\]
are crepant birational. 
Therefore  $\Delta'^+ + X'^{\rm red}_0$ is a complement of $(X',\Delta')$.
Since $(X',\Delta'^+ + X'^{\rm red}_0)$ is lc and $(X',\Delta')$ is klt on $X\setminus X_0$, the pair $(X',\Delta')$ is klt. 
Therefore $(X',\Delta')$ is a log Fano pair over $C$.

It remains to verify that (1) and (2) hold. 
Condition (1) holds by construction.
To verify that (2) holds,
fix a divisor $F$ over $X'$ such that $C_{X'}(F)\neq X'^{\rm red}_0$.
To show that  $A_{X',\Delta'}(F)>0$, fix an effective $\bQ$-divisor $H'\sim_{\bQ}-K_{X'}-\Delta'$ such that $C_{X'}(F') \subset H'$ and $X'^{\rm red}_0\not\subset \Supp(H')$. 
Let $H $ denote the birational transform of $H'$ on $X$.
By Lemma \ref{l:complement>=given divisor}, there exist $c, t>0$ and  a complement $\Lambda^+$ of $(X, \Delta -t X_0+ c H) $ with lc place along $E$.
Let $\Lambda'^+$ denote the birational transform of $\Lambda^+$ on $X'$. 
Since $A_{X,\Lambda^+}(E)=0$, 
\[
(X,\Lambda^+)
\quad \text{ and }\quad 
(X',\Lambda'^++X'^{\rm red}_0)
\] 
are crepant birational and so 
 $(X',\Lambda'^++X'^{\rm red}_0)$ is lc. 
Now observe that
\[
A_{X',\Delta'+X'^{\rm red}_0}(F) \ge 
A_{X',\Delta'+cH'+X'^{\rm red}_0}(F) 
> A_{X',\Lambda'^++ X'^{\rm red}_0}(F) \geq 0
.\]
Therefore $(X',\Delta'+X'^{\rm red}_0)$ is plt, which verifies (2).
\end{proof}

The existence of the following model can be easily deduced from Proposition \ref{p:Fanoreplacement}. 

\begin{prop}\label{p:extractmin}
The exists a proper birational morphism $g:Y\to X$ such that $Y$ is normal, $E_Y := {\rm Exc}(g)$ is a prime divisor that is the birational transform of $E$ on $Y$, and $-E_Y$ is ample over $X$.
\end{prop}

\begin{proof}
We return to the morphism  $Z\to X$ constructed in the first paragraph of Proposition \ref{p:Fanoreplacement} and the $\bQ$-divisor $G$.
Since $(Z,G)$ is klt, 
\[
-E_Z - (K_Z+G_Z) 
\sim_{X,\bQ} 
-E_Z
,\]
and $-E_Z$ is nef and big over $X$, the base-point free theorem implies that $-E_Z$ is semiample over $X$.
Let $Z\to Y$ denote the morphism over $X$ to the relative ample model of $-E_Z$ and write $g$ for the morphism $Y\to X$ . Note that $E_Z \equiv_{Y} 0$ by construction. 
The morphism $Z\to Y$ cannot contract $E_Z$ as otherwise the negativity lemma would imply that $E_Z=0$, which is absurd. 
Let $E_Y$ denote the birational transform of $E_Z$ on $Y$. 
Since $Y\to X$ contracts $E_Y$, $E_Y \subset {\rm Exc}(g)$. 
Since $-E_Y$ is ample over $X$, ${\rm Exc}(g) \subset E_Y$
and so the result holds.
\end{proof}

\subsection{Finite Base change}\label{ss:basechange}
We now consider base changes of the above pairs constructed in Section \ref{ss:birmodels} via a finite morphism of spectrum of DVRs. 
The base change is needed to replace  $(X',\Delta')\to C$ with a model with reduced special fiber.

Let $A:=\cO_C(C)$, which is a DVR, and fix a uniformizer $\pi \in A$. 
Let $v=\ord_E$ be as before and let $B:= A[\pi^{1/d}]$, where $d:= v(X_0)= v(\pi)$. 
Note that $B$ is a DVR with uniformizer $\pi^{1/d}$ and has residue field isomorphic to $\bk$. 
The finite extension $A\hookrightarrow B$ 
induces a finite morphism 
\[
\oC:= \Spec(B)  \to \Spec(A)= C
.\]
We write $0$ for the closed point of $\oC$; this is slightly abusive notation as $0$ also denotes the closed point of $\oC$.

We now apply a base change to define a relative sub-pair and pair
\[
(\oX,\oDe) \to \oC
\quad \text{ and } \quad
(\oX',\oDe')\to \oC
.\]
First, let $\oX:= X\times_C \oC$ and write $\rho: \oX \to X$ for the first projection, which is finite.
Since $\oX\to \oC$ is flat with normal fibers and $\oC$ is normal, $\oX$ is normal.
Let $\oDe$ be the $\bQ$-divisor on $\oX$ satisfying 
\[
K_{\oX}+ \oDe=\rho^*(K_X+\Delta)
.
\]
Next,  let 
$
\oX'\to X'\times_C \oC
$
denote the normalization morphism and write $\rho':\oX'\to \oC$ for the composition of the normalization and  projection morphisms. 
Let $\oDe'$ be the $\bQ$ divisor on $\oX'$ satisfying
\[
K_{\oX'}+ \oDe' +\oX'^{\rm red}_{0}=\rho'^*(K_{X'}+\Delta' +X'^{\rm red}_0)
.
\] 
Note that the definition of $\oDe'$ is chosen so that $\Supp(\oDe')$ does not contain irreducible components of $\oX'_0$. 

The isomorphism $(X,\Delta)_{C\setminus 0}\simeq (X',\Delta')_{C\setminus 0}$ over $C\setminus 0$ induces an isomorphism 
\[(\oX,\oDe)_{\oC\setminus 0}\simeq (\oX',\oDe')_{\oC\setminus 0}
\]over $\oC\setminus 0$.
By the proposition below,  $\oX'_0$ is irreducible. 
Thus we may view $\oE: =\oX'_0$ as a prime divisor over $\oX$.


\begin{prop}\label{p:logFano}
The following assertions hold:
\begin{enumerate}
\item $(\oX,\oDe) \to \oC$ is a relative log Fano sub-pair and
\item $(\oX',\oDe')\to \oC$ is a family of log Fano pairs.
\end{enumerate}
Furthermore, both $\oX_0$ and $\oX'_0$ are reduced and irreducible.
\end{prop}

Note that $\oDe$ may fail to be effective as the Hurwitz formula implies that
\[
1-{\rm coeff}_{\oX_0}(\oDe )= d (1- {\rm coeff}_{\oX_0}(\Delta) )
.\]
In contrast, $\Supp(\oDe')$ does not contain $\oX'_0$ in its support by construction.

\begin{proof}
We first analyze the fibers over $0 \in \oC$.
The divisor $\oX_0$ is reduced and irreducible as it is isomorphic to $X_0$. 
Since the ramification index of $\oC \to C$ over $0$ equals the multiplicity of $X'_0$, 
\cite[Lemma 2.53]{Kol23} implies that $\oX'\to \oC$ has reduced fibers.

Now, we verify that (1) and (2) hold.
Since $\rho$ and $\rho'$ are finite and $-K_{X}-\Delta$ and $-K_{X'}-\Delta'$ are ample over $C$, $-K_{\oX}-\oDe$ and $-K_{\oX'}-\oDe'$ are both ample over $\oC$. 
Since $(X,\Delta)$ is klt and $(X',\Delta'+X_0^{\rm red})$ is plt,
\cite[Corollary 2.43]{Kol13} implies that 
$(\oX,\oDe)$  is klt  and $(\oX',\oDe'+\oX'_0)$ is plt.
Thus (1) and (2) hold. 
Furthermore, the plt condition implies that $\oX'_0$ is irreducible.
\end{proof}

The above sub-pairs admit $\mu_d$ actions defined as follows.
There is a natural $\mu_d$-action on $\oC$ induced by the action of  $\zeta \in \mu_d \subset \bk^\times$  on  $B:=A[\pi^{1/d}]$ which is trivial on $A$ and multiplication by $\zeta$ on $\pi^{1/d}$.
The ring $B$ admits a direct sum decomposition
\[
B= 1 A \oplus \pi^{1/d} A \oplus \cdots \oplus \pi^{(d-1)/d}A
,\]
where $\zeta \in \mu_d$ acts on $\pi^{i/d}A$ via multiplication by $\zeta^i$.
Let $\mu_d$ act on 
\[
X\times_C \oC
\quad  \text{and } \quad X'\times_C \oC
\] 
as the product of the trivial action and the above action on $\oC$.
This induces  $\mu_d$-actions on $\oX$ and $\oX'$ that fix $\oDe$ and $\oDe'$.

Next, we analyze the filtrations of the relative section rings induced by $E$ and $\oE$. 
Let
\[
L:=-K_X-\Delta, \quad \oL := -K_{\oX}+\oDe, \quad \text{ and } \quad \oL':= -K_{\oX'}-\oDe'
.\]
Fix an integer $r>0$ such that $rL$, $r\oL$, and $r\oL'$ are Cartier divisors. 
We consider the relative section rings
\[
\cR:= \cR(X,L)
\quad \text{ and }\quad 
\ocR:= \cR(\oX,\oL)
\]
and the section rings on the special fibers
\[
R:=R(X_0,L_0)
\quad \text{ and }\quad 
\oR:= R(\oX_0,\oL_0)
.\]
By convention, the rings are indexed by $r\bN$.
Let $\cF_E$ and $\cF_{\oE}$ denote the respective filtrations of $\cR$ and $\ocR$ induced by $E$ and $\overline{E}$.
We additionally write 
\[
\cG:= \cF_E \vert_{X_0}
\quad \text{ and }\quad
\ocG:=\cF_{\oE}\vert_{\oX_0}
\]
for the filtrations of $R$ and $\oR$ defined via restriction.
To compare the various filtrations, note that $R^if_* \cO_X(mL)=0$ by Kawamata-Viehweg vanishing. 
Therefore 
$H^0(\oX, m\oL) \simeq H^0(X,mL)\otimes_A B $ by the cohomology and base change theorem.
 The latter isomorphism induces a $\mu_d$-equivariant isomorphism 
\begin{equation}\label{e:ocR}
\ocR_m \simeq \cR_m \oplus \pi^{1/d} \cR_m \oplus \cdots \oplus \pi^{ (d-1)/d}\cR_m 
,
\end{equation}
where $\zeta\in \mu_d \subset \bk^\times$ acts on the summand $\pi^{i/d} \cR_m$ via multiplication by $\zeta^i$.
Additionally, the isomorphism 
 $(X_0,\Delta_0)\simeq (\oX_0,\oDe_0)$ induces an isomorphism of graded rings $R\simeq \oR$.

\begin{prop}\label{p:comparefilt}
The following hold:
\begin{enumerate}
\item The isomorphism \eqref{e:ocR} sends the subspace $\cF_{\oE}^\la \ocR_m \subset  \ocR_m$ to the subspace
\[
\cF_E^{\la/d} \cR_m \oplus \pi^{1/d} \cF_E^{(\la -1)/d}\cR_m \oplus \cdots \oplus \pi^{ (d-1)/d} \cF_E^{(\la  -(d-1))/d}\cR_m
.\]
\item The isomorphism $ \oR_m\to R_m$ sends $\ocG^\la R_m$ to $\oG^{\la/d}R_m$.
\end{enumerate}
\end{prop}

\begin{proof}
The $B$-module $\ocR_m$ admits a direct sum decomposition 
\[
\ocR_m = \ocR_m^{(0)} \oplus \cdots \oplus \ocR_m^{(d-1)}
\]	
into $\mu_d$-eigenspaces, where $\zeta \in \mu_d$ acts on $\ocR_m^{(i)}$ via multiplication by $\zeta^i$.
Under this decomposition, the isomorphism \eqref{e:ocR} sends $\ocR_m^{(i)} $ to $\pi^{i/d} \cR_m$.
Since the $\mu_d$-action on $\oX'$ fixes $\oX'_0$, the valuation $\ord_{\overline{E}}$ is $\mu_d$-invariant.
Therefore the filtration $\cF_{\oE}$ of $\ocR$ is $\mu_d$-invariant, which means that the  $\mu_d$-action on $\ocR_m$ fixes each subspace $\cF_{\oE}^\la \ocR_m$.
Therefore 
\[
\cF_{\oE}^\la \ocR_m 
=
(\cF_{\oE}^\la \ocR_m   \cap \ocR_m^{(0)})
\oplus \cdots 
\oplus 
(\cF_{\oE}^\la \ocR_m   \cap \ocR_m^{(d-1)})
.\]
To prove (1),  it suffices to show that the isomorphism $\ocR_m^{(i)} \to \pi^{i/d} \cR_m$ sends $(\cF_{\oE}^\la \ocR_m   \cap \ocR_m^{(i)})$
to $\pi^{i/d}\cF_E^{(\la -i)/d}\cR_m$. 
This holds, since if $\overline{s}= \pi^{i/d} \rho^*(s) \in \ocR_m^{(i)}$, where $s\in \cR_m$, then 
\[
\ord_{\overline{E}}(\pi^{i/d}\rho^*s) 
=
i + \ord_{\overline{E}}(\rho^*s) 
=
i + d \cdot  \ord_{E}(s).\]
Therefore  (1) holds. 
Statement (2) follows immediately from (1) as the composition $\ocR_m \to \cR_m \to R_m$ sends an element $\sum_{i=0}^{d-1} \pi^{i/d} \rho^*s_i \in \ocR_m$  to $s_0\vert_{X_0} \in R_m$.
 \end{proof}
The next proposition will play a key role in the proof of  Lemma  \ref{l:subtleBJinequality}.

\begin{prop}\label{p:weaklyspecialtc}
The $\bN$-graded $\bk[x]$-algebra ${\rm gr}_\cG R $ is finitely generated and the scheme ${\rm Proj}\, {\rm gr}_\cG R$ is reduced.
\end{prop}

\begin{proof}
Consider the Rees algebras
\[
{\rm Rees}(\ocG):=  \bigoplus_{m \in r\bN} \bigoplus_{ \la \in \bZ} \ocG^\la \oR_m t^{-\la}
\quad \text{ and } \quad 
{\rm Rees}(\cF_{\oE}):= \bigoplus_{m \in r\bN} \bigoplus_{ \la \in \bZ} \cF_{\oE}^\la \ocR_m t^{-\la} 
.\]
The first  has the structure  of a  $\bk[t]$-algebra, while the second 
has the structure of a  $B[s,t]/(st-\pi^{1/d})$-algebra, where $s$ and $t$ act  on $f t^{-\la}$, with $f \in \ocR_m$  and $\la \in\bZ$, by
\[
s \cdot  (f t^{-\la} )=  \pi^{1/d} f t^{-\la-1}  
\quad \text{ and } \quad 
t \cdot (f t^{-\la}) =  f t^{-\la+1}
.\]
We claim that there is an isomorphism 
\[
{\rm Rees}(\cF_{\oE}) / (s,t) {\rm Rees}(\cF_{\oE})  \simeq {\rm gr}_{\ocG} R.
\]
of $r\bN$-graded algebras. 
Indeed, first there is an isomorphism 
\begin{equation}\label{e:Rees/s}
{\rm Rees}(\cF_{\oE})/
s {\rm Rees}(\cF_{\oE}) 
\simeq {\rm Rees}(\ocG)
\end{equation}
of $\bN$-graded $\bk[t]$-algebras,
since  we have canonical isomorphisms
\[
\frac{{\rm Rees}(\cF_{\oE})}
{s {\rm Rees}(\cF_{\oE})}
\simeq 
 \bigoplus_{m \in r\bN} \bigoplus_{ \la \in \bZ} \frac{\cF_{\oE}^\la \ocR_m t^{-\la} }{ \pi^{i/d} \cF_{\oE}^{\la-1} \ocR_mt^{-\la}}
\]
and 
\[
\frac{\cF_{\oE}^\la \ocR_m }{ \pi^{1/d} \cF_{\oE}^{\la-1} \ocR_m}
\simeq 
\frac{\cF_{\oE}^\la \ocR_m }{ \pi^{1/d} \ocR_m \cap  \cF_{\oE}^\la \ocR_m}
\simeq 
\im (\cF_{\oE}^\la \cR_m \to \ocR_m /\pi^{1/d} \ocR_m)
\simeq 
\ocG^\la \oR_m,
.\]
In addition, there are isomorphisms
\begin{equation}\label{e:Rees/t}
{\rm Rees}(\ocG)/ t {\rm Rees}(\ocG) \simeq {\rm gr}_{\ocG} \oR
\simeq {\rm gr}_\cG R
\end{equation}
where the second holds by Proposition \ref{p:comparefilt}. 
Therefore \eqref{e:Rees/s}  and \eqref{e:Rees/t} imply the claim.

We will proceed to use results on boundary polarized CY pairs in  \cite{ABB+} to analyze ${\rm Rees}(\cF_{\oE})$.
By  Assumption \ref{assu:logFano}, there exists a complement $\Delta^+$ of $(X,\Delta)$ such that $A_{X,\Delta^+}(E)=0$ and  $\Delta^+ \geq X_0$.
Define a $\bQ$-divisor $\overline{\Delta}^+$ by the formula 
\[
K_{\overline{X}} + \oDe^+ = \rho^*(K_X+\Delta^+) 
.
\]
Note that $(\oX,\overline{\Delta}^+)$ is an lc pair by \cite[Corollary 2.43]{Kol13} and, hence, $\oDe^+$ is a complement of $(\oX,\oDe)$. 
Furthermore, $A_{\oX,\oDe^+}(\overline{E})=A_{X,\Delta^+}(E) = 0$ 
and $\oDe^+ \geq \overline{X}_0$.
Now let $\oDe'^+$ denote the birational transform of $\oDe^+$ on $\oX'$. 
Since $K_{\oX}+\oDe^+\sim_{\bQ}0$ and $A_{\oX,\oDe^+}(\oE)=0$, the pairs
$(\oX,\oDe^+)$ and $ (\oX',\oDe'^++\oX'_0)$ are crepant birational.
Thus $(\oX',\oDe'^++\oX'_0)$ is lc and $K_{\oX'}+\oDe'^++ \oX'_0\sim_{\bQ}0$.
Now let
\[
\overline{B}:= \overline{\Delta} - {\rm coeff}_{\overline{X}_0}(\oDe)
\quad \text{and }  \quad 
\overline{B}': =\oDe' 
.\]
Define $\bQ$-divisors $\oD$ and $\oD'$ on $\oX$ and $\oX'$ by the formulas
\[
\oB+\oD +\oX_0= \oDe^+  \quad \text{ and } \quad 
\oB'+\oD'+ \oX_0 = \oDe'^+
.\]
In particular, $\oD$ and $\oD'$ are effective and horizontal over $\oC'$. 
Now observe that 
\begin{enumerate}
	\item[(i)] $-K_{\oX}-\oB$ and $-K_{\oX'}-\oB'$ are ample over $\oC'$;
	\item[(ii)] $K_{\oX}+\oB+\oD\sim_{\bQ}0$ and $K_{\oX'}+\oB'+\oD'\sim_{\bQ}0$;
	\item[(iii)] $(\oX,\oB+\oD+\oX_0)$ and $(\oX',\oB'+\oD'+\oX'_0)$ are lc.
\end{enumerate}
Therefore $(\oX,\oB+\oD)\to \oC$ and $(\oX',\oB+\oD')\to \oC$ are families of boundary polarized CY pairs in the terminology of \cite[Definition 2.9]{ABB+}. 
By \cite[Proposition 5.8]{ABB+}, ${\rm Rees}(\cF_{\oE})$
is a finitely generated $B[s,t]/(st-\pi^{1/d})$ algebra.
Thus ${\rm gr}_\cG R$ is a finitely generated $\bk$-algebra.
Now consider
\[\mathfrak{X} := {\rm Proj} \, {\rm Rees}(\cF_{\oE})
\to \Spec \, B[s,t]/(st-\pi^{1/d})
.\]
By \cite[Proof of Theorem 5.4]{ABB+}, $(\fX,\fB)\to S$ is a  family of lc pairs and, in particular, the fibers are reduced.
Therefore the fiber $\fX_{s=t=0}$, which is isomorphic to  ${\rm Proj}\, {\rm gr}_\cG R$, is reduced.
\end{proof}
\subsection{Threshold along the minimizer}
In this subsection, we introduced an invariant that measures singularities of basis type divisors along the minimizer $\ord_E$. 
This invariant will be used  to construct certain complements adapted to basis type divisors compatible with the minimizer.

\subsubsection{Definition and lower bound}
Throughout the following discussion, let 
\[
 E_Y\subset Y \overset{g}{\to}
 X
\]
denote the proper birational morphism  constructed in Proposition \ref{p:extractmin} with the property that $E_Y$, which   denotes the birational transform of $E$ on $Y$,  is the sole exceptional divisor of $g$  and $-E_Y$ is ample over $X$. 
Recall that $L:= -K_X-\Delta$ and $r$ is a fixed positive integer such that $rL$ is a Cartier divisor.

\begin{defn}
For each  integer $m>0$ divisible by $r$, we define $\la_m(X,\Delta;E)$ to be  
\[
\la_m(X,\Delta;E) 
:=\\
 \sup \left\{ \la \in \bQ 
\, \left\vert\, 
\begin{tabular}{c}
\mbox{$(Y,g_*^{-1}(\Delta+ \la D) +E_Y )$ is lc for all $m$-basis} \\  
\mbox{type divisors $D$  of $L$ compatible with $E$} 
\end{tabular}
\right.\right\} 
.
\]
We set 
\begin{eqnarray}\label{eq-lambda}
\la(X,\Delta;E) = \liminf_{m \to \infty} \la_{mr}(X,\Delta,E)
.
\end{eqnarray}
A priori, the value $\lambda(X,\Delta)$ could depend on the choice of the positive integer $r$ above.
\end{defn}

The following statement is the main result of this section.

\begin{prop}\label{p:lowerboundlambda}
After possibly replacing $r>0$ with a positive multiple, the following inequality holds:
\[
\la(X,\Delta;E) \geq \delta(X,\Delta).
\]
\end{prop}

The proof of the inequality mirrors the proof that the stability has a valuative interpretation, but is significantly more subtle.
Before proving the inequality, we give a lower bound for the non-asymptotic version.

\begin{lem}\label{l:la_minequality}
For each  integer $m>0$ divisible by $r$,
\[
\la_m(X,\Delta;E) \geq 
\inf_{w\in \DivVal_X}
\frac{A_{X,\Delta}(w) - w(E_Y) A_{X,\Delta}(E)}{S_m(w) - w(E_Y)S_m(E)},
\] 
Furthermore,   the  numerator  and  denominator of the fraction inside the infimum are always non-negative.
\end{lem}

In the lemma, we use the convention that the fraction inside the infimum is $+\infty$ when the denominator is zero.

\begin{proof}	
First, we verify that  $(Y, g_*^{-1}(\Delta) +E_Y)$ is lc. 
By Assumption \ref{assu:logFano}, 
$E$ is an lc place of some complement $\Delta^+$ of $(X,\Delta)$. 
Therefore $(Y, g_*^{-1}(\Delta^+)+E_Y)$ is crepant birational to $(X,\Delta^+)$ and, in particular, is lc.
As $\Delta^+\geq \Delta$, it follows that $(Y, g_*^{-1}(\Delta)+E_Y)$ is lc.
	
Next, we prove that the numerators and denominators of the fraction inside the infimum are always non-negative. 
To analyze the numerator, note that
\[
K_{Y} +g_*^{-1}(\Delta)+ E_Y = g^*(K_X+\Delta)  + A_{X,\Delta}(E) E_Y
.\]	
Thus, for $w\in \DivVal_X$,
\[
A_{X,\Delta}(w) 
=
A_{Y,g_*^{-1}(\Delta)+E_Y}(w) + A_{X,\Delta}(E)w(E_Y)
.\]
As $(Y,g_*^{-1}(\Delta)+E_Y)$ is lc by the first paragraph, it follows that 
$A_{X,\Delta}(w)-A_{X,\Delta}(E)w(E_Y)$ is non-negative.
To analyze the denominator, let $D$ be a relative $m$-basis type divisor of $(X,\Delta)$ compatible with $E$. Then 
\[
g^*D=g_*^{-1}(D)+S_m(E)E_Y\,
,\]  which implies that 
\[
0\le w(g_*^{-1}(D))= w(g^*D)-S_m(E)w(E_Y)
=
w(D) - S_m(E) w(E_Y).
\]
As $S_m(w) \geq w(D)$, we conclude that the denominator is also always non-negative.

To prove the main inequality, fix a relative $m$-basis type divisor  $D$  compatible with $\ord_E$ and $\la\in \bQ_{\geq0}$. 
Since
\[
K_Y +g_*^{-1}(\Delta+\la D)+E_Y
=
g^*(K_X+\Delta+\la D) +A_{X,\Delta+\la D}(E)E_Y
,\]
for any $w\in \DivVal_X$, we have
\begin{multline*}
A_{Y,g_*^{-1} (\Delta+\la D)+E_Y}(w)
=
A_{X,\Delta+\la D}(w) -w(E_Y)A_{X,\Delta+\la D}(E)\\
=
A_{X,\Delta}(w) 
-
w(E_Y)A_{X,\Delta}(E)
- \la (w(D)  -w(E_Y) \ord_E(D))
\\
\geq
A_{X,\Delta}(w) 
-
w(E_Y)A_{X,\Delta}(E)
- \la (S_m(w)  - w(E_Y)S_m(E)),
\end{multline*}
where the last line uses that $S_m(w)\geq w(D)$ and $S_m(E)=\ord_E(D)$ as $D$ is compatible with $E$ by assumption.
Thus the pair $(Y, g_*^{-1}(\Delta+\la D)+E_Y)$ is lc when
\[
\la \leq \frac{A_{X,\Delta}(w)-w(E_Y)A_{X,\Delta}(E)}
{S_m(w)-w(E_Y)S_m(E)}
\]
for all $w\in \DivVal_X$. 
Therefore the desired inequality holds.
\end{proof}

We now deduce Proposition \ref{p:lowerboundlambda} from Lemma \ref{l:la_minequality} and a technical result that will appear in the next section. 

\begin{proof}[Proof of Proposition \ref{p:lowerboundlambda}]
We first replace $r$ with a positive multiple satisfying Lemma \ref{l:subtleBJinequality}.  
Fix $\varepsilon>0$ and let $m_0:= m_0(\varepsilon)$ be a positive integer satisfying the conclusion of Lemma \ref{l:subtleBJinequality} stated and proven in the next subsection. 
For $m\geq m_0$ and divisible by $r$, we have that
\[
\la_{m}(X,\Delta;E)
\geq 
\inf_w \frac{A_{X,\Delta}(w) - w(E_Y) A_{X,\Delta}(E)}{S_m(w) - w(E_Y)S_m(E)}
\geq
\frac{1}{1+\varepsilon}\inf_w 
\frac{A_{X,\Delta}(w) - w(E_Y) A_{X,\Delta}(E)}{S(w) - w(E_Y)S(E)}.
\]
by Lemmas  \ref{l:la_minequality} and \ref{l:subtleBJinequality}.
Note that 
\[
\frac{A_{X,\Delta}(w)}{S(w)}
\geq \delta(X,\Delta)
=
\frac{A_{X,\Delta}(E)}{S(E)}
\]
by Proposition \ref{p:delta=inf A/S} and the assumption that $E$ computes $\delta(X,\Delta)$. 
Thus 
\[
A_{X,\Delta}(w) - w(E_Y)A_{X,\Delta}(E)
\geq 
\delta(X,\Delta) \big( S(w) - w(E_Y)S(E)\big)
\]
for any $w\in \DivVal_X$.
So
\[
\inf_w\frac{A_{X,\Delta}(w) - w(E_Y)A_{X,\Delta}(E)}{ S(w) - w(E_Y)S(E)} \ge \delta(X,\Delta)
\]
Combining this with the first inequality gives that 
\[
\la_{m}(X,\Delta;E)
\geq \frac{1}{1+\varepsilon}\delta(X,\Delta)
\]
for $m\geq m_0$ divisible by $r$. 
Since $\varepsilon >0$ was arbitrary, it follows that 
\[
\la(X,\Delta;E)
\geq \delta(X,\Delta)
\]
 as desired.
\end{proof}

\subsubsection{Convergence result}\

\begin{lem}\label{l:subtleBJinequality}
After possibly replacing $r>0$ with a positive multiple, the following holds.
For each $\varepsilon>0$, there exists $m_0:=m_0(\varepsilon)$ such that 
\[
S_m(w)- w(E_Y) S_m(E)  \leq  (1+\varepsilon)(S(w)- w(E_Y) S(E))
\]
for all $m \in r \bN$ with $m\geq m_0$   and $w\in \DivVal_X$.
\end{lem}

While we would like to deduce the lemma as a simple consequence of Proposition \ref{prop-comparetwofiltration}, we cannot as the restricted filtration $\cF_E\vert_{X_0}$ is not necessarily induced by a divisor over $X_0$. 
We will instead degenerate $X_0$ to a projective reduced scheme on which we can apply Proposition \ref{prop-comparetwofiltration} to the irreducible components.

\begin{proof}
First observe that by Proposition \ref{prop-reducedBJinequality}
there exists $m_1:=m_1(\varepsilon)$ such that
\[
S_m(w)- w(E_Y) S_m(E)  \leq  (1+\varepsilon)(S(w)- w(E_Y) S(E))
\]
 holds for all $m \in r\bN$  with $m \geq m_1$ and $w\in \DivVal_X$ with $w(E_Y)=0$. 
Thus it remains to prove a similar convergence statement for $w\in \DivVal_X$ with  $w(E_Y)>0$.

Let $\cG:= \cF_E\vert_{X_0}$ denote the filtration of $R:=R(X_0,L_0)$ defined by restricting $\cF_E$ to a filtration of $R$.
Consider the graded $\bk$-algebra
\[
{\rm gr}_\cG R  
:= \bigoplus_{m \in r\bN} {\rm gr}_\cG R_m
:= \bigoplus_{m \in r\bN} \bigoplus_{\la \in \bZ} {\rm gr}_\cG^\la R_m
.\]
We define a  filtration of ${\rm gr}_\cG R$   by
\[
\cG^\mu ({\rm gr}_\cG R_m ) =  \bigoplus_{\la =\lceil  \mu \rceil }^{+\infty  }  {\rm gr}^\la_\cG R_m
\]
that we a bit abusively also denote by $\cG$.
Note that $\dim \cG^\la R_m  = \dim \cG^\la  ({\rm gr}_\cG R_m)$ and so 
$S_m(\cG, R) = S_m( \cG,{\rm gr}_\cG R)$.
We seek to verify the following key claim.

\medskip

\noindent {{\bf Claim}: After possibly replacing $r>0$ with a positive multiple, the following holds. For any $\varepsilon>0$, there exists an integer $m_2:=m_2(\varepsilon)>0$ such that
\[
S_m(\cF, {\rm gr}_\cG R ) - S_m(\cG , {\rm gr}_\cG R ) \leq (1+\varepsilon)( S(\cF, {\rm gr}_\cG R) - S(\cG, {\rm gr}_\cG R) )
\]
for all $m \in r\bN$ with $m\geq m_2$  and linearly bounded filtrations $\cF$ of ${\rm gr}_\cG^R$ with $\cG\subset \cF$.}
\medskip

\begin{proof}[Proof of Claim]	
First, we relate  ${\rm gr}_\cG R$ to the section ring of a reduced projective scheme.
By Proposition \ref{p:weaklyspecialtc}, ${\rm gr}_\cG R$ is a finitely generated $\bk$-algebra
and 
\[
Z:={\rm Proj}\, {\rm gr}_\cG R\,
\]
is reduced. 
After replacing $r$ with a positive multiple, we may assume that $\cO_Z(r)$ is a line bundle and write $H:= \cO_Z(r)^{\otimes 1/r}$, which is a $\bQ$-line bundle. 
After further replacing $r$ with a positive multiple, we may  assume that the natural map
\[
{\rm gr}_\cG R \to R(Z,H) 
\]
is an isomorphism of graded rings.
Note that the $\bZ$-grading on ${\rm gr}_\cG R$ induces a $\bG_m$-action on $Z$ and a $\bG_m$-linearization of $\cO_Z(1)$ such that  ${\rm gr}_\cG R \to R(Z,H)$ restricts to an isomorphism 
\[
 {\rm gr}_\cG^\la R_m \to R_m(Z,H)_{\la}  
,\]
 where $ R_m(Z,H)_{\la}$ is the $\la$-weight space of the $\bG_m$-action on $R_m(Z,H)$.
Since there is a natural isomorphism ${\rm gr_\cG} R\cong R(Z,H)$, 
we will refer to filtrations of ${\rm gr}_\cG R$ and $R(Z,H)$ without distinction. 
In particular, we write $\cG$ for the filtration $R(Z,H)$ such that 
\[
\cG^\mu R_m(Z,H) =  \bigoplus_{\la =\lceil  \mu \rceil }^{+\infty  } R_m(Z,H)_\la
.\]
 
Next, we consider the decomposition
 $
 Z = Z_1 \cup \cdots  \cup Z_s,
 $
 where each $Z_i$ is a distinct irreducible component of $Z$, and set
  $H_i := H\vert_{Z_i}$. 
 After replacing $r>0$ with a further multiple, we may assume that each restriction map
 \[
r_i : R(Z,H)\to R(Z_i,H_i)
\]
is surjective.
The $\bG_m$-action on $Z$ and linearization of $\cO_Z(H)$
restricts to a $\bG_m$-action on $Z_i$ and linearization of $\cO_{Z_i}(H_i)$. 
Observe that 
\[
 r_i ( R_m(Z,H)_\la) = R_m(Z_i,H_i)_\la
,\]
where $R_m(Z_i,H_i)_\la$ is  the $\la$-weight space of $R_m(Z_i,H_i)$.
We define a filtration $\cG_i$ of  of $R_m(Z_i,H_i)$ by
 \[
 \cG_i^\mu R_m (Z_i,H_i):=  \bigoplus_{\la =\lceil  \mu \rceil }^{+\infty  }  R_m(Z_i,H_i)_\la.
 \]
 By \cite[Lemma 5.17]{BHJ} applied to the product test configuration induced by the $\bG_m$-action on the normalization of $(Z_i, H_i)$, there exists  a prime divisor $E_i$ 
 over $Z_i$ and $a_i,b_i\in \bQ$ with $a_i\geq 0$ satisfying
 \[
 \cG_i^\la R_m(Z_i,H_i) 
 =
 \{ s\in H^0(Z_i,mH_i) \, \vert\, \ord_{E_i}(s) \geq  a_i\la  -b_im \} 
\]
for all $m\in \bN$ and $\la \in \bR$. 

Following Section \ref{ss:reducible},  consider the graded
 linear series $V_\bullet^i$ of $H_i$ defined by 
\[
V_m^{i} := \im \left ( W_m^i  \to H^0(Z_i, mH_i) \right)
,\]
where 
\[
W_m^i := H^0(Z,\cO_Z(mH) \otimes \cI_{Z_1\cup \cdots \cup Z_{i-1 }}).
\]
Write  $\cG_i$ for the filtration of $V_\bullet^i$ defined by 
\[
\cG_i^\la V^i_m: = \cG_i^\la R_m(Z_i,H_i) \cap V_m^i 
.\]
By Proposition \ref{prop-comparetwofiltration}, there exists $M:=M(\varepsilon)>0$ such that 
\[
S_m(\cF_i)-
S_m(\cG_i)< (1+ \varepsilon)^{1/2} (S(\cF_i)-S(\cG_i))
\]
for all $i\in \{1,\ldots, s\}$,  $m\in r\bN$ with $m\geq M$, and linearly bounded filtrations of $\cF_i$ of $V_\bullet^i$ satisfying $\cG_i \subset \cF_i$.
By Proposition \ref{p:vollimit}  and Lemma \ref{l:Vicontainsampleseries}, there exists $N:=N(\varepsilon)>0$ such that 
\[
\frac{\dim V_m^i }{\dim R_m(Z,H)} 
< (1+\varepsilon)^{1/2} \frac{\vol(V_\bullet^i)}{ \vol (H)}
\]
for all $m\in r\bN$ and $m \geq N$ and $i \in \{1,\ldots, s\}$. 

We now verify that the  claim holds with $m_2:= \max\{M ,N\}$. Fix a linearly bounded filtration of $\cF$ of $R(Z,H)$ with $\cG\subset \cF$.
Let $\cF_i$ be the  induced filtration of $V_\bullet^i$  as in Section \ref{ss:reducible}. 
If $m\in r\bN$ satisfies $m \geq m_2$, then 
\begin{multline*}
S_m(\cF) -S_m(\cG)= \sum_{i=1}^s \frac{ \dim V_m^i}{ \dim  R_m(Z,H) }
  (S_m(\cF_i)-S_m(\cG_i))\\
\leq (1+\varepsilon) \sum_{i=1}^r  \frac{\vol(V_\bullet^i)}{\vol (H)} (S(\cF_i)-S(\cG_i))
= (1+\varepsilon)(S(\cF) - S(\cG)),
\end{multline*}
where the first relation  holds by Lemma \ref{l:formulafiltsonXi}, the second  by the assumption that $m \geq m_2$, and the  third by Lemma \ref{l:formulafiltsonXi}.
Therefore the claim holds. 
\end{proof}

\medskip

We now verify that the lemma holds with $m_0 :=
\max\{m_1, m_2\}$.
Fix $w\in \DivVal_{X}$ with $w(E_Y)>0$ and set $\cF:= \mathcal{F}_w \vert_{X_0}$. 
After scaling $w$, we may assume that $w(E_Y)=1$.
Observe that $\cG\subset \cF$, since if $s\in H^0(X, mL)$ and  $c : = \ord_{E_Y}(s)$, then 
\[
w(s) = w (g_*^{-1}\{s=0\} +c E_Y) \geq c w(E_Y ) = c
.
\]
Now $\cF$ induces a filtration of ${\rm gr}_\cG(R)$ defined by 
\[
\cF^\mu  ({\rm gr}_\cG R_m) := \im \Big( \bigoplus_{\la \in \bZ} (\cF^\mu R_m \cap \cG^\la R_m) \to \bigoplus_{\la \in \bZ} {\rm gr}_\cG^\la R_m \Big)
\subset {\rm gr}_\cG R_m
.\]
We abusively also refer to this filtration as $\cF$.
By choosing a basis for $R_m$ that simultaneously diagonalizes both $\cG$ and $\cF$, we see that $\dim \cF^\la R_m = \dim \cF^\la {\rm gr}_\cG R_m$.
Therefore 
\[
S_m(w) = S_m(\cF,R) = S_m(\cF, {\rm gr}_\cG R)
\]
and 
\[
S_m(E_Y) = S_m(\cG,R) = S_m(\cG, {\rm gr}_\cG R)
.\]
Since there is an inclusion $\cG\subset \cF$ as filtrations of $R$, the same inclusion holds for the filtrations of ${\rm gr}_\cG R$ and so the claim implies that 
\[
S_m(\cF,{\rm gr}_\cG R) - S_m(\cG,{\rm gr}_\cG R)< (1+\varepsilon)(S(\cF,{\rm gr}_\cG R)-S(\cG,{\rm gr}_\cG R))
\]
for all integers $m \geq m_0$  divisible by $r$.
Combining the previous three expressions and the assumption that $w(E_Y)=1$, we deduce that 
\[
S_m(w) - w(E_Y) S_m(E_Y) < (1+\varepsilon)( S(w)-w(E_Y) S(E_Y))
\]
for all  integers $m \geq m_0$  divisible by $r$.
\end{proof}

\subsubsection{Complements to basis type divisors}
We now use Proposition \ref{p:lowerboundlambda} to prove a result on constructing complements with prescribed singularities along  $m$-basis type divisors compatible with the minimizer $E$.

\begin{prop}\label{p:complementsforbasistype+minimizer}
After possibly replacing $r>0$ with a positive multiple, the following holds:
For any $\varepsilon>0$, there exists $m_0:=m_0(\varepsilon)$  such that if  $m\geq m_0$ is divisible by $r$  and $D$ is a relative $m$-basis type divisor of $-K_X-\Delta$  compatible with $E$, then
there exists a complement $\Delta^+$ of $(X,\Delta)$ such that
\begin{enumerate}
\item  $\Delta^+\geq \Delta+ (\delta(X,\Delta)-\varepsilon) D$ and
\item $A_{X,\Delta^+}(E)=0$.
\end{enumerate} 
\end{prop}

\begin{proof}
Fix $r>0$ such that the conclusion of Proposition \ref{p:lowerboundlambda} holds. 
Fix $\varepsilon>0$.
For each positive integer $m$ divisible by $r$, set 
\[
c_m:= \min \{\delta_m(X,\Delta), \la_m(X,\Delta,E)\}
.\]
By Proposition \ref{p:delta=inf A/S} and Proposition \ref{p:lowerboundlambda}, 
\[
\lim_{m\to \infty} c_{mr} =\delta(X,\Delta)\, .
\]
Thus there exists an integer  $m_1$ such  that $c_m \geq \delta(X,\Delta)-\varepsilon$ for all $m\geq m_1$ divisible by $r$.

We claim that there exists an integer $m_2$ such that  if $m\geq m_2$ divisible by $r$ and $D$ is a relative $m$-basis type divisor of $-K_X-\Delta$ compatible with $E$, then 
\[
-K_{Y}
-
g_*^{-1}(\Delta+c_m D) - E
\]
is semiample.
To verify the claim, observe that 
\begin{multline*}
-K_Y 
	-g_*^{-1}(\Delta+c_m D) - E
	=
	-
	g^*(K_X+ \Delta+c_m D) -A_{X,\Delta+c_m D}(E)E\\
	\sim_{\bQ} 
	-
	(1-c_m)g^*(K_X+\Delta) - a_m E
	\sim_{\bQ}
	(1-c_m) \left( -g^*(K_X+\Delta) - \tfrac{a_m}{1-c_m} E\right)
,
\end{multline*}
where $a_m := A_{X,\Delta}(E)- c_m S_m(E)\ge 0$. 
Now note  that
\[
\lim_{m \to \infty}(1-c_{mr})
=
1- \delta(X,\Delta)
<1
\quad \text{ and } \quad
\lim_{m\to \infty} \frac{a_{mr}}{1-c_{mr}}
=
\frac{
	A_{X,\Delta}(E) - \delta(X,\Delta)S(E)}{1-\delta(X,\Delta)}
=
0
.\]
Using that $-K_X-\Delta$ is ample and $-E$ is $g$-ample, it then follows that 
\[
-g^*(K_X+\Delta) - \frac{a_{m}}{1-c_{m}} E
\]
is semiample for $m>0$ sufficiently large and divisible by $r$.

Set $m_0 = \max \{m_1,m_2\}$.
Fix $m\geq m_0$ divisible by $r$ and $D$  an $m$-basis type divisor of $(X,\Delta)$ compatible with $E$.
Since $m \geq m_2$,
$-K_Y-g_*^{-1}(\Delta+c_m D) - E$ is semi-ample. 
Since $c_m \geq \la_{m}(X,\Delta,E)$, the pair
\[
(Y,  g_*^{-1}(\Delta+c_m D)+E)
\]
is log canonical.
Thus we can use Bertini's Theorem to construct a complement $\Delta^+_Y$ of the latter pair.
Thus  $\Delta^+:=g_*\Delta^+_Y$ is a complement of  $(X,\Delta+c_m D)$ that satisfies condition (2).
Since $m\geq m_1$, condition (1) is also satisfied.
\end{proof}

\subsection{Bounding the stability threshold}
In this subsection, we will complete the proof of Theorem \ref{t:increasedelta}.
The following is the key proposition towards this goal.

\begin{prop}\label{p:boundlct}
After possibly replacing $r>0$ with a positive multiple, the following holds:
For any $\varepsilon>0$, there exists $m_0:=m_0(\varepsilon)$ such that if $m \in r\bN$ satisfies $m \geq m_0$ and
\[
\oD'= \tfrac{1}{mN_m} (\{s'_1 =0\}+ \cdots + \{s'_{N_m} =0\})
\]
is a relative $m$-basis type divisor of $\oL':=-K_{\oX'}-\oDe'$ compatible with $\ord_{\oX_0}$ and  
$s'_1,\ldots, s'_{N_m} \in H^0(\oX', m \oL')$ are eigenvectors with respect to the $\mu_d$-action,
then 
\[
\lct(\oX',\Delta'+\oX'_0; \oD')\geq \delta(X,\Delta)-\varepsilon.
\]
\end{prop}

To prove the result, we will relate $\oD'$ to a basis type divisor on $\oX$ that descends to $X$ and then apply Proposition \ref{p:complementsforbasistype+minimizer}.

\begin{proof}
Fix $r>0$  satisfying Proposition \ref{p:complementsforbasistype+minimizer}.
Next let $m_0:=m_0(\varepsilon)$ be a positive integer satisfying the conclusion of Proposition \ref{p:complementsforbasistype+minimizer}.
Fix an integer $m\geq m_0$ divisible by $r$ and an $m$-basis type divisor of $\oL'$
\[
\oD'=
 \tfrac{1}{mN_m} (\{s'_1 =0\}+ \cdots + \{s'_{N_m} =0\})
\]
such that the $s'_i$ are eigenvectors with respect to the $\mu_d$-action and
$\oD'$ is compatible with $\ord_{\oX_0}$.
Using the isomorphism $(\oX,\oDe)_{\oC\setminus 0}\simeq (\oX',\oDe')_{\oC\setminus 0}$, 
we may view each $s'_i$ as a rational section of $\cO_{\oX}(m\oL)$. 
Let $b'_i$ denote the order of vanishing of this rational section $s'_i$ along $\oX_{0}$ and set 
\[
	s_i := \pi^{-b'_i/d} s'_i \in H^0(\oX,m \oL) 
.\]
By Lemma \ref{lem-dualfiltration}, 
\[
\oD := \tfrac{1}{mN_m} (\{s_1=0 \}+ \cdots + \{s_{N_m}=0\})  
\]
is a relative $m$-basis type divisor of $\oL$ on $\oX$ that is compatible with $\ord_{\oE}$. 
Since the $s'_i$ are  eigenvectors with respect to the $\mu_d$-action, the $s_i$  are also eigenvectors with respect to the $\mu_d$-action. 
	
We now show that  $\oD$ is the pullback of an $m$-basis type divisor $D$ on $X$ compatible with $\ord_E$.
To see this, first consider the decomposition into eigenspaces
\[
\ocR_m =  \ocR_m^{(0)} \oplus \ocR_m^{(1)}\oplus \cdots \oplus \ocR_m^{(d-1)}
,\]
where $\zeta \in \mu_d$ acts on $\ocR_m^{(j)}$ as multiplication by $\zeta^j$.
As $s_i$ is a $\mu_d$ eigenvector, it lies in a single summand. 
As $\ocR_m^{(j)} = \pi^{j/d} \rho^* \ocR_m$ by \eqref{e:ocR} and $s_i$ does not vanish along $\oX_0$,  $s_i$ lies in  $\ocR_m^{(0)}$.
Thus $s_i = \rho^*(r_i) $ for some $r_i \in \cR_m$.
Now set 
\[
D:= \tfrac{1}{mN_m} \{r_1= 0\} + \cdots + \{ r_{N_m}=0\}
,\]
which satisfies 
$\oD= \rho^*D$.
Using that $\oD$ is compatible with $\ord_{\oE}$, Proposition \ref{p:comparefilt} implies that $D$ is compatible with $\ord_{E}$.

We now analyze the log canonical threshold of $\oD'$.
By Proposition \ref{p:complementsforbasistype+minimizer},
 there exists a complement $\Delta^+$ of $(X,\Delta)$ such that 
\[
\Delta^+ \geq \Delta+ (\delta(X,\Delta)-\varepsilon) D
\quad \text{ and  } A_{X,\Delta^+}(E)=0.
\] 
Define $\overline{\Delta}^+$ by the formula
$
K_{\oX}+\oDe^+ = \rho^*(K_X+\Delta^+)
$.
Note that $\oDe^+$ is a complement of $(\oX,\oDe)$  satisfying 
\[
\oDe^+ \geq \oDe + (\delta(X,\Delta)-\varepsilon) \oD \quad \text{ and } \quad 
A_{\oX,\oDe^+}(\oE)=0
.\] 
Let $\oDe'^+$ denote the birational transform of $\oDe^+$ on $\oX'$. 
By the previous properties,  $\oDe'^+$ is a complement of $(\oX',\oDe')$ satisfying
%
 \[
 \oDe'^+ \geq  \oDe'+(\delta(X,\Delta)-\varepsilon)\oD'+\oX'_0
 .\]
Therefore the pair
\[
(\oX',\oDe'+(\delta(X,\Delta)-\varepsilon)\oD'+\oX'_0)
\]
is lc and so 
\[
\lct(\oX',\oDe'+\oX'_0;\oD')\geq \delta(X,\Delta)-\varepsilon
\]
 as desired.
\end{proof}

We will now  deduce Theorem \ref{t:increasedelta} as a consequence of the above proposition. 

\begin{proof}[Proof of Theorem \ref{t:increasedelta}]
Consider the pair $(\oX',\oDe') \to \oC$ with isomorphism 
\[
(\oX',\oDe')_{\oC\setminus 0} \simeq (\oX,\oDe)_{\oC\setminus 0}
\]
constructed in Section \ref{ss:basechange}. 
By Proposition \ref{p:logFano}, $(\oX',\oDe')\to \oC$ is a family of log Fano pairs. 
It remains to verify that $\delta(\overline{X}'_0,\overline{\Delta}'_0) \ge \delta(X,\Delta)$. 
We may assume that $\delta(\overline{X}'_0,\overline{\Delta}'_0) < 1$  as otherwise the inequality holds trivially as $\delta(X,\Delta)<1$ by Setup \ref{assu:logFano}.

Fix $r>0$ satisfying the conclusion of Proposition \ref{p:boundlct}.
Next, fix any $\varepsilon>0$, and let $m_0 =m_0(\varepsilon)$ be an integer satisfying the conclusion of Proposition \ref{p:boundlct}. 
Note that $\mu_d$-action on $(\oX',\oDe')$ fixes $\oX'_{0}$ and, hence, induces a $\mu_d$-action on $(\oX'_{0},\oDe'_{0})$. 
By \cite[Theorem 1.1]{Z-equivariant}, there exists a $\mu_d$-invariant valuation $v' \in \DivVal_{\oX'_{0}}$ such that 
\[
\frac{A_{\oX'_{0}, \oDe'_{0}} (v')}{S(v')} \leq \delta(\oX'_0,\oDe'_0)+\varepsilon.
\]
Since $S(v')= \lim_{m} S_{mr}(v')$, after increasing $m_0$, we may assume that 
\[
\frac{A_{\oX'_{0}, \oDe'_{0}}(v')}{S_m(v')} \leq \delta(\oX'_0,\oDe'_0)+2\varepsilon.
\]
for all $m\geq m_0$ divisible by $r$. 

We now seek to use Proposition \ref{p:boundlct} to bound the left hand side of the previous equation from below. 
Fix $m \geq m_0$ divisible by $r$.
Note that the $\mu_d$-action on $(\oX',\oDe')$ induces a direct sum decomposition into weight spaces
\[
H^0(\oX',m\oL') = \bigoplus_{i=0}^{d-1}H^0(\oX',m\oL')^{(i)}
\quad \text{ and } \quad
H^0(\oX'_{0},m\oL'_{0}) = \bigoplus_{i=0}^{d-1}H^0(\oX'_{0},m\oL'_{0})^{(i)}
,\]
where $\zeta\in \mu_d$ acts on the $i$-th summand as multiplication by $\zeta^i$.
Since the  $\mu_d$ action fixes $\ord_{\oX_{0}}$ and $v'$,
the  filtrations $\cF_{\ord_{\overline{X}_0}}$  and $\cF_{v'}$ of $H^0(\oX'_0,m\oL'_0)$ respect the weight decomposition. 
Therefore there exists an  eigenbasis
$
(s'_{1,0}, \ldots, s'_{N_m,0})
$
 for 
$H^0(\oX'_0,\oL'_0)$ that is compatible with both $\cF_{\ord_{\overline{X}_0}}\vert_{\oX'_{0}}$ and $\cF_{v'}$.
By the definition of $\cF_{\ord_{\overline{X}_0}}\vert_{\oX'_{\oo}}$,
there exists an eigenbasis
$
\{s'_1,\ldots, s'_{N_m}\}
$
for  $H^0(\oX', \oL')$  compatible with $\cF_{\ord_{\overline{X}_{0}}}$ and satisfying $s'_{i}\vert_{\oX'_0} =s'_{i,0}$.
Now set 
\[
\oD':= \tfrac{1}{mN_m}( \{s'_1=0\} + \cdots + \{s'_{N_m}=0 \})
.\]
Observe that
\begin{multline*}
\delta(X,\Delta)-\varepsilon
\le
\lct(\oX',\oDe'+\oX'_{0}; \oD') 
=
\lct(\oX_{0}, \oDe'_{0}; \oD'\vert_{\oX'_0})
\leq
\frac{A_{\oX'_{0}, \oDe'_{0}}(v')}{S_m(v')}
\leq 
\delta(\oX'_{0},\oDe'_{0})+2\varepsilon
,\end{multline*}
where the first relation holds as $m\geq m_0$, the second by inversion of adjunction, and the third by the fact that $\oD'_{0}$ is an $m$-basis type divisor compatible with $v'$.
Therefore 
\[
\delta(X,\Delta) 
\leq \delta(\oX'_{0},\oDe'_{0})+3\varepsilon
.\]
As the argument holds for arbitrary $\varepsilon>0$, it follows that $\delta(X,\Delta) \leq \delta(\oX'_0,\oDe'_0)$.
\end{proof}

\subsection{Proof of main results}
We now prove the results stated in the introduction.

\begin{proof}[Proof of Theorem \ref{thm-main}] 
By Corollary \ref{cor-minimizer complement}, there exists $\varepsilon>0$ such that for any $t\in (0,\varepsilon) \cap \bQ$ the following holds: 
\begin{enumerate}
   \item $\delta(X_0,B_0)< \delta(X,B+(1-t)X_0) <1$;
    \item there exist a divisorial valuation $v= \ord_E$ with $C_X(v) \subsetneq X_0$ computing 
    \[\delta(X,B+(1-t)X_0);\]
    and  $v$ is an lc place of a $\bQ$-complement $B^+$ of $(X,B)$ with $B^+ \geq X_0$.
\end{enumerate}
In addition,  $X_0$ is reduced and irreducible as $(X,B) \to C$ is a family of log Fano pairs by assumption.
Thus the relative log Fano pair $(X,\Delta)\to C$ with $\Delta:= B+(1-t)X_0$
and $v=\ord_E$ 
satisfies the conditions of Setup \ref{assu:logFano}.
Therefore Theorem \ref{t:increasedelta} implies that there exists a  morphism  $C'={\rm Spec}(R')\to C$ induced by a finite extension of DVRs  and family of log Fano pairs $(X',B')\to C'$ with an isomorphism
\[
(X,B)\times_{C\setminus 0}  (C'\setminus 0)\simeq (X', B')\vert_{C'\setminus 0}
\]
such that the restriction of $\ord_{X'_0}$ to $K(X) \subset K(X')$ is a multiple of  $ \ord_E$
 and 
$
\delta(X,\Delta)\leq \delta(X'_0, B'_0) \,
$.
Therefore 
$\delta(X_0,B_0)
<
\delta(X'_0, B'_0)
$
as desired.
\end{proof}

\begin{proof}[Proof of Corollary \ref{cor-deltaconstant}]
By \cite{LX-special} (see also \cite[Corollary 2.50]{Xu-Kbook} for the result for pairs),
 there exist an extension of DVRs $R\to R^1$ such that  $(X_K,B_K)\times_K K^1$ extends to a family of log Fano pairs 
 \[
 (X^1, B^1 ) \to  C^1:= \Spec(R^1)
.\]
By \cite{BLX-openness} and Proposition \ref{p:deltafieldext} ,
 \[
 \min \{1,  \delta(X^1_{0}, B^1_{0})\}
\leq
 \min \{1 , \delta(X^1_{K^1}, B^1_{{K}^1})\}
.\]
We claim that if the inequality is strict, then there exists an extension of DVRs
$R^1\subset R^2$ with extension of fraction fields $K^1\subset K^2=: {\rm Frac}(R^2)$ such that $(X_K, B_K )\times_K K^2$  extends to a family of log Fano pairs 
 \[
 (X^2, \Delta^2) \to  C^2:=\Spec(R^2)
 \]
satisfying
\[
\delta(X^1_{0},B^1_{0})
 <
 \delta(X^2_{0}, B^2_{0}) 
.
\]
Indeed, if the residue field $\bk^1$ of $R^1$ is algebraically closed, then this is an immediate consequence of Theorem \ref{thm-main}.
If not, then consider the extensions of DVRs
\[
R^1 \hookrightarrow \widehat{R^1} \cong \bk^1[\![t]\!] \hookrightarrow \overline{\bk^1}[\![t]\!] =: R'^1.
\]
If the inequality is strict, then, using  Proposition \ref{p:deltafieldext}, we may apply  Theorem \ref{thm-main}.1 to  
\[
(X^1, B^1) \times_{C^1}C'^1 \to C'^1 := \Spec(R'^1)
\] 
to get an extension of DVRs $R'^1 \subset R^2$ and an extension $(X^2, B^2) \to C^2:=\Spec(R^2)$ satisfying the claim. 

Continuing in this way, we obtain a sequence of extension of DVRs $R^{i-1}\subset R^{i}$ and extensions of $(X_K,B_K)\times_K K^i$ to a family  of log Fano pairs 
\[
(X^i, B^i) \to \Spec( R^i) 
\]
satisfying
\[
\delta (X^{i-1}_0, B^{i-1}_0)
<
\delta(X^i_0, B^i_0)
\quad \text{ and } \quad 
\min \{1,\delta(X^i_0, B^i_0)\}\leq
\min \{1 , \delta(X^i_{K^i}, B^i_{K^i})\},
\]
This process terminates if the last inequality ever fails to be strict. 

We claim that the process  terminates in finitely many steps.
Let $n:= \dim X_K$,  $v:= (-K_{X_K}-\Delta_K)^n$, $\delta:= \delta(X^1_0,\Delta^1_0)$, and $N$   a positive integer such that $N\Delta_K$ is integral.
Let $\mathcal{T}$ be the set of log Fano pairs $(Y,G)$ 
such that $n= \dim Y$, $v= (-K_Y-G)^n$, $\delta(Y,G)\geq \delta$, and $NG$ is integral.
Since $\mathcal{T}$ is a bounded family of pairs by \cite{Jia20,XZ-minimizer-unique}  and the function $\min \{1, \delta(\, \cdot \, )\}$ is constructible in families of log Fano pairs \cite{BLX-openness}, the set
\[
\{ \min \{1, \delta(Y,G)\}\, \vert\, (Y,G) \in \mathcal{T}\} 
\]
is finite. 
Since  each $(X_0^i,B_0^i)$ is in $ \mathcal{T}$, we conclude that the process must terminate. 
 Therefore, for some positive  integer $s$, we must have 
\[
\min \{1,
\delta(X^s_0, B^s_0)\} 
=
\min \{1 ,\delta(X^s_{K^s}, B^s_{K^s})\}.
\]
as desired. 
Furthermore, if $(X_K, B_K)$ is K-semistable, then 
\[
1
=
 \min \{1 ,\delta(X_K,B_K)\} = \min \{ 1,\delta(X^s_{K^s},B^s_{K^s})\}
 =
\min \{ 1,\delta(X^s_{0},B^s_{0})\},
\]
which the second equality is Proposition \ref{p:deltafieldext}.
Therefore $\delta(X^s_{0},B^s_{0})\geq 1$ and the fibers of $(X^s,\Delta^s) \to C^s$ are  K-semistable. 
\end{proof}

\begin{proof}[Proof of Theorem \ref{t:propKmod}]
This is special case of Corollary \ref{cor-deltaconstant}.
\end{proof}

\begin{proof}[Proof of Corollary \ref{c:propKmod}]
Let $\fX^{\rm K}_{n,V,N}$ denote the K-moduli stack defined in  \cite[Definition 7.23]{Xu-Kbook}. 
By \cite[Theorem 8.14]{Xu-Kbook}), $\fX^{\rm K}_{n,V,N}$ is a finite type algebraic stack and admits a good moduli space morphism to the K-moduli space $X^{\rm K}_{n,V,n}$, which is a finite type separated algebraic space. 
By Theorem \ref{t:propKmod}, the K-moduli stack satisfies the existence part of the version of the valuative criterion for properness allowing extensions of DVRs. 
Thus the K-moduli stack is universally closed by \cite[\href{https://stacks.math.columbia.edu/tag/0CLW}{Lemma 0CLW}]{stacks-project}.
Therefore the K-moduli space is proper by \cite[Proposition 3.48]{AHLH}.
\end{proof}

\bibliography{ref}

\begin{bibdiv}
\begin{biblist}

\bib{ABB+}{article}{
      author={Ascher, Kenneth},
      author={Bejleri, Dori},
      author={Blum, Harold},
      author={DeVleming, Kristin},
      author={Inchiostro, Giovanni},
      author={Liu, Yuchen},
      author={Wang, Xiaowei},
       title={{Moduli of boundary polarized Calabi-Yau pairs}},
        date={2023},
  note={\href{https://arxiv.org/abs/2307.06522}{\textsf{arXiv:2307.06522}}},
}

\bib{ABHLX-goodmoduli}{article}{
      author={Alper, Jarod},
      author={Blum, Harold},
      author={Halpern-Leistner, Daniel},
      author={Xu, Chenyang},
       title={Reductivity of the automorphism group of {K}-polystable {F}ano
  varieties},
        date={2020},
     journal={Invent. Math.},
      volume={222},
      number={3},
       pages={995\ndash 1032},
}

\bib{AHLH}{article}{
      author={Alper, Jarod},
      author={Halpern-Leistner, Daniel},
      author={Heinloth, Jochen},
       title={Existence of moduli spaces for algebraic stacks},
        date={2023},
     journal={Invent. Math.},
      volume={234},
      number={3},
       pages={949\ndash 1038},
}

\bib{AZ-K-adjunction}{article}{
      author={Abban, Hamid},
      author={Zhuang, Ziquan},
       title={K-stability of {F}ano varieties via admissible flags},
        date={2022},
     journal={Forum Math. Pi},
      volume={10},
       pages={Paper No. e15},
}

\bib{BCHM}{article}{
      author={Birkar, Caucher},
      author={Cascini, Paolo},
      author={Hacon, Christopher~D.},
      author={McKernan, James},
       title={Existence of minimal models for varieties of log general type},
        date={2010},
     journal={J. Amer. Math. Soc.},
      volume={23},
      number={2},
       pages={405\ndash 468},
}

\bib{BFJIzumi}{incollection}{
      author={Boucksom, S\'ebastien},
      author={Favre, Charles},
      author={Jonsson, Mattias},
       title={A refinement of {I}zumi's theorem},
        date={2014},
   booktitle={Valuation theory in interaction},
      series={EMS Ser. Congr. Rep.},
   publisher={Eur. Math. Soc., Z\"urich},
       pages={55\ndash 81},
}

\bib{BHJ}{article}{
      author={Boucksom, S\'ebastien},
      author={Hisamoto, Tomoyuki},
      author={Jonsson, Mattias},
       title={Uniform {K}-stability, {D}uistermaat-{H}eckman measures and
  singularities of pairs},
        date={2017},
     journal={Ann. Inst. Fourier (Grenoble)},
      volume={67},
      number={2},
       pages={743\ndash 841},
}

\bib{BHLNK}{article}{
      author={Bu, Chenjing},
      author={Halpern-Leistner, Daniel},
      author={Ib\'a\~nez N\'u\~nez, Andr\'es},
      author={Kinjo, Tasuki},
       title={{Intrinsic Donaldson-Thomas theory. I. Component lattices of
  stacks}},
        date={2025},
  note={\href{https://arxiv.org/abs/2502.13892}{\textsf{arXiv:2502.13892}}},
}

\bib{BHLLX-theta}{article}{
      author={Blum, Harold},
      author={Halpern-Leistner, Daniel},
      author={Liu, Yuchen},
      author={Xu, Chenyang},
       title={On properness of {K}-moduli spaces and optimal degenerations of
  {F}ano varieties},
        date={2021},
     journal={Selecta Math. (N.S.)},
      volume={27},
      number={4},
       pages={Paper No. 73, 39},
}

\bib{BJ-delta}{article}{
      author={Blum, Harold},
      author={Jonsson, Mattias},
       title={Thresholds, valuations, and {K}-stability},
        date={2020},
     journal={Adv. Math.},
      volume={365},
       pages={107062, 57},
}

\bib{BLX-openness}{article}{
      author={Blum, Harold},
      author={Liu, Yuchen},
      author={Xu, Chenyang},
       title={Openness of {K}-semistability for {F}ano varieties},
        date={2022},
        ISSN={0012-7094},
     journal={Duke Math. J.},
      volume={171},
      number={13},
       pages={2753\ndash 2797},
         url={https://doi.org/10.1215/00127094-2022-0054},
}

\bib{BLXZ-soliton}{article}{
      author={Blum, Harold},
      author={Liu, Yuchen},
      author={Xu, Chenyang},
      author={Zhuang, Ziquan},
       title={The existence of the {K}\"{a}hler-{R}icci soliton degeneration},
        date={2023},
     journal={Forum Math. Pi},
      volume={11},
       pages={Paper No. e9, 28},
}

\bib{BX-separated}{article}{
      author={Blum, Harold},
      author={Xu, Chenyang},
       title={Uniqueness of {K}-polystable degenerations of {F}ano varieties},
        date={2019},
     journal={Ann. of Math. (2)},
      volume={190},
      number={2},
       pages={609\ndash 656},
}

\bib{CP-CM-positive}{article}{
      author={Codogni, Giulio},
      author={Patakfalvi, Zsolt},
       title={Positivity of the {CM} line bundle for families of {$K$}-stable
  klt {F}ano varieties},
        date={2021},
     journal={Invent. Math.},
      volume={223},
      number={3},
       pages={811\ndash 894},
}

\bib{DonSun-GHlimits}{article}{
      author={Donaldson, Simon},
      author={Sun, Song},
       title={Gromov-{H}ausdorff limits of {K}\"ahler manifolds and algebraic
  geometry},
        date={2014},
        ISSN={0001-5962,1871-2509},
     journal={Acta Math.},
      volume={213},
      number={1},
       pages={63\ndash 106},
         url={https://doi.org/10.1007/s11511-014-0116-3},
      review={\MR{3261011}},
}

\bib{ELMNP-resvol}{article}{
      author={Ein, Lawrence},
      author={Lazarsfeld, Robert},
      author={Musta\c~t\u a, Mircea},
      author={Nakamaye, Michael},
      author={Popa, Mihnea},
       title={Restricted volumes and base loci of linear series},
        date={2009},
        ISSN={0002-9327,1080-6377},
     journal={Amer. J. Math.},
      volume={131},
      number={3},
       pages={607\ndash 651},
         url={https://doi.org/10.1353/ajm.0.0054},
      review={\MR{2530849}},
}

\bib{FO-delta}{article}{
      author={Fujita, Kento},
      author={Odaka, Yuji},
       title={On the {K}-stability of {F}ano varieties and anticanonical
  divisors},
        date={2018},
     journal={Tohoku Math. J. (2)},
      volume={70},
      number={4},
       pages={511\ndash 521},
}

\bib{Fujvalcrit}{article}{
      author={Fujita, Kento},
       title={A valuative criterion for uniform {K}-stability of
  {$\mathbb{Q}$}-{F}ano varieties},
        date={2019},
     journal={J. Reine Angew. Math.},
      volume={751},
       pages={309\ndash 338},
}

\bib{FujCouple}{article}{
      author={Fujita, Kento},
       title={{On the coupled stability thresholds of graded linear series}},
        date={2024},
  note={\href{https://arxiv.org/abs/2412.04019}{\textsf{arXiv:2412.04019}}},
}

\bib{DHL-Theta}{article}{
      author={Halpern-Leistner, Daniel},
       title={On the structure of instability in moduli theory},
        date={2014},
     journal={\texttt{arXiv:1411.0627}},
}

\bib{HMX-ACC}{article}{
      author={Hacon, Christopher~D.},
      author={McKernan, James},
      author={Xu, Chenyang},
       title={A{CC} for log canonical thresholds},
        date={2014},
        ISSN={0003-486X},
     journal={Ann. of Math. (2)},
      volume={180},
      number={2},
       pages={523\ndash 571},
         url={https://doi.org/10.4007/annals.2014.180.2.3},
      review={\MR{3224718}},
}

\bib{Jia20}{article}{
      author={Jiang, Chen},
       title={Boundedness of {$\mathbb{Q}$}-{F}ano varieties with degrees and
  alpha-invariants bounded from below},
        date={2020},
     journal={Ann. Sci. \'{E}c. Norm. Sup\'{e}r. (4)},
      volume={53},
      number={5},
       pages={1235\ndash 1248},
}

\bib{JM08}{article}{
      author={Jow, Shin-Yao},
      author={Miller, Ezra},
       title={Multiplier ideals of sums via cellular resolutions},
        date={2008},
     journal={Math. Res. Lett.},
      volume={15},
      number={2},
       pages={359\ndash 373},
}

\bib{JM-val-ideal-seq}{article}{
      author={Jonsson, Mattias},
      author={Musta\c{t}\u{a}, Mircea},
       title={Valuations and asymptotic invariants for sequences of ideals},
        date={2012},
     journal={Ann. Inst. Fourier (Grenoble)},
      volume={62},
      number={6},
       pages={2145\ndash 2209 (2013)},
}

\bib{KM-book}{book}{
      author={Koll\'{a}r, J\'{a}nos},
      author={Mori, Shigefumi},
       title={Birational geometry of algebraic varieties},
      series={Cambridge Tracts in Mathematics},
   publisher={Cambridge University Press, Cambridge},
        date={1998},
      volume={134},
        ISBN={0-521-63277-3},
         url={https://doi.org/10.1017/CBO9780511662560},
        note={With the collaboration of C. H. Clemens and A. Corti, Translated
  from the 1998 Japanese original},
      review={\MR{1658959}},
}

\bib{Kol13}{book}{
      author={Koll\'{a}r, J\'{a}nos},
       title={Singularities of the minimal model program},
      series={Cambridge Tracts in Mathematics},
   publisher={Cambridge University Press, Cambridge},
        date={2013},
      volume={200},
        ISBN={978-1-107-03534-8},
         url={https://doi.org/10.1017/CBO9781139547895},
        note={With a collaboration of S\'{a}ndor Kov\'{a}cs},
      review={\MR{3057950}},
}

\bib{Kol23}{book}{
      author={Koll\'{a}r, J\'{a}nos},
       title={Families of varieties of general type},
      series={Cambridge Tracts in Mathematics},
   publisher={Cambridge University Press, Cambridge},
        date={2023},
      volume={231},
        ISBN={978-1-009-34610-8},
        note={With the collaboration of Klaus Altmann and S\'{a}ndor J.
  Kov\'{a}cs},
      review={\MR{4566297}},
}

\bib{Laz-positivity-2}{book}{
      author={Lazarsfeld, Robert},
       title={Positivity in algebraic geometry. {II}},
      series={Ergebnisse der Mathematik und ihrer Grenzgebiete. 3. Folge. A
  Series of Modern Surveys in Mathematics [Results in Mathematics and Related
  Areas. 3rd Series. A Series of Modern Surveys in Mathematics]},
   publisher={Springer-Verlag, Berlin},
        date={2004},
      volume={49},
        note={Positivity for vector bundles, and multiplier ideals},
}

\bib{Livalcrit}{article}{
      author={Li, Chi},
       title={K-semistability is equivariant volume minimization},
        date={2017},
        ISSN={0012-7094,1547-7398},
     journal={Duke Math. J.},
      volume={166},
      number={16},
       pages={3147\ndash 3218},
         url={https://doi.org/10.1215/00127094-2017-0026},
      review={\MR{3715806}},
}

\bib{LM09}{article}{
      author={Lazarsfeld, Robert},
      author={Musta\c{t}\u{a}, Mircea},
       title={Convex bodies associated to linear series},
        date={2009},
     journal={Ann. Sci. \'{E}c. Norm. Sup\'{e}r. (4)},
      volume={42},
      number={5},
       pages={783\ndash 835},
}

\bib{Lyu-Murayama}{article}{
      author={Lyu, Shiji},
      author={Murayama, Takumi},
       title={The relative minimal model program for excellent algebraic spaces
  and analytic spaces in equal characteristic zero},
        date={2023},
}

\bib{LX-special}{article}{
      author={Li, Chi},
      author={Xu, Chenyang},
       title={Special test configuration and {K}-stability of {F}ano
  varieties},
        date={2014},
        ISSN={0003-486X},
     journal={Ann. of Math. (2)},
      volume={180},
      number={1},
       pages={197\ndash 232},
         url={https://doi.org/10.4007/annals.2014.180.1.4},
      review={\MR{3194814}},
}

\bib{LX-higher-rank}{article}{
      author={Li, Chi},
      author={Xu, Chenyang},
       title={Stability of {V}aluations: {H}igher {R}ational {R}ank},
        date={2018},
     journal={Peking Math. J.},
      volume={1},
      number={1},
       pages={1\ndash 79},
}

\bib{LXZ-HRFG}{article}{
      author={Liu, Yuchen},
      author={Xu, Chenyang},
      author={Zhuang, Ziquan},
       title={Finite generation for valuations computing stability thresholds
  and applications to {K}-stability},
        date={2022},
     journal={Ann. of Math. (2)},
      volume={196},
      number={2},
       pages={507\ndash 566},
}

\bib{Odakacones}{article}{
      author={Odaka, Yuji},
       title={{Compact moduli of Calabi-Yau cones and Sasaki-Einstein spaces}},
        date={2024},
  note={\href{https://arxiv.org/pdf/2405.07939}{\textsf{arXiv:2405.07939}}},
}

\bib{Odakanovikov}{article}{
      author={Odaka, Yuji},
       title={{Stability theory over toroidal or Novikov type base and
  Canonical modifications}},
        date={2024},
  note={\href{https://arxiv.org/abs/2406.02489}{\textsf{arXiv:2406.02489}}},
}

\bib{stacks-project}{misc}{
      author={{Stacks project authors}, The},
       title={The stacks project},
         how={\url{https://stacks.math.columbia.edu}},
        date={2025},
}

\bib{Takagi-multiplier-ideal}{article}{
      author={Takagi, Shunsuke},
       title={Formulas for multiplier ideals on singular varieties},
        date={2006},
     journal={Amer. J. Math.},
      volume={128},
      number={6},
       pages={1345\ndash 1362},
}

\bib{Tak-subadditive-pairs}{article}{
      author={Takagi, Shunsuke},
       title={A subadditivity formula for multiplier ideals associated to log
  pairs},
        date={2013},
     journal={Proc. Amer. Math. Soc.},
      volume={141},
      number={1},
       pages={93\ndash 102},
}

\bib{TianWang-conicKE}{article}{
      author={Tian, Gang},
      author={Wang, Feng},
       title={On the existence of conic {K}\"ahler-{E}instein metrics},
        date={2020},
        ISSN={0001-8708,1090-2082},
     journal={Adv. Math.},
      volume={375},
       pages={107413, 42},
         url={https://doi.org/10.1016/j.aim.2020.107413},
      review={\MR{4170229}},
}

\bib{Xu-quasimonomial}{article}{
      author={Xu, Chenyang},
       title={A minimizing valuation is quasi-monomial},
        date={2020},
     journal={Ann. of Math. (2)},
      volume={191},
      number={3},
       pages={1003\ndash 1030},
}

\bib{Xu-Kbook}{book}{
      author={Xu, Chenyang},
       title={K-stability of {F}ano varieties},
      series={New Mathematical Monographs},
   publisher={Cambridge University Press, Cambridge},
        date={2025},
      volume={50},
        ISBN={978-1-009-53877-0},
}

\bib{XZ-minimizer-unique}{article}{
      author={Xu, Chenyang},
      author={Zhuang, Ziquan},
       title={Uniqueness of the minimizer of the normalized volume function},
        date={2021},
     journal={Camb. J. Math.},
      volume={9},
      number={1},
       pages={149\ndash 176},
}

\bib{XZ-SDC}{article}{
      author={Xu, Chenyang},
      author={Zhuang, Ziquan},
       title={Stable degenerations of singularities},
        date={2025},
     journal={J. Amer. Math. Soc.},
      volume={38},
      number={3},
       pages={585\ndash 626},
}

\bib{Z-equivariant}{article}{
      author={Zhuang, Ziquan},
       title={Optimal destabilizing centers and equivariant {K}-stability},
        date={2021},
     journal={Invent. Math.},
      volume={226},
      number={1},
       pages={195\ndash 223},
}

\end{biblist}
\end{bibdiv}

\end{document}